\documentclass[11pt,reqno]{amsart}

\usepackage[margin=1in]{geometry}
\usepackage[utf8]{inputenc}
\usepackage{amsmath,amsthm,amsfonts,amssymb}
\numberwithin{equation}{section}
\usepackage{times}
\usepackage{enumitem}
\setlist[enumerate]{label=(\roman*), ref=\roman*}
\usepackage{mathrsfs}
\usepackage{comment}
\usepackage{graphicx}

\usepackage{algpseudocode}
\usepackage{algorithm}
\usepackage{algorithmicx}

\usepackage{caption}
\usepackage{subcaption}
\usepackage{makecell}
\usepackage{xcolor}
\usepackage{booktabs,longtable}
\usepackage[nosort]{cite}
\usepackage[colorlinks=true, pdfstartview=FitV, linkcolor=blue,citecolor=blue, urlcolor=blue]{hyperref}
\usepackage[nameinlink,capitalise]{cleveref}
\hypersetup{pdftitle={Nonlinear stability of the Guderley imploding shock wave in radial symmetry}, pdfauthor={Giorgio Cialdea}}

\newtheorem{theorem}{Theorem}[section]
\newtheorem{lemma}[theorem]{Lemma}
\newtheorem{proposition}[theorem]{Proposition}
\newtheorem{corollary}[theorem]{Corollary}

\newtheorem{remark}[theorem]{Remark}

\newcommand{\dd}{\mathrm{d}}
\newcommand{\pp}{\partial}
\newcommand{\RR}{\mathbb{R}}
\newcommand{\one}{\mathbf{1}}
\newcommand{\Cr}{C_{r}}
\newcommand{\Cu}{C_{u}}
\newcommand{\Cb}{C_{b}}
\newcommand{\Rs}{R_{\rm s}}

\newcommand{\ubar}{\bar U}
\newcommand{\sigbar}{\bar\Sigma}
\newcommand{\Bbar}{\bar B}
\newcommand{\Vbar}{\bar V}
\newcommand{\Qbar}{\bar Q}
\newcommand{\Kbar}{\bar K}

\newcommand{\calC}{\mathcal C}

\newcommand{\rA}{\mathring A}
\newcommand{\tZ}{\widetilde{\mathring Z}}
\newcommand{\tA}{\widetilde{\mathring A}}
\newcommand{\tW}{\widetilde{\mathring W}}
\newcommand{\Y}{\mathcal Y}
\newcommand{\Ybar}{\bar{\mathcal Y}}
\newcommand{\Ytil}{\widetilde{\mathcal Y}}
\newcommand{\Tmat}{\mathcal T}
\newcommand{\Dmat}{\mathcal D}
\newcommand{\Ncal}{\mathcal N}
\newcommand{\diag}{\operatorname{diag}}
\newcommand{\Ocal}{\mathcal O}
\newcommand{\uu}{\mathbf{u}}
\newcommand{\s}{\mathsf{s}}
\renewcommand{\ss}{\s}
\newcommand{\ds}{\dot{\mathsf{s}}}
\newcommand{\rW}{\mathring{W}}
\newcommand{\rZ}{\mathring{Z}}
\newcommand{\usp}{u{\scriptstyle|_{\s}^+}}

\newcommand{\rsp}{\rho{\scriptstyle|_{\s}^+}}

\newcommand{\rsm}{\rho{\scriptstyle|_{\s}^-}}
\newcommand{\usm}{u{\scriptstyle|_{\s}^-}}
\newcommand{\ssp}{\sigma{\scriptstyle|_{\s}^+}}
\newcommand{\ssm}{\sigma{\scriptstyle|_{\s}^-}}
\def\jump#1{{[\hspace{-1.5pt}[#1]\hspace{-1.5pt}]}}
\newcommand{\crr}{\mathsf{c}_\mathsf{r}}
\newcommand{\crrb}{\overline{\mathsf{c}}_\mathsf{r}}
\newcommand{\tcrr}{\widetilde{\mathsf{c}}_\mathsf{r}}
\newcommand{\cbb}{\mathsf{c}_\mathsf{b}}
\newcommand{\cuu}{\mathsf{c}_\mathsf{u}}
\newcommand{\cbbb}{\overline{\mathsf{c}}_\mathsf{b}}

\newcommand{\tcbb}{\widetilde{\mathsf{c}}_\mathsf{b}}
\newcommand{\tcuu}{\widetilde{\mathsf{c}}_\mathsf{u}}
\newcommand{\LZ}{\bar\Lambda_1}
\newcommand{\LA}{\bar\Lambda_2}
\newcommand{\LW}{\bar\Lambda_3}

\title{Nonlinear stability of the Guderley imploding shock wave in radial symmetry}
\author{Giorgio Cialdea}
\address{Courant Institute, Department of Mathematics, New York University, New York, NY 10012.}
\email{giorgio.cialdea@cims.nyu.edu}
\date{\today}

\begin{document}
\begin{abstract}
We prove nonlinear stability of the Guderley imploding shock under radial perturbations for the three-dimensional compressible Euler equations in the adiabatic range $1<\gamma\leq\frac{5}{3}$. The perturbed solutions develop an asymptotically self-similar implosion in finite time. Our result allows perturbations with small, \emph{nonzero} constant pressure in the quiescent region ahead of the shock. The proof combines weighted $L^\infty$ estimates along characteristics with a computer-assisted proof of stability inequalities for the Guderley profile.
\end{abstract}
\maketitle
\setcounter{tocdepth}{1}

\section{Introduction} \label{introduction}

Implosions form a remarkable class of singularities in compressible fluid motion, in which at least one of the primary thermodynamic flow variables (density, pressure, temperature) becomes unbounded at a single point in spacetime. Among the most significant is the Guderley imploding shock~\cite{Gud42}, a self-similar solution to the Euler equations that describes a strong, radially symmetric shock wave converging to a point~\cite{Sed59,Sta60,ZR67,Bar96}.

A natural question to study is the stability of such solutions. Linear mode analysis suggests~\cite{Mor51,CZP95} that the Guderley imploding solution is linearly stable under radial perturbations, while unstable under non-radial perturbations~\cite{BK63,WR96}. In this paper, we prove that the Guderley imploding solution is \emph{nonlinearly} stable under radial perturbations.

The full compressible Euler equations in $\RR^d$ for the unknown density $\rho$, the momentum $\rho \uu$, and the total energy $E$ are given by the system of conservation laws
\begin{subequations} \label{eq:euler:cl}
 \begin{align}  
 \partial_t \rho +  \operatorname{div}(\rho   \uu) & = 0  \,, \\ 
 \partial_t (\rho \uu) + \operatorname{div}(\rho \uu \otimes \uu  +  p I) & = 0 \,,  \\ 
 \partial_t E  + \operatorname{div} ((p+E) \uu) &=0 \,,  
 \end{align}
 \end{subequations}  
 where the pressure $p$ is given by 
 \begin{equation*}  
 p := (\gamma-1) \bigl( E - \tfrac 12 \rho |\uu|^2 \bigr)
 \end{equation*} 
and $\gamma>1$ is the adiabatic exponent. We also set $\alpha:=\frac{\gamma-1}{2}$ and introduce the square root of the pseudo-entropy\footnote{Usually, the quantity $\frac{\gamma p}{\rho^\gamma}$ is denoted as $e^S$, where $S$ is the specific entropy.}
 \begin{equation*} 
b:= \sqrt{\tfrac{\gamma p}{\rho^\gamma}}
\end{equation*} 
and the rescaled sound speed
\begin{equation*} 
\sigma := \tfrac{1}{\alpha} \rho^\alpha b \,.
\end{equation*}

\subsection{Radial symmetry}
We will consider radial solutions to the compressible Euler equations~\eqref{eq:euler:cl}. By making a radial/spherical ansatz
\begin{equation*} 
 \sigma(x,t) =  \sigma(|x|,t), \quad b(x,t) =  b(|x|,t), \quad \uu(x,t) = \tfrac{x}{|x|} u^r(|x|,t),
\end{equation*}
where $|x|=r$, and using that $\operatorname{div}\uu= \pp_r u^r + \frac{d-1}{r} u^r$, we may write~\eqref{eq:euler:cl} as
\begin{subequations}\label{eq:euler:rad}
     \begin{align}
\partial_t \sigma + u^r \, \partial_r \sigma + \alpha \sigma \left( \partial_r u^r + \tfrac{d-1}{r}u^r \right) &= 0
\,, \\
\partial_t u^r + u^r \, \partial_r u^r + \alpha \sigma \partial_r \sigma - \tfrac{\alpha \sigma^2}{\gamma b} \partial_r b &= 0 \,, \\
\partial_t b + u^r \, \partial_r b &= 0 \,.
\end{align}\end{subequations}

\subsection{The Rankine--Hugoniot jump conditions in spherical symmetry}
We will analyze weak solutions of the Euler equations~\eqref{eq:euler:cl}. More precisely we will work with piecewise $C^1_{x,t}$ solutions, which are discontinuous across a \emph{shock surface}, denoted by $\mathcal{S} = \{(x,t): \s(t)-|x|=0\} \subset \RR^{d+1}$ (see for instance~\cite{Maj83exist}). Across such a surface, weak solutions of Euler must satisfy the Rankine--Hugoniot jump conditions. In radial symmetry, we write $u=u^r$ for the radial velocity and express the Rankine--Hugoniot jump conditions associated with~\eqref{eq:euler:cl} as the one-dimensional system
\begin{subequations} 
\label{intro:RH}
\begin{align}
\ds \jump{\rho u}&= \jump{\rho u^2 + p}, \\
\ds \jump{\rho}&= \jump{\rho u}, \\
\ds \jump{E}&= \jump{(p+E)u}.
\end{align}
\end{subequations}
Here $\jump{f}:=f|_{\s}^{+}-f|_{\s}^{-}$, where $+$ and $-$ denote the exterior and interior traces, respectively.

\subsubsection{The Lax entropy condition} Shock solutions of Euler, in order to have uniqueness, must also satisfy an entropy condition. We will ask that our solution satisfy the Lax entropy conditions for a $1$-shock, namely:
\begin{equation} \label{intro:lax}
(u-\alpha\sigma)|_\s^{+} < \ds< (u-\alpha\sigma)|_\s^{-}, \qquad
\ds< u|_\s^+ \le (u+\alpha\sigma)|_\s^+, \qquad
\ds< u|_\s^- \le (u+\alpha\sigma)|_\s^-.
\end{equation}
By assuming the Lax inequalities~\eqref{intro:lax}, the exterior state $(\usp, \rsp, \ssp)$ can be explicitly computed in terms of $(\usm, \rsm, \ssm)$ and $\ds$ by inverting~\eqref{intro:RH}
\begin{subequations} \label{intro:RH:re}
\begin{align}
\usp
&= \ds
+ \tfrac{\alpha\left((\usm-\ds)^2+\alpha(\ssm)^2\right)}
{(1+\alpha)(\usm-\ds)}, \\
\rsp
&= \tfrac{(1+\alpha)(\usm-\ds)^2\rsm}
{\alpha\left((\usm-\ds)^2+\alpha(\ssm)^2\right)}, \\
\ssp
&= \tfrac{\sqrt{\left(\gamma(\usm-\ds)^2-\alpha^3(\ssm)^2\right)
\left((\usm-\ds)^2+\alpha(\ssm)^2\right)}}
{\sqrt{\alpha}(1+\alpha)(\usm-\ds)}.
\end{align}
\end{subequations}

\subsection{The Guderley solution} The classical imploding Guderley scenario was first described by Guderley~\cite{Gud42}, and then rediscovered by Landau and Stanyukovich~\cite{LS45,Sta60}. A radial shock wave converges into a quiescent medium (constant density, zero velocity and zero pressure), leading to a collapse at the spatial origin. In particular, the Guderley imploding solution is an exact radially symmetric self-similar solution of~\eqref{eq:euler:cl}. That is,  there exists a self-similar exponent $0<\crrb<1$, which is classically determined through a nonlinear eigenvalue problem, such that for $t<0$ the solution has the form
\begin{equation*}
\begin{aligned}
\bar u(r,t)&=
\begin{cases}
0, & 0\le r<(-t)^{\crrb},\\
\makebox[10em][l]{$(-t)^{\crrb-1}\ubar\left(\tfrac{r}{(-t)^{\crrb}}\right),$} & r>(-t)^{\crrb},
\end{cases}
\\
\bar\sigma(r,t)&=
\begin{cases}
0, & 0\le r<(-t)^{\crrb},\\
\makebox[10em][l]{$(-t)^{\crrb-1}\sigbar\left(\tfrac{r}{(-t)^{\crrb}}\right),$} & r>(-t)^{\crrb},
\end{cases}
\\
\bar\rho(r,t)&=
\begin{cases}
1, & 0\le r<(-t)^{\crrb},\\
\makebox[10em][l]{$\overline{\mathsf R}\left(\tfrac{r}{(-t)^{\crrb}}\right),$} & r>(-t)^{\crrb},
\end{cases}
\end{aligned}
\end{equation*}
for some profiles $(\bar U, \bar \Sigma, \bar{\mathsf{R}})$ (see Section~\ref{sec:guderley} for more details on the construction).
\subsubsection{Previous work on the Guderley imploding shock wave}
For much of its history, the Guderley problem has been investigated through semi-analytic and numerical methods~\cite{LazR77,Laz81,MTS82,RKB12,GBK23,Sed59}. This classical description included both the converging shock and its reflected continuation after collapse~\cite{Gud42,LS45,Sta60,Laz81}. The self-similar exponent was numerically selected by a shooting method from the sonic point $R_s$.  A rigorous construction, which overcame these longstanding challenges, was obtained by Jang, Liu, and Schrecker~\cite{JLS25b} for $1<\gamma\leq3$; their subsequent work established also the reflected continuation~\cite{JLS25a}.  Jenssen and Tsikkou~\cite{JK18} rigorously showed that the Guderley imploding solution, and the subsequent reflection, provide a genuine weak solution for the compressible Euler equations globally in time. In~\cite{CSV25}, Cialdea, Shkoller and Vicol constructed radial initial data of class $C^{1,\frac{1}{3}}$ whose evolution produces a shock that strengthens and matches the Guderley solution before the singularity time.

\subsubsection{Previous numerical work on the radial stability of Guderley}
The question of stability of the Guderley solution has also been the subject of many numerical studies. Morawetz's doctoral thesis~\cite{Mor51} studied the linear operator arising from linearizing the Euler equations around the Guderley solution, and provided numerical evidence of stability. More precisely, Morawetz studied a spherical piston enclosing a gas initially at rest, with constant density and zero pressure. In this scenario, by prescribing the motion of the piston to follow exactly a fluid trajectory of Guderley, we can recreate the imploding Guderley shock wave. In the thesis, she studied how modifying the trajectory of the piston affects the stability of the Guderley solution. She then wrote the corresponding linear problem and reduced it to a singular (through the sonic point) second-order ordinary differential equation depending on the complex mode parameter $k$. After identifying the modes corresponding to scaling and time translation, she performed numerical computation suggesting that the remaining modes decay. Her work left open a rigorous proof of the linear stability under radial perturbations, nonlinear radial stability, and the extension to nonzero pressure in the interior.

Chen, Zhang, and Panarella~\cite{CZP95} later obtained further evidence for radial stability from perturbative calculations. 
\subsubsection{Previous numerical work on the instability of Guderley outside symmetry}
Outside radial symmetry, the numerical studies of the linear stability of Guderley actually show a different picture. Indeed, Brushlinskii and Kazhdan~\cite{BK63} identified unstable modes, and later on Wu and Roberts~\cite{WR96} refined their analysis. Related works by Gardner, Book, and Bernstein~\cite{GBB82} in the Chester--Chisnell--Whitham approximation, and by Evans~\cite{Eva96} in an approximate shock-propagation model, found also evidence of instability outside radial symmetry.

These works paint a consistent picture: the Guderley imploding shock is stable in radial symmetry, but unstable outside symmetry. In this paper, we rigorously justify the first half of this picture, by providing a proof of nonlinear stability of the Guderley imploding shock wave.

\subsection{Main result}
In this paper we study the problem of nonlinear stability of the Guderley imploding shock. In our main result we are able to prescribe initial data with \emph{nonzero} pressure (albeit constant) in the interior. 

We state our main result. For simplicity, we will consider the case $d=3$ and $1 < \gamma \le \frac{5}{3}$; we expect our techniques to apply with minor modifications to the remaining parameter regimes. 

\begin{theorem}[Radial stability, abbreviated version]
\label{thm:intro-main-physical}
Fix $d=3$, $1<\gamma\leq\frac{5}{3}$, and the associated Guderley profile $(\bar U,\bar\Sigma,\bar B)$, with similarity exponent $\crrb$. Here $\bar B$ denotes the self-similar profile of $b$, as defined in~\eqref{bw:ss}. Consider piecewise-$C^1$ radial initial data $(\rho,u^r,p)(\cdot,-1)$ for~\eqref{eq:euler:cl} with a single shock at $r=1$, that satisfy the compatibility conditions as specified in~\eqref{eq:initial-finite-pressure-shock-traces} and~\eqref{eq:initial-differentiated-compatibility} \footnote{The compatibility conditions guarantee that the initial data generate a piecewise $C^1_{x,t}$ shock solution to the Euler equations, satisfying the Rankine--Hugoniot conditions~\eqref{intro:RH} and the Lax entropy inequalities~\eqref{intro:lax}.}. Assume that, in the interior region
\begin{equation*}
(\rho,u^r,p)(r, -1)=(1,0,p_0) \quad \text{for }0\leq r<1,
\end{equation*}
 where $p_0 \ge 0$ is small enough.  For $r>1$, denote 
\begin{equation*}
(\tilde\rho,\tilde u^r,\tilde p)(r,-1)=(\rho,u^r,p)(r,-1)-(\bar\rho,\bar u,\bar p)(r,-1)
\end{equation*}
and assume $(\tilde\rho,\tilde u^r,\tilde p)(\cdot,-1)$ are small enough in an appropriate weighted $C^1$ norm (see Corollary~\ref{thm:main-physical}). Then there exist a finite time $t_\ast$ close to $0$ and a unique radial classical shock, asymptotically self-similar implosion solution of~\eqref{eq:euler:cl} on $-1\leq t<t_\ast$, with the prescribed initial data. The shock converges to the origin at time $t_\ast$, with the Guderley rates\footnote{We say $f(t) \simeq g(t)$ as $t \uparrow t_*$ if  $\frac{f(t)}{g(t)}$ stays uniformly away from zero and bounded as $t \uparrow t_*$. }
\begin{equation*}\s(t)\simeq(t_\ast-t)^{\crrb},\qquad -\ds(t)\simeq(t_\ast-t)^{\crrb-1}\quad\text{as }t\uparrow t_\ast.
\end{equation*}
In particular, velocity and sound speed blow up with rates
\begin{equation*}
\| u \|_\infty(t) \simeq (t_\ast-t)^{\crrb-1}, \quad \| \sigma \|_\infty(t) \simeq (t_\ast-t)^{\crrb-1}, 
\end{equation*}
while the density remains bounded.

In the interior region ahead of the shock (that is, for $r<\ss(t)$), for all $-1 \le t < t_\ast$, the density and pressure remain constant while velocity vanishes identically.

In modulated self-similar coordinates, for $R>1$ the exterior fields $(\Sigma,U,B)$ exist for all $\tau\geq0$ and converge exponentially to $(\bar\Sigma,\bar U,\bar B)$ in the weighted $C^1$ norms specified in~\eqref{eq:main-self-similar-perturbation-decay}. 
\end{theorem}
\begin{remark}[What is an implosion?] Informally, an implosion is a flow in which geometric focusing causes a gas to converge inward toward a point, compressing matter to infinite density or pressure in finite time. We make this notion mathematically precise, using Definition 1.3 in~\cite{CSV26}:

A solution of the compressible Euler equations on $\RR^d\times[-1,t_*)$ is called an \emph{implosion solution} with \emph{implosion center} $x_*\in\RR^d$ and \emph{implosion time} $t_*>-1$, if the following hold:
\begin{enumerate}[leftmargin=2em]
\item \emph{(Blowup of a primary flow variable).} As one approaches the spacetime point of collapse, $(x_*,t_*)$, at least one of the primary thermodynamic variables (density or pressure) becomes unbounded; that is, for every $\varepsilon\in(0,1]$ we have
\begin{equation*}
\limsup_{t\to t_*^-}\;\sup_{|x-x_*|<\varepsilon}
\tfrac{\rho(x,t)}{\|\rho(\cdot,-1)\|_{L^\infty(B_{100}(0))}}
+\limsup_{t\to t_*^-}\;\sup_{|x-x_*|<\varepsilon}
\tfrac{p(x,t)}{\|p(\cdot,-1)\|_{L^\infty(B_{100}(0))}}
=+\infty.
\end{equation*}
\item \emph{(Inward-focusing character).} The flow is compressive toward $x_*$; that is, there exists a time-dependent radius $\mathsf{r}\colon[-1,t_*)\to(0,1]$ such that
\begin{equation*}
\bigl(\uu(x,t)-\uu(x_*,t)\bigr)\cdot(x-x_*)\leq0,
\end{equation*}
for all $|x-x_*|\leq\mathsf{r}(t)$, and all $t\in[-1,t_*)$.
\item \emph{(No amplitude blowup away from the spacetime point of collapse).} The primary flow variables $(\rho,\uu,p)$ are bounded at all spacetime points which are away from $(x_*,t_*)$; that is, for any $\varepsilon\in(0,1]$ there exists $C_\varepsilon>0$ such that
\begin{equation*}
\sup_{\substack{(x,t)\in\RR^d\times[-1,t_*)\\
\varepsilon\leq|x-x_*|+|t-t_*|\leq1/\varepsilon}}
\tfrac{\rho(x,t)}{\|\rho(\cdot,-1)\|_{L^\infty(B_{100}(0))}}
+\tfrac{|\uu(x,t)|}{\|\uu(\cdot,-1)\|_{L^\infty(B_{100}(0))}}
+\tfrac{p(x,t)}{\|p(\cdot,-1)\|_{L^\infty(B_{100}(0))}}
\leq C_\varepsilon.
\end{equation*}
\end{enumerate}
\end{remark}

\begin{remark}[No other singularity before the implosion time]
We observe that in particular the solution $(\rho, u^r, p)$ does not form any other singularity before the implosion time $t_\ast$. Indeed, the $C^1$ estimates~\eqref{eq:main-self-similar-VQ-decay} imply that in the exterior region $r>\ss(t)$ the gradient is bounded for every $t<t_\ast$; in particular no other shock can form. At the same time, the interior region remains constant, and thus free of singularities.
\end{remark}

\begin{remark}[The modulation $\crr$ coefficient]
In the modulated self-similar analysis, we introduce the modulation functions $\cbb$ and $\cuu$. These modulation functions, as the self-similar time $\tau \to + \infty$, converge to the Guderley exponent $\crrb-1$. On the other hand, the modulation function $\crr$ is set to be identically equal to the Guderley exponent $\crrb$ (this is different from similar works using modulated self-similar analysis, see for instance~\cite[Remark~1.5]{CSV26}).
\end{remark}

\begin{remark}[Non-constant interior state]
A strategy similar to the one employed in this paper can be used to prove stability under non-constant small perturbations in the interior state. The key observation is  that one can obtain a priori estimates on the maximal development of the solution in the interior region independently of the outer region. For the outer region, then one obtains similar estimates from the Rankine--Hugoniot jump conditions to the ones we obtained here in Sections~\ref{subsec:rh-modulated} and~\ref{subsec:bootstrap-modulation}. For brevity, we do not pursue such a generalization in this paper.
\end{remark}

\subsection{Related literature}
We now review other examples of known implosions for the Euler equations (and related models) and their known stability properties.
\subsubsection{Isentropic smooth implosions}
The first example of smooth imploding solutions for the compressible Euler equations was constructed by Merle, Rapha{\"e}l, Rodnianski, and Szeftel~\cite{MRRS22a}; for almost every adiabatic exponent $\gamma>1$ they constructed an infinite sequence of self-similar imploding profiles. Buckmaster, Cao-Labora, and G\'omez-Serrano~\cite{BCG25} extended their work by constructing a profile for every $\gamma>1$, while, later on, Shao, Wei, Wang, and Zhang~\cite{SWWZ25} constructed an infinite sequence of implosion profiles for the three-dimensional monatomic gas $\gamma=\frac{5}{3}$. For these profiles, in radial symmetry, stability of finite codimension in a weighted Sobolev space was proven by Merle, Rapha{\"e}l, Rodnianski, and Szeftel~\cite{MRRS22b}. An analogous statement, but outside radial symmetry, was obtained by Cao-Labora, G\'omez-Serrano, Shi, and Staffilani~\cite{CGSS23}. Perturbing these profiles in a class of two-dimensional axisymmetric perturbations, Chen, Cialdea, Shkoller, and Vicol~\cite{CCSV24} constructed smooth solutions of compressible Euler exhibiting vorticity blowup. This work was subsequently extended by Chen~\cite{Chen25} for $d\ge3$. We observe that none of these works (both in and outside symmetry) provide any quantitative bound (lower or upper) on the dimension of the unstable manifold, in particular leaving open the question of full stability. Indeed, the numerical work of Biasi~\cite{Bia21} indicates that the profiles from~\cite{MRRS22a} are genuinely not stable in radial symmetry, and the nonlinear Euler evolution of some initial perturbation forms a shock before the implosion time.

We also note that these profiles and their associated stability properties have been used to establish blowup for other related models. The works of Merle, Rapha{\"e}l, Rodnianski, and Szeftel~\cite{MRRS22b}, Buckmaster, Cao-Labora, and G\'omez-Serrano~\cite{BCG25}, Cao-Labora, G\'omez-Serrano, Shi, and Staffilani~\cite{CGSS23}, and Shao, Wei, Wang, and Zhang~\cite{SWWZ25} established implosions for the barotropic compressible Navier-Stokes equations with linear viscosity. Merle, Rapha{\"e}l, Rodnianski, and Szeftel~\cite{MRRS22c} used these profiles to also prove finite time blowup for the energy-supercritical defocusing NLS equation. In another direction, Bedrossian, Chen, Gualdani, Ji, Vicol, and Yang~\cite{BCGJVY26} proved blowup for the Landau equation with very hard potentials.

\subsubsection{Non-isentropic smooth implosions} A different class of smooth implosions was discovered by Chen, Shkoller, and Vicol~\cite{CSV26}. As opposed to the profiles constructed by Merle, Rapha{\"e}l, Rodnianski, and Szeftel~\cite{MRRS22a}, these profiles are genuinely non-isentropic. In the same work, the authors were able to prove that the ground state profile from~\cite{CSV26} is stable to radially symmetric perturbations for the full nonlinear compressible Euler equations, modulo the one-dimensional compatibility condition of having vanishing pressure at the physical origin. For the same ground state, for the case of monatomic and diatomic gases, they were also able to provide, outside symmetry, a full characterization of the set of initial data that yield nonlinear stability.

\subsubsection{Other implosions} Other classes of implosions for the compressible Euler equations have been constructed by other authors. We notice here the collapsing-cavity example treated by Hunter~\cite{Hun60} and then Lazarus~\cite{Laz81}, the continuous isentropic implosions constructed by Jenssen~\cite{Jen25} and Shao, Wang, Wei, and Zhang~\cite{SWWZ26} and the non-isentropic shock implosions by Jenssen and Tsikkou~\cite{JK18}. We observe that no stability analysis is available for any of these results.

\subsubsection{Imploding singularities in related models}

Smooth self-similar imploding singularities have also been constructed for systems closely related to compressible Euler equations.

For the isothermal Euler--Poisson system, which governs the gravitational collapse of a self-gravitating gas, Guo, Had\v{z}i\'c, and Jang~\cite{GHJ21} rigorously constructed the \emph{Larson--Penston implosion profile}, a radially symmetric, $C^\infty$-smooth, globally self-similar imploding solution first postulated on astrophysical grounds by Larson~\cite{Lar69} and Penston~\cite{Pen69}. In contrast to the smooth isentropic implosions of~\cite{MRRS22a,BCG25,SWWZ25}, the Larson--Penston profile was expected to be \emph{stable} under radial perturbations (see, e.g.,~\cite{HN97,OP88}, and the discussion in~\cite{GHJS25}). This expectation was recently turned into a theorem by Guo, Had\v{z}i\'c, Jang, and Schrecker~\cite{GHJS25}, who established nonlinear dynamical stability of the Larson--Penston profile under radially symmetric perturbations.

For the relativistic Euler equations, Shao, Wei, and Zhang~\cite{SWZ24} constructed smooth imploding self-similar solutions and used them to prove finite-time singularity formation for the supercritical defocusing nonlinear wave equation with complex-valued solutions~\cite{SWZ25}. In a parallel direction, Buckmaster and Chen~\cite{BC24} constructed relativistic imploding profiles and used them to establish blowup in dimension $d=4$ for the nonlinearity $|u|^{p-1}u$ with $p=7$.

\section{Strategy of the proof}
We now discuss the strategy of the proof. We start by introducing the modulated self-similar variables that we will use in the proof of Theorem~\ref{thm:intro-main-physical}.

\subsection{Modulated self-similar variables in radial symmetry}

Let $\tau\in[0,\infty)$ denote the self-similar time. Introduce the modulation functions
\begin{equation*}
\crr,\cuu,\cbb\colon[0,\infty)\longrightarrow\RR,
\end{equation*}
and define the associated scaling factors
\begin{subequations}
\label{eq:scaling-factors}
\begin{align}
\Cr(\tau)&:=\Cr(0)\exp\left(-\int_0^\tau\crr(s)\,\dd s\right),
\label{eq:scaling-factor-r}\\
\Cu(\tau)&:=\Cu(0)\exp\left(-\int_0^\tau\cuu(s)\,\dd s\right),
\notag\\
\Cb(\tau)&:=\Cb(0)\exp\left(-\int_0^\tau\cbb(s)\,\dd s\right).
\notag
\end{align}
\end{subequations}
The physical time and radial coordinate are
\begin{subequations}
\label{eq:modulated-coordinates}
\begin{align}
t=t(\tau)&:=-1+\int_0^\tau\Cr(s)\Cu(s)^{-1}\,\dd s,
\label{eq:modulated-time}\\
R&:=r\Cr(\tau)^{-1}.
\label{eq:modulated-radius}
\end{align}
\end{subequations}
We then define the modulated self-similar variables $(U,\Sigma,B)$ by
\begin{subequations}
\label{eq:modulated-ansatz}
\begin{align}
u^r(r,t)&=\Cu(\tau)U(R,\tau),\label{eq:modulated-ansatz-u}\\
\sigma(r,t)&=\Cu(\tau)\Sigma(R,\tau),\label{eq:modulated-ansatz-sigma}\\
b(r,t)&=\Cb(\tau)B(R,\tau).\label{eq:modulated-ansatz-b}
\end{align}
\end{subequations}
Substitution of~\eqref{eq:modulated-ansatz} into \eqref{eq:euler:rad} gives
\begin{subequations}
\label{eq:modulated-euler}
\begin{align}
\partial_\tau\Sigma-\cuu\Sigma+\crr R\partial_R\Sigma
+U\partial_R\Sigma
+\alpha\Sigma\left(\partial_RU+\tfrac{d-1}{R}U\right)&=0,
\label{eq:modulated-euler-sigma}\\
\partial_\tau U-\cuu U+\crr R\partial_RU
+U\partial_RU+\alpha\Sigma\partial_R\Sigma
-\tfrac{\alpha}{\gamma}\Sigma^2\tfrac{\partial_RB}{B}&=0,
\label{eq:modulated-euler-u}\\
\partial_\tau B-\cbb B+\crr R\partial_RB+U\partial_RB&=0.
\label{eq:modulated-euler-b}
\end{align}
\end{subequations}
Given two similarity exponents $\crrb>0$ and $\cbbb \in \RR$, we say a solution $(\bar u, \bar \sigma, \bar b)$ is a globally self-similar implosion of~\eqref{eq:euler:rad} for $t<0$ if it is of the form\footnote{In the notation of Lazarus~\cite{Laz81}, and later used by~\cite{JK18,JT23,Jen25,JLS25a}, the exponent $\lambda$ corresponds to $\frac{1}{\crrb}$, while $\kappa$ corresponds to $\frac{\crrb-1-\cbbb}{\alpha\crrb}$.}
\begin{subequations}\label{bw:ss}
\begin{align} 
\bar u(r,t) &:= (-t)^{\crrb -1} \bar U\left( \tfrac{r}{(-t)^{\crrb}} \right), \\
 \bar \sigma(r,t) &:= (-t)^{\crrb-1} \bar \Sigma\left( \tfrac{r}{(-t)^{\crrb}} \right), \\ 
  \bar b(r,t) &:= (-t)^{\cbbb} \bar B\left( \tfrac{r}{(-t)^{\crrb}} \right).
\end{align}
\end{subequations}
This is equivalent to asking that $(\bar U, \bar \Sigma, \bar B)$ is a solution to~\eqref{eq:modulated-euler} independent of $\tau$ and with constant modulation functions
\begin{equation*}
(\crr,\cuu,\cbb)=(\crrb,\crrb-1,\cbbb).
\end{equation*}
Moreover, the system~\eqref{eq:modulated-euler} reduces to a system of ODEs for $(\bar U, \bar \Sigma, \bar B)$ in the variable $R$
\begin{subequations} \label{bw:ss:euler}
\begin{align}
 (1-\crrb) \bar \Sigma + \crrb R \pp_R \bar \Sigma + \bar U \pp_R \bar \Sigma + \alpha \bar \Sigma ( \pp_R \bar U + \tfrac{d-1}{R} \bar U) &=0,  \\
(1- \crrb) \bar U + \crrb R \pp_R \bar U + \bar U \pp_R \bar U + \alpha \bar \Sigma \pp_R \bar \Sigma - \tfrac{\alpha}{\gamma} \bar \Sigma^2 \tfrac{\pp_R \bar B}{\bar B}&=0, \\
- \cbbb \bar B + \crrb R \pp_R \bar B + \bar U \pp_R \bar B &=0. \label{bw:ss:euler:B}
\end{align}\end{subequations}
With $C_r(0)=C_u(0)=C_b(0)=1$, one has
\begin{equation*}
\begin{aligned}
\Cr(\tau)&=e^{-\crrb\tau},
&\Cu(\tau)&=e^{(1-\crrb)\tau},\\
\Cb(\tau)&=e^{-\cbbb\tau},
&t(\tau)&=-e^{-\tau},\\
R&=\tfrac{r}{(-t)^{\crrb}}.
\end{aligned}
\end{equation*}
In particular, the Guderley imploding shock wave is an exact globally self-similar solution with self-similar exponent $\crrb$ and $\cbbb=\crrb-1$.

\subsection{Strategy of the proof}
\subsubsection{Initial data and evolution in the interior region}
We are interested in studying the evolution of the Euler equations, with initial data at time $t=-1$, in physical variables given by
\begin{equation*}
\begin{aligned}
u(r,-1)&=0,\quad \rho(r,-1)=1,\quad p(r,-1)=p_0, &&r<1,
\\
u(r,-1)&=\bar u(r,-1)+\tilde u(r,-1), &&r>1,
\\
\rho(r,-1)&=\bar\rho(r,-1)+\tilde\rho(r,-1), &&r>1,
\\
\sigma(r,-1)&=\bar\sigma(r,-1)+\tilde\sigma(r,-1), &&r>1,
\end{aligned}
\end{equation*}
where the perturbation $(\tilde u, \tilde\rho, \tilde\sigma)$ is small enough and satisfies the shock compatibility conditions~\eqref{eq:initial-finite-pressure-shock-traces} and~\eqref{eq:initial-differentiated-compatibility}. By classical local well-posedness theory for shock fronts (see Majda~\cite{Maj83exist,Maj83stab} for the theory in Sobolev spaces and Li~\cite{Li80} and Dafermos~\cite{Daf10} for the one-dimensional theory in $L^\infty$-based spaces), such initial data admit a unique local shock solution with shock front $\ss$. In particular, by uniqueness, in the interior region $r<\ss(t)$ the solution must remain in the original quiescent state $u=0, \rho=1, p=p_0$ for all later times. As a consequence, to study the stability problem, we focus on the evolution in the exterior region $r>\ss(t)$, coupled with the Rankine--Hugoniot jump conditions~\eqref{intro:RH} at the shock front $\{r=\ss(t)\}$.

Then, the stability analysis of the profiles $(\ubar,\sigbar,\Bbar)$ in radial symmetry is carried out using modulated self-similar coordinates, using the transformation~\eqref{eq:modulated-coordinates}, \eqref{eq:modulated-ansatz} and the modulation functions $(\crr,\cuu,\cbb)$.
\subsubsection{Choice of the modulation functions}
We write $\widetilde U=U-\bar U$, $\widetilde\Sigma=\Sigma-\bar\Sigma$, and $\widetilde B=B-\bar B$ for the perturbations of the Guderley profiles.

The choice of the modulation functions $(\crr, \cuu, \cbb)$ is fundamental in order to close the stability argument. These functions are chosen so that:
\begin{itemize}
    \item the shock location $\ss(t)$ in the self-similar coordinates is constant $R=1$; this is achieved by imposing the condition~\eqref{eq:shock-freezing-gauge},
    \item the traces of the undifferentiated self-similar variables at the shock $(\tilde U, \tilde \Sigma, \tilde B)(1^+, \tau)$ decay as $\tau \to \infty$; this is achieved by enforcing $\crr=\crrb$ and $\tcuu = \tcbb$.
\end{itemize}
\subsubsection{The system in the outer region $R \ge 1$.}
As we explained earlier, in the stability analysis, we will focus on the exterior region (that is, $r>\ss(t)$, or, in modulated self-similar coordinates, $R>1$). Since $R$ is away from $0$, we define the renormalized variables $V$, $Q$, and $K$ as in~\eqref{eq:VQK-def}:
\begin{equation*}
V:=\tfrac{U}{R},\qquad
Q:=\tfrac{\Sigma}{R},\qquad
K:=\log\left(\tfrac{B}{R}\right).
\end{equation*}
The goal is to prove that, given initial data $(V,Q,K)|_{\tau=0}$ for $R>1$, close enough to the Guderley profiles $(\Vbar,\Qbar,\Kbar)$ and satisfying the compatibility conditions~\eqref{eq:initial-finite-pressure-shock-traces} and~\eqref{eq:initial-differentiated-compatibility}, the solution $(V,Q,K)\to(\Vbar,\Qbar,\Kbar)$ for $R>1$ and $(\crr,\cuu,\cbb)\to(\crrb,\crrb-1,\crrb-1)$. By writing the solution as
\begin{equation*}
\begin{aligned}
V&=\Vbar+\widetilde V,
&Q&=\Qbar+\widetilde Q,
&K&=\Kbar+\widetilde K,
\\
\crr&=\crrb,
&\cuu&=\crrb-1+\tcuu,
&\cbb&=\crrb-1+\tcbb,
\end{aligned}
\end{equation*}
we obtain a system for $(\widetilde V,\widetilde Q,\widetilde K)$ from~\eqref{eq:VQK-system}, where the modulation functions $\tcuu$ and $\tcbb$ enter through the source terms. The modulation functions are then chosen, as we mentioned, to anchor the shock at the location $R=1$, and to simplify the Rankine--Hugoniot jump conditions at the shock location. Schematically, after the choice of the modulation functions $(\crr,\cuu,\cbb)$ in~\eqref{eq:fixed-radial-equal-amplitude-consequences}, the system for $(\widetilde V,\widetilde Q,\widetilde K)$ is coupled with boundary conditions~\eqref{eq:finite-pressure-RH-tilded-traces} at $R=1$ of the form
\begin{equation*}
(\widetilde V,\widetilde Q,\widetilde K)(1^+,\tau)
=\makebox{exponentially decaying error term}.
\end{equation*}
We observe that if we were to choose $p_0=0$ (that is, we are in the cold gas case as in the original Guderley profile), the error on the right-hand side would vanish exactly. Also, the error on the right-hand side decays because, crucially, we have the inequality $\crrb<1$, which guarantees that $C_u(\tau)$ grows exponentially~\eqref{eq:nonlinear-bootstrap-Cu-growth}.
To study the stability properties for $R>1$, we observe that it is convenient to work with the differentiated Riemann-type variables of Chen, Shkoller, and Vicol (Section~3.6 in~\cite{CSV26}), defined in~\eqref{eq:ring-def} by 
\begin{equation*}
\begin{aligned}
\rZ&:=R\partial_RV-R\partial_RQ
+\tfrac{1}{\gamma}QR\partial_RK,
\\
\rA&:=\tfrac{\alpha}{\gamma}QR\partial_RK,
\\
\rW&:=R\partial_RV+R\partial_RQ
-\tfrac{1}{\gamma}QR\partial_RK.
\end{aligned}
\end{equation*}
We collect the differentiated variables into the vector $\Y:=(\rZ,\rA,\rW)^\top$. They satisfy the system~\eqref{eq:ring-evolution},
\begin{equation*}
\partial_\tau\Y+\Tmat R\partial_R\Y+\Dmat\Y+\Ncal(\Y,\Y)=0,
\end{equation*}
where the diagonal transport matrix $\Tmat$, the damping matrix $\Dmat$, and the bilinear form $\Ncal$ are defined in~\eqref{eq:ring-transport-matrix}, \eqref{eq:ring-damping-matrix}, and~\eqref{eq:ring-quadratic-form}, respectively.

We denote the corresponding Guderley vector by $\Ybar$ and introduce the perturbation $\Ytil:=\Y-\Ybar=(\tZ,\tA,\tW)^\top$. Subtracting the stationary equation for $\Ybar$ gives~\eqref{eq:ring-perturbation-system},
\begin{equation}
\partial_\tau\Ytil+\Tmat R\partial_R\Ytil+\Dmat\Ytil=-\widetilde{\Tmat} R\partial_R\Ybar-\widetilde{\Dmat}\Ybar-\Ncal(\Ybar,\Ytil)-\Ncal(\Ytil,\Ybar)-\Ncal(\Ytil,\Ytil).
\label{eq:intro-strategy-Y-system}
\end{equation}
Here $\widetilde{\Tmat}$ and $\widetilde{\Dmat}$ are the matrix perturbations defined in~\eqref{eq:ring-perturbation-matrices}. The transport matrix is diagonal, with
\begin{equation*}
\Tmat=\diag(\Lambda_1,\Lambda_2,\Lambda_3)=\diag(\crr+V-\alpha Q,\crr+V,\crr+V+\alpha Q).
\end{equation*}
This system is coupled with the boundary conditions~\eqref{eq:finite-pressure-RH-solved-feedback}, which can be schematically represented as:
\begin{subequations}\label{eq:intro-strategy-feedback}
\begin{align}
\tW(1^+,\tau)
&=b_W(\gamma)\tZ(1^+,\tau)
+\makebox{exponentially decaying error term},
\label{eq:intro-strategy-feedback-W}\\
\tA(1^+,\tau)
&=b_A(\gamma)\tZ(1^+,\tau)
+\makebox{exponentially decaying error term},
\label{eq:intro-strategy-feedback-A}\\
\tcuu(\tau)
&=\kappa_\gamma\tZ(1^+,\tau)
+\makebox{exponentially decaying error term},
\label{eq:intro-strategy-feedback-cu}
\end{align}
\end{subequations}
with the coefficients $b_W(\gamma)$, $b_A(\gamma)$, and $\kappa_\gamma$ given by~\eqref{eq:finite-pressure-RH-feedback-coefficients}.
If we ignore for the moment the nonlinear terms, we are left with studying the linear diagonal transport problem, coupled with the boundary conditions~\eqref{eq:intro-strategy-feedback} at $R=1$,
\begin{equation}
\partial_\tau\Ytil+\diag(\LZ,\LA,\LW)R\partial_R\Ytil=\mathcal L_{\mathrm{bg}}\Ytil+\tcuu\Ybar+\widetilde V\mathcal F_V+\alpha\widetilde Q\mathcal F_Q.
\label{eq:intro-strategy-linear-system}
\end{equation}
Here $\mathcal L_{\mathrm{bg}}$ and the forcing vectors $\mathcal F_V,\mathcal F_Q$ are obtained by linearizing the operator in~\eqref{eq:intro-strategy-Y-system}. They are defined in~\eqref{eq:profile-background-linearization} and~\eqref{eq:profile-forcing-vectors}. The starting point of the analysis is the collection of properties of the characteristic families $\bar \Lambda_j$:
\begin{itemize}
\item the characteristic speeds $\bar \Lambda_2$ and $\bar \Lambda_3$ have the outgoing property
\begin{equation} \label{eq:out:strat:aw}
\inf_{R\geq1}\LA(R)\geq\tfrac{\alpha\crrb}{1+\alpha}>0,\qquad
\inf_{R\geq1}\LW(R)\geq\tfrac{\alpha\crrb}{1+\alpha}>0,
\end{equation}
\item the characteristic speed $\bar \Lambda_1$ instead vanishes at a unique point $R_s>1$ (the sonic point); it's negative in the region next to the shock $1<R<R_s$ and positive for $R>R_s$. In particular, for any $R_2>R_s$ we have the outgoing property
\begin{equation*} 
\inf_{R\geq R_2}\LZ(R)>0.
\end{equation*}
\end{itemize}
The stability argument is then roughly split in two parts: the region between the shock and the sonic line $R_s>1$; and the region behind the sonic line.
\subsubsection{The stability in the region between the shock and the sonic point}
In the region between the shock and the sonic point, for $R<\Rs$, we identify three profile inequalities, one for each variable $\rZ$, $\rA$, and $\rW$, given in~\eqref{eq:three-profile-inequalities}, that guarantee decay of the perturbations inside this region (observe that for the linearized problem, by finite speed of propagation, and the sign of the characteristic speeds, this interior region is not influenced by $R>\Rs$). These three profile inequalities are proved using a computer-assisted argument; see Appendix~\ref{app:three-profile-interval-proof} for the details.
The $\rZ$ inequality~\eqref{eq:profile-inequality-Z} guarantees that the term $\ell_{ZZ}\tZ$ generates enough damping to dominate the remainder terms in the equation~\eqref{eq:intro-strategy-linear-system}. Here $\ell_{ZZ}$ denotes the $Z$-diagonal entry of $\mathcal L_{\mathrm{bg}}$. Decay for this variable then follows simply by integrating along the associated $\LZ$ characteristic. The $\rW$ and $\rA$ inequalities~\eqref{eq:profile-inequality-W} and~\eqref{eq:profile-inequality-A} instead guarantee that the perturbations $\tW$ and $\tA$, when integrated along the respective characteristics from the shock $R=1$, do not grow excessively before reaching the sonic location $\Rs$. In particular, for these two last variables, the outgoing properties~\eqref{eq:out:strat:aw} are crucial, since they guarantee that the associated characteristic families remain for a finite time inside the region $1<R<\Rs$. By continuity, the same arguments apply \emph{slightly behind} the sonic point $R=\Rs$, in a strip $1<R<R_2$ for some $R_2>\Rs$ close to $\Rs$.

\subsubsection{The role of the computer-assisted argument} As we mentioned, in order to prove stability in the region between the shock and the sonic point, we will use three inequalities, involving the profile and its derivatives, which we prove using interval arithmetic. In Appendix~\ref{app:three-profile-interval-proof}, we construct explicit upper and lower barriers (involving only polynomials and elementary functions like $\log$) for the Guderley profiles. Using the phase portrait equations~\eqref{eq:guderley:VQ-autonomous}, then each inequality~\eqref{eq:three-profile-inequalities} is reduced to checking that a complicated (but explicit) expression in $(\Vbar,\Qbar)$ does not vanish on the region of the phase portrait specified by the upper and lower barriers.  Each of these steps consists in checking the sign of an explicit expression, for which we use interval arithmetic. We refer the reader to Appendix~\ref{app:three-profile-interval-proof} for the mathematical construction and the derivations by hand, while we refer the reader to Appendix~\ref{app:three-profile-computation} for details about the computer-assisted implementation.

\subsubsection{The stability in the region behind the sonic point}
 In the region $R>R_2$ we instead take full advantage of the outgoing properties~\eqref{eq:nonlinear-bootstrap-nonlinear-outgoing-speeds}. We first introduce coordinates $R_1<R_2<R_3$. On the compact outgoing region $[R_2,R_3]$, we use the weight $\psi$ and the outgoing properties to generate damping $-\Tmat R\partial_R\log\psi\,(\psi\Ytil)$ in~\eqref{eq:nonlinear-bootstrap-weighted-perturbation-system}. For $R\geq R_3$, we take advantage of the damping provided by the matrix $\mathcal{D}$, after estimating the remaining terms using the decay properties of the profile~\eqref{eq:nonlinear-bootstrap-far-field-coefficients}. The construction of the weight $\psi$ and the estimates used here were inspired by the radial stability argument by Chen, Shkoller and Vicol in~\cite{CSV26} (see Section~3 there).

At the technical level, the proof is carried out with a standard continuity bootstrap argument, where the bootstrap assumptions involve the exponential decay in $\tau$ of the differentiated variables $\tZ, \tA, \tW$ and of the modulated variables $\tcuu, \tcbb$.

\begin{figure}[ht]
\centering
\includegraphics[width=0.60\textwidth]{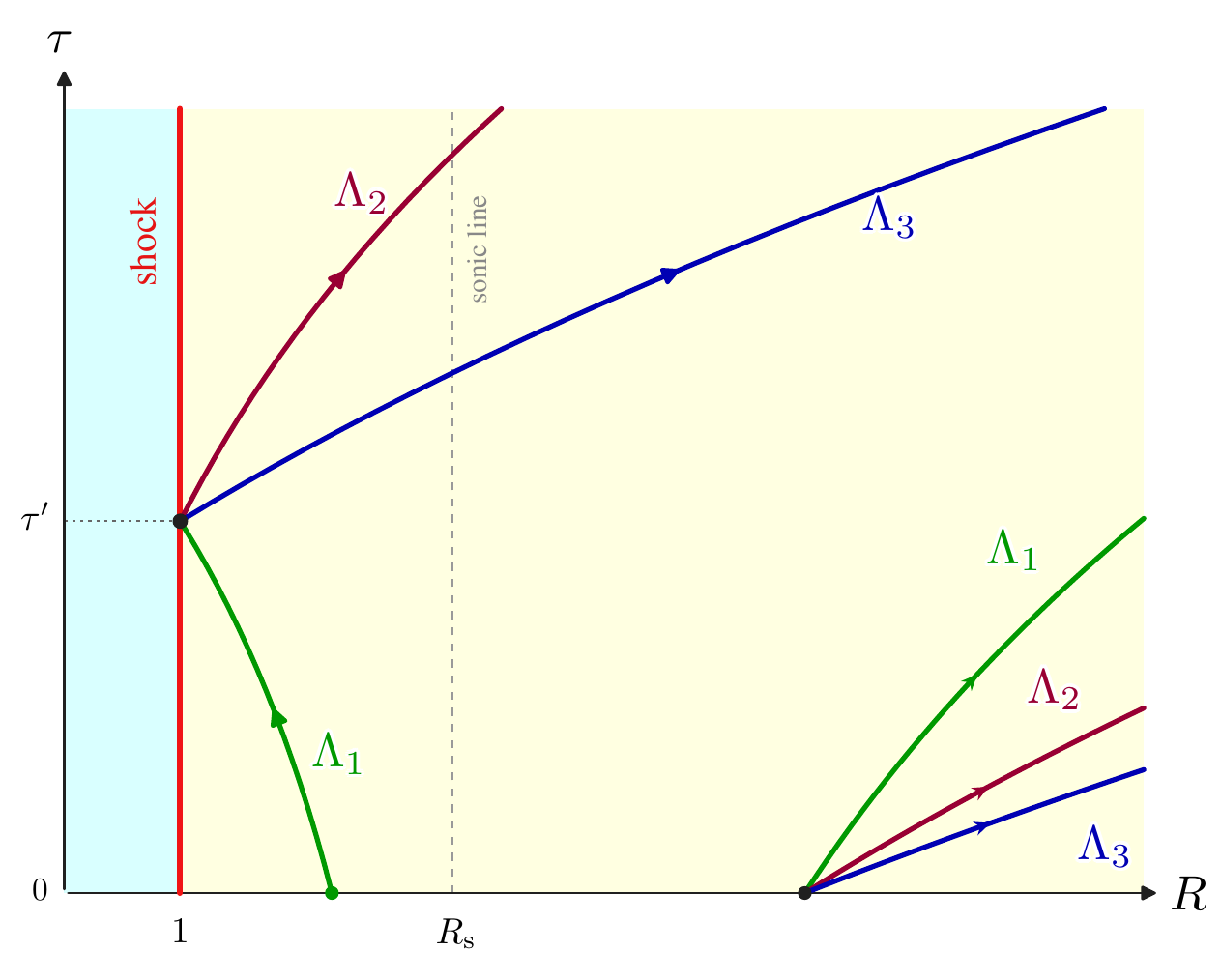}
\caption{\textbf{The Guderley characteristic geometry in self-similar variables.} The horizontal axis represents the self-similar coordinate $R$, while the vertical axis represents the self-similar time $\tau$. The red line is the shock $R=1$, which separates the interior region ahead of the shock (in light blue) and the exterior region behind the shock (in yellow).  The dashed line represents the sonic line at $R=\Rs$.  The green curve on the left is a $\Lambda_1$ characteristic curve starting from $\tau=0$ and reaching the shock at time $\tau'$. The magenta and blue curves on the left are $\Lambda_2$ and $\Lambda_3$ characteristics, respectively, emanating from the same point. On the right of the sonic line, in the same color scheme, we have three characteristics $\Lambda_1$, $\Lambda_2$ and $\Lambda_3$ escaping to infinity.}
\label{fig:guderley-shock-characteristics}
\end{figure}

\section{The Guderley imploding shock} \label{sec:guderley}
As we discussed in the introduction, the Guderley imploding shock wave is a globally self-similar imploding shock solution of the form~\eqref{bw:ss} with self-similar exponents
\begin{equation}
\cbbb=\crrb-1,
\qquad
0<\crrb<1.
\label{eq:guderley:exponents}
\end{equation}
The exponent $\crrb$ is not determined by dimensional analysis. Instead, it is selected by requiring that the solution of the self-similar system~\eqref{bw:ss:euler}, starting from the Rankine--Hugoniot data at the shock, pass regularly through the sonic point. Thus, $\crrb$ is determined through a nonlinear eigenvalue problem, which is classically solved numerically by a shooting procedure.

Using $\cbbb=\crrb-1$, we introduce its self-similar profile $\overline{\mathsf R}$ by
\begin{equation*}
\bar\rho(r,t)=\overline{\mathsf R}(R),
\qquad
\overline{\mathsf R}(R)
=\left(\tfrac{\alpha\bar\Sigma(R)}{\bar B(R)}\right)^{\frac{1}{\alpha}}
\quad\text{for }R>1.
\end{equation*}
The shock is fixed at $R=1$ in self-similar variables and, in physical variables, travels along
\begin{equation}
\s(t)=(-t)^{\crrb},
\qquad
\ds(t)=-\crrb(-t)^{\crrb-1}<0.
\label{eq:guderley:shock}
\end{equation}
The gas ahead of the shock, in the interior region $0\le r<\s(t)$, is cold and quiescent:
\begin{equation*}
u=0,
\qquad
\rho=1,
\qquad
p=\sigma=b=0.
\end{equation*}
Equivalently, throughout the interior self-similar region $0\le R<1$, the profiles are
\begin{equation}
\bar U(R)=0,
\qquad
\bar\Sigma(R)=0,
\qquad
\bar B(R)=0,
\qquad
\overline{\mathsf R}(R)=1.
\label{eq:guderley:interior-profiles}
\end{equation}
In the exterior region $r>\s(t)$, or equivalently $R>1$, the gas is described by the self-similar profiles solving~\eqref{bw:ss:euler}. Substituting the interior state into~\eqref{intro:RH:re}, with $\ds$ given by~\eqref{eq:guderley:shock}, yields the exterior traces
\begin{subequations}
\label{eq:guderley:profile-traces}
\begin{align}
\bar U(1^+)
&=-\tfrac{2\crrb}{\gamma+1},
&\overline{\mathsf R}(1^+)
&=\tfrac{\gamma+1}{\gamma-1},\label{eq:guderley:profile-traces-urho}\\
\bar\Sigma(1^+)
&=\tfrac{\crrb\sqrt{2\gamma(\gamma-1)}}{\alpha(\gamma+1)},
&\bar B(1^+)
&=\tfrac{\crrb\sqrt{2\gamma(\gamma-1)}}{\gamma+1}
\left(\tfrac{\gamma+1}{\gamma-1}\right)^{-\alpha}.
\label{eq:guderley:profile-traces-sigmab}
\end{align}
\end{subequations}
Together, \eqref{eq:guderley:interior-profiles} and~\eqref{eq:guderley:profile-traces} give the  values of the self-similar profiles on the two sides of the shock. They satisfy the Rankine--Hugoniot conditions~\eqref{intro:RH} and the Lax inequalities~\eqref{intro:lax}.

\subsection{The autonomous phase portrait}
Following the renormalization used in~\cite{CSV26}, for $R>1$ we introduce
\begin{equation}
\Vbar(R):=\tfrac{\bar U(R)}{R},
\qquad
\Qbar(R):=\tfrac{\bar\Sigma(R)}{R}.
\label{eq:guderley:VQ-def}
\end{equation}
Using~\eqref{bw:ss:euler:B} and the identity $\cbbb=\crrb-1$ we obtain
\begin{equation}
\tfrac{R\partial_R\bar B}{\bar B}
=\tfrac{\cbbb}{\crrb+\Vbar}
=\tfrac{\crrb-1}{\crrb+\Vbar}.
\label{eq:guderley:B-elimination}
\end{equation}
Substituting~\eqref{eq:guderley:VQ-def} and~\eqref{eq:guderley:B-elimination} into~\eqref{bw:ss:euler} yields the $2\times2$ ODE for $(\bar V,\bar Q)$
\begin{subequations}
\label{eq:guderley:VQ-system}
\begin{align}
&\bigl(1+(1+\alpha d)\Vbar\bigr)\Qbar
+(\crrb+\Vbar)R\partial_R\Qbar
+\alpha\Qbar R\partial_R\Vbar=0,
\label{eq:guderley:VQ-system-Q}\\
&(\crrb+\Vbar)
\bigl(\Vbar+\Vbar^2+\alpha\Qbar^2\bigr)
-\tfrac{\alpha(\crrb-1)}{\gamma}\Qbar^2
+\alpha(\crrb+\Vbar)\Qbar R\partial_R\Qbar
+(\crrb+\Vbar)^2R\partial_R\Vbar=0.
\label{eq:guderley:VQ-system-V}
\end{align}
\end{subequations}
Solving~\eqref{eq:guderley:VQ-system} for $R\partial_R\Vbar$ and $R\partial_R\Qbar$ gives the autonomous ODE
\begin{subequations}
\label{eq:guderley:VQ-autonomous}
\begin{align}
R\partial_R\Vbar
&=\tfrac{\Delta_{\Vbar}[\Vbar,\Qbar]}{\Delta[\Vbar,\Qbar]},
\label{eq:guderley:V-autonomous}\\
R\partial_R\Qbar
&=\tfrac{\Delta_{\Qbar}[\Vbar,\Qbar]}{\Delta[\Vbar,\Qbar]}.
\label{eq:guderley:Q-autonomous}
\end{align}
Here
\begin{align}
\Delta[\Vbar,\Qbar]
&:=(\crrb+\Vbar)(\crrb+\Vbar-\alpha\Qbar)
(\crrb+\Vbar+\alpha\Qbar),
\label{eq:guderley:Delta-def}\\
\Delta_{\Vbar}[\Vbar,\Qbar]
&:=(\crrb+\Vbar)
\left[
\alpha\Qbar^2\bigl(1+(1+\alpha d)\Vbar\bigr)
-(\crrb+\Vbar)\bigl(\Vbar+\Vbar^2+\alpha\Qbar^2\bigr)
+\tfrac{\alpha(\crrb-1)}{\gamma}\Qbar^2
\right],
\label{eq:guderley:Delta-V-def}\\
\Delta_{\Qbar}[\Vbar,\Qbar]
&:=\alpha\Qbar
\left[
(\crrb+\Vbar)\bigl(\Vbar+\Vbar^2+\alpha\Qbar^2\bigr)
-\tfrac{\alpha(\crrb-1)}{\gamma}\Qbar^2
\right]
-(\crrb+\Vbar)^2\bigl(1+(1+\alpha d)\Vbar\bigr)\Qbar.
\label{eq:guderley:Delta-Q-def}
\end{align}
\end{subequations}
\begin{samepage}
$(\Vbar,\Qbar)$ is then constructed by solving the autonomous ODE \eqref{eq:guderley:VQ-autonomous} on $R \ge 1$, with boundary conditions determined by the Rankine--Hugoniot jump conditions~\eqref{eq:guderley:profile-traces}:
\begin{equation}
\Vbar(1^+)=-\tfrac{2\crrb}{\gamma+1},
\qquad
\Qbar(1^+)=\tfrac{\crrb\sqrt{2\gamma(\gamma-1)}}{\alpha(\gamma+1)}.
\label{eq:guderley:VQ-shock-data}
\end{equation}
\end{samepage}
The exponent $\crrb$ is chosen in order to guarantee that at the point where $\Delta[\bar V, \bar Q]$ vanishes, $\Delta_{\bar V}[\bar V, \bar Q]$ and $\Delta_{\bar Q}[\bar V, \bar Q]$ vanish as well. An elementary computation reveals that in the region $\Qbar>0$, $\crrb+\Vbar>0$, only two such points exist.  At such a point, the vanishing of $\Delta$ gives 
\begin{equation*}
\Qbar=\tfrac{\crrb+\Vbar}{\alpha}.
\end{equation*}
Substituting into the equations $\Delta_{\bar V}[\bar V, \bar Q]=\Delta_{\bar Q}[\bar V, \bar Q]=0$ gives the quadratic equation
\begin{equation}
\begin{aligned}
0={}&\gamma(d-1)\bigl(\crrb+\Vbar\bigr)^2-\bigl(\crrb d\gamma-2\crrb\gamma+2\crrb+\gamma-2\bigr)
\bigl(\crrb+\Vbar\bigr)
+\crrb\gamma(1-\crrb).
\end{aligned}
\label{eq:guderley:sonic-quadratic}
\end{equation}

\begin{figure}[ht]
\centering
\includegraphics[width=0.60\textwidth]{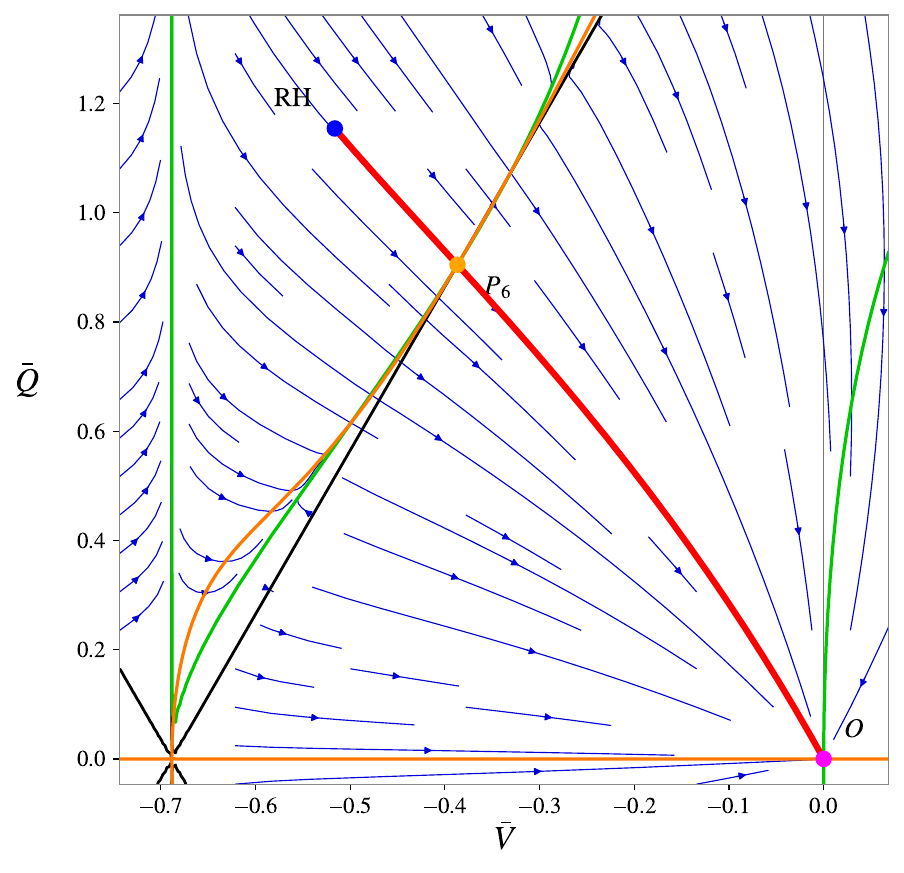}
\caption{\textbf{The Guderley implosion phase portrait for $\gamma=\frac{5}{3}$ and $d=3$.} The horizontal and vertical axes represent $\Vbar$ and $\Qbar$, respectively. The black curves represent $\{\Delta=0\}$, the green curves represent $\{\Delta_{\Vbar}=0\}$, and the orange curves represent $\{\Delta_{\Qbar}=0\}$. The red curve is the numerical Guderley trajectory from the blue Rankine--Hugoniot point $\mathrm{RH}$, through the orange triple point $P_6$, and to the magenta origin. The blue curves are streamlines of the autonomous phase field.}
\label{fig:guderley-phase-portrait}
\end{figure}
In the notation of~\cite{JLS25b}\footnote{In the notation of~\cite{JLS25b}, the variables $V,C,x$ correspond respectively to $\frac{\Vbar}{\crrb}$, $\frac{\alpha\Qbar}{\crrb}$, and $-R^{-\frac{1}{\crrb}}$, while $m=d-1$ and $z=\frac{\crrb^{-1}-1}{(d-1)\gamma}$.} (see for instance Lemmas~2.2 and~2.6 there), the two physical roots of~\eqref{eq:guderley:sonic-quadratic} correspond to the points $P_6=(\Vbar_6, \Qbar_6)$ and $P_8=(\bar V_8, \bar Q_8)$.  In the range $1<\gamma\le\frac{5}{3}$, only $P_6$ is relevant for the Guderley analysis; see Proposition~4.7 in~\cite{JLS25b}. We therefore record only its coordinates
\begin{subequations}
\label{eq:guderley:sonic-values}
\begin{align}
\Vbar_6
&=\tfrac{\gamma-2+2\crrb-\crrb d\gamma
-\sqrt{\bigl(\crrb d\gamma-2\crrb\gamma+2\crrb+\gamma-2\bigr)^2
-4\gamma^2(d-1)\crrb(1-\crrb)}}
{2\gamma(d-1)},
\label{eq:guderley:sonic-V-value}\\
\Qbar_6
&=\tfrac{\crrb d\gamma-2\crrb\gamma+2\crrb+\gamma-2
-\sqrt{\bigl(\crrb d\gamma-2\crrb\gamma+2\crrb+\gamma-2\bigr)^2
-4\gamma^2(d-1)\crrb(1-\crrb)}}
{2\alpha\gamma(d-1)}.
\notag
\end{align}
\end{subequations}

\subsection{Properties of the Guderley profiles}
We now state the rigorous existence theorem (in the adiabatic range under consideration) proved in~\cite{JLS25b}, together with monotonicity properties and asymptotic properties of the profiles. We start by recording the monotonicity properties from Theorem~2.9 and Proposition~8.1 in~\cite{JLS25b} and the asymptotic analysis from Lemma~2.7 in~\cite{JLS25a}.

\begin{theorem}[Theorem~2.9, Proposition~8.1 in~\cite{JLS25b} and Lemma~2.7 in~\cite{JLS25a}]
\label{lem:guderley-profile-monotonicity-decay}
Let $d=3$. For every $1 < \gamma \le \frac{5}{3}$ there exists a unique self-similar exponent $\crrb$ such that the ODE~\eqref{eq:guderley:VQ-autonomous} with initial condition~\eqref{eq:guderley:VQ-shock-data} admits an analytic solution $(\bar V, \bar Q)(R)$ on $[1,+\infty)$ satisfying the monotonicity and decay properties stated below. For every $R \ge 1$, the Guderley profiles satisfy
\begin{equation*}
\partial_R\Vbar(R)>0,
\quad
-\tfrac{2\crrb}{\gamma+1} \le \Vbar(R)<0,
\quad
\partial_R\Qbar(R)<0,
\quad
0<\Qbar(R) \le \tfrac{\crrb\sqrt{2\gamma(\gamma-1)}}{\alpha(\gamma+1)}.
\end{equation*}
For all $k \ge 0$, we have the following decay at infinity
\begin{equation} \label{eq:guderley-profile-decay}
\pp_R^k \Vbar(R)=\mathcal{O}\left(R^{-\frac{1}{\crrb}-k}\right),
\qquad
\pp_R^k \Qbar(R)=\mathcal{O}\left(R^{-\frac{1}{\crrb}-k}\right).
\end{equation}
Moreover, there exists a unique $\Rs>1$ such that $(\bar V, \bar Q)(\Rs)=P_6$. 
\end{theorem}
In Lemma~2.7 in~\cite{JLS25a}, the authors prove~\eqref{eq:guderley-profile-decay} only for $k=0,1$; the decay estimates for higher derivatives follow likewise by differentiating the autonomous system~\eqref{eq:guderley:VQ-autonomous}.

By taking advantage of the monotonicity of $\Vbar$, it is convenient to regard $\Qbar$ as a function of $\Vbar$. Lemma~\ref{lem:guderley-profile-monotonicity-decay} shows that
\begin{equation*}
\Vbar\colon[1,\infty)\longrightarrow
\left[-\tfrac{2\crrb}{\gamma+1},0\right)
\end{equation*}
is strictly increasing. We therefore define
\begin{equation*}
\mathcal{Q}(v):=\Qbar\bigl(\Vbar^{-1}(v)\bigr),
\qquad
-\tfrac{2\crrb}{\gamma+1}\leq v<0,
\end{equation*}
or, equivalently,
\begin{equation*}
\Qbar(R)=\mathcal{Q}(\Vbar(R)),
\qquad R\geq1.
\end{equation*}
By~\eqref{eq:guderley-profile-decay}, setting $\mathcal{Q}(0):=0$ extends $\mathcal{Q}$ continuously to $[-\frac{2\crrb}{\gamma+1},0]$, with endpoint values
\begin{equation*}
\mathcal{Q}\left(-\tfrac{2\crrb}{\gamma+1}\right)
=\tfrac{\crrb\sqrt{2\gamma(\gamma-1)}}{\alpha(\gamma+1)},
\qquad
\mathcal{Q}(0)=0.
\end{equation*}
Thus,
\begin{equation*}
\mathcal{Q}\colon
\left[-\tfrac{2\crrb}{\gamma+1},0\right]
\longrightarrow
\left[0,\tfrac{\crrb\sqrt{2\gamma(\gamma-1)}}{\alpha(\gamma+1)}\right]
\end{equation*}
is strictly decreasing. 

We observe that the first characteristic speed
\begin{equation*}
\bar\Lambda_1(R):=\crrb+\Vbar(R)-\alpha\Qbar(R)
\end{equation*}
is strictly increasing, since
\begin{equation}
\partial_R\bar\Lambda_1(R)
=\partial_R\Vbar(R)-\alpha\partial_R\Qbar(R)>0.
\label{eq:guderley:first-characteristic-monotonicity}
\end{equation}
Moreover,
\begin{equation*}
\bar\Lambda_1(1^+)
=\tfrac{\crrb}{\gamma+1}
\left(\gamma-1-\sqrt{2\gamma(\gamma-1)}\right)<0,
\qquad
\lim_{R\to\infty}\bar\Lambda_1(R)=\crrb>0.
\end{equation*}
Consequently, there exists a unique sonic coordinate $\Rs>1$ such that
\begin{equation}
\bar\Lambda_1(\Rs)=0.
\label{eq:guderley:unique-sonic-point}
\end{equation}

\subsubsection{The sonic point $P_6$}

At the sonic point $P_6$ we can formally compute the Taylor series of $\mathcal{Q}$. Such a Taylor series was constructed, and shown to converge, in Theorem~3.8 in~\cite{JLS25b}. Thus, for $v$ sufficiently close to $\bar V_6$, we have
\begin{equation*}
\mathcal{Q}(v)=\Qbar_6
+\sum_{\ell=1}^{\infty}\eta_\ell
\bigl(v-\Vbar_6\bigr)^\ell.
\end{equation*}
Here
\begin{equation*}
\eta_\ell
:=\tfrac{1}{\ell!}\mathcal{Q}^{(\ell)}(\Vbar_6),
\qquad \ell\geq 1.
\end{equation*}
The first Taylor coefficient is the negative root (see equation~(3.18) in~\cite{JLS25b})
\begin{equation}
\eta_1
=\mathcal{Q}'(\Vbar_6)
=\left.
\tfrac{
\partial_{\Qbar}\Delta_{\Qbar}-\partial_{\Vbar}\Delta_{\Vbar}
+\sqrt{\bigl(\partial_{\Vbar}\Delta_{\Vbar}
-\partial_{\Qbar}\Delta_{\Qbar}\bigr)^2
+4\partial_{\Qbar}\Delta_{\Vbar}
\partial_{\Vbar}\Delta_{\Qbar}}}
{2\partial_{\Qbar}\Delta_{\Vbar}}
\right|_{(\Vbar,\Qbar)=(\Vbar_6,\Qbar_6)}<0.
\label{eq:guderley:sonic-phase-slope}
\end{equation}
 Evaluating~\eqref{eq:guderley:VQ-system-Q} at $R=\Rs$ then gives
\begin{equation*}
\begin{aligned}
\Rs\partial_R\Vbar(\Rs)
&=-\tfrac{1+(1+\alpha d)\Vbar_6}
{\alpha(1+\eta_1)},
\\
\Rs\partial_R\Qbar(\Rs)
&=\eta_1\Rs\partial_R\Vbar(\Rs).
\end{aligned}
\end{equation*}
Here $1+\eta_1\neq0$, as follows from Theorem~2.9(ii) and Lemma~2.6 in~\cite{JLS25b}. The higher Taylor coefficients can be computed recursively from~\eqref{eq:guderley:VQ-autonomous}:
\begin{subequations}
\label{eq:guderley:sonic-Taylor-expansion}
\begin{equation*}
\Delta_{\Vbar}[v,\mathcal{Q}(v)]\mathcal{Q}'(v)
=\Delta_{\Qbar}[v,\mathcal{Q}(v)].
\end{equation*}
For every $\ell\geq2$, differentiating $\ell$ times at $v=\Vbar_6$ and isolating $\eta_\ell$ gives
\begin{align}
&\ell!\eta_\ell
\left.
\bigl[\ell\partial_{\Vbar}\Delta_{\Vbar}
+(\ell+1)\eta_1\partial_{\Qbar}\Delta_{\Vbar}
-\partial_{\Qbar}\Delta_{\Qbar}\bigr]
\right|_{(\Vbar,\Qbar)=(\Vbar_6,\Qbar_6)}
\notag\\
&\quad=\tfrac{\dd^\ell}{\dd v^\ell}
\Biggl\{
\Delta_{\Qbar}\left[v,\Qbar_6+
\sum_{j=1}^{\ell-1}\eta_j(v-\Vbar_6)^j\right]
\notag\\
&\qquad\quad-
\left(\sum_{j=1}^{\ell-1}j\eta_j(v-\Vbar_6)^{j-1}\right)
\Delta_{\Vbar}\left[v,\Qbar_6+
\sum_{j=1}^{\ell-1}\eta_j(v-\Vbar_6)^j\right]
\Biggr\}\bigg|_{v=\Vbar_6}.
\label{eq:guderley:sonic-coefficient-recurrence}
\end{align}
\end{subequations}
The right-hand side depends only on $\eta_1,\ldots,\eta_{\ell-1}$, and the coefficient in front of $\eta_\ell$ is nonzero at $P_6$ by~Lemmas~3.4--3.5 in~\cite{JLS25b}. Thus the recurrence uniquely determines each coefficient.

\section{Modulated self-similar and differentiated variables}
\label{sec:modulated-ring-variables}
To conduct the stability analysis we will use the modulated self-similar variables introduced in~\eqref{eq:modulated-coordinates}--\eqref{eq:modulated-ansatz}. In particular we will consider renormalized and differentiated versions of the variables $(U, \Sigma, B)$, which will be the objects of this section.

\subsection{Renormalized and differentiated variables}

Following~\cite{CSV26}, in the region $R\ge1$ we introduce
\begin{equation}
U(R,\tau)=RV(R,\tau),
\qquad
\Sigma(R,\tau)=RQ(R,\tau),
\qquad
B(R,\tau)=RH(R,\tau),
\qquad
K:=\log H.
\label{eq:VQK-def}
\end{equation}
The system~\eqref{eq:modulated-euler} becomes
\begin{subequations}
\label{eq:VQK-system}
\begin{align}
\partial_\tau Q
&+\bigl(\crr-\cuu+(1+\alpha d)V\bigr)Q
+(\crr+V)R\partial_RQ
+\alpha QR\partial_RV=0,
\label{eq:Q-eq}\\
\partial_\tau V
&+(\crr-\cuu)V+V^2+\tfrac{2\alpha^2}{\gamma}Q^2
+(\crr+V)R\partial_RV
+\alpha QR\partial_RQ
-\tfrac{\alpha}{\gamma}Q^2R\partial_RK=0,
\label{eq:V-eq}\\
\partial_\tau K
&+(\crr-\cbb+V)+(\crr+V)R\partial_RK=0.
\label{eq:K-eq}
\end{align}
\end{subequations}
The physical density is recovered through
\begin{equation}
\rho(r,t)
=\left(\tfrac{\Cu(\tau)}{\Cb(\tau)}\right)^{\frac{1}{\alpha}}
\left(\alpha Q(R,\tau)e^{-K(R,\tau)}\right)^{\frac{1}{\alpha}}.
\label{eq:density-VQK}
\end{equation}
We will indicate the associated Guderley profiles with
\begin{equation}
(V,Q,K)=(\Vbar,\Qbar,\Kbar),
\qquad
\Kbar(R):=\log\left(\tfrac{\Bbar(R)}{R}\right).
\label{eq:guderley-VQK-background}
\end{equation}
We write perturbations of the background Guderley profile as
\begin{equation}
V=\Vbar+\widetilde V,
\qquad
Q=\Qbar+\widetilde Q,
\qquad
K=\Kbar+\widetilde K.
\label{eq:VQK-perturbations}
\end{equation}
For the entropy, we may also write:
\begin{equation*}
\widetilde B:=B-\Bbar.
\end{equation*}
We also set $\tcrr:=\crr-\crrb$.

Following Section~3.6 in~\cite{CSV26},  we define the differentiated Riemann-type variables
\begin{equation}
\label{eq:ring-def}
\rZ:=R\partial_RV-R\partial_RQ
+\tfrac{1}{\gamma}QR\partial_RK,
\quad
\rA:=\tfrac{\alpha}{\gamma}QR\partial_RK,
\quad
\rW:=R\partial_RV+R\partial_RQ
-\tfrac{1}{\gamma}QR\partial_RK.
\end{equation}
These are equivalent to
\begin{equation}
\label{eq:ring-inverse}
R\partial_RV=\tfrac{1}{2}(\rZ+\rW),
\qquad
R\partial_RQ=-\tfrac{1}{2}\rZ+\tfrac{1}{\alpha}\rA
+\tfrac{1}{2}\rW,
\qquad
R\partial_RK=\tfrac{\gamma}{\alpha Q}\rA.
\end{equation}
We will use
\begin{equation*}
\Y:=(\rZ,\rA,\rW)^\top.
\end{equation*}
Applying $R\partial_R$ to~\eqref{eq:VQK-system} and diagonalizing gives
\begin{equation}
\partial_\tau\Y+\Tmat R\partial_R\Y+\Dmat\Y+\Ncal(\Y,\Y)=0,
\label{eq:ring-evolution}
\end{equation}
where the diagonal transport matrix is
\begin{equation}
\Tmat:=(\crr+V)I+\alpha Q\diag(-1,0,1),
\label{eq:ring-transport-matrix}
\end{equation}
and the damping matrix is
\begin{equation}
\Dmat:=(\crr-\cuu)I+VD^{(v)}+\alpha QD^{(q)},
\label{eq:ring-damping-matrix}
\end{equation}
with
\begin{subequations}
\label{eq:ring-constant-matrices}
\begin{align}
D^{(v)}&:=
\begin{pmatrix}
\tfrac{3+\alpha d}{2} & 0 & \tfrac{1-\alpha d}{2}\\
0 & 1+\alpha d & 0\\
\tfrac{1-\alpha d}{2} & 0 & \tfrac{3+\alpha d}{2}
\end{pmatrix},
\label{eq:ring-Dv}\\
D^{(q)}&:=
\begin{pmatrix}
-\tfrac{d+2}{2} & \tfrac{4}{\gamma}
&-\tfrac{(d+2)+2\alpha(d-2)}{2\gamma}\\
\tfrac{1}{2\gamma} & 0 & \tfrac{1}{2\gamma}\\
\tfrac{(d+2)+2\alpha(d-2)}{2\gamma}
&\tfrac{4}{\gamma} & \tfrac{d+2}{2}
\end{pmatrix}.
\label{eq:ring-Dq}
\end{align}
\end{subequations}
For $i\in\{1,2,3\}$, the components of the quadratic bilinear form $\mathcal{N}$ are
\begin{equation}
\begin{aligned}
\Ncal_i(\Y',\Y'')
={}&\left(
\tfrac{1+(2-i)\alpha}{2}Y'_1+(i-2)Y'_2
+\tfrac{1+(i-2)\alpha}{2}Y'_3
\right)Y''_i\\
&+\left(\tfrac{2-i}{2}+\one_{i=2}\tfrac{\alpha}{2}\right)
(Y'_1+Y'_3)Y''_2.
\label{eq:ring-quadratic-form}
\end{aligned}
\end{equation}
Finally, if we let $\Ybar$ denote the Riemann variables obtained from $(\Vbar,\Qbar,\Kbar)$ through~\eqref{eq:ring-def}, we set
\begin{equation}
\Ytil:=\Y-\Ybar=(\tZ,\tA,\tW)^\top.
\label{eq:ring-perturbation}
\end{equation}
Subtracting the stationary Guderley equation from~\eqref{eq:ring-evolution} gives the equations for the perturbation $\Ytil$
\begin{equation}
\partial_\tau\Ytil
+\Tmat R\partial_R\Ytil
+\Dmat\Ytil
=-\widetilde{\Tmat} R\partial_R\Ybar
-\widetilde{\Dmat}\Ybar
-\Ncal(\Ybar,\Ytil)
-\Ncal(\Ytil,\Ybar)
-\Ncal(\Ytil,\Ytil).
\label{eq:ring-perturbation-system}
\end{equation}
Here
\begin{subequations}
\label{eq:ring-perturbation-matrices}
\begin{align} 
\widetilde{\Tmat}
&:=(\tcrr+\widetilde V)I
+\alpha\widetilde Q\diag(-1,0,1),
\label{eq:ring-perturbation-Ttilde}\\
\widetilde{\Dmat}
&:=(\tcrr-\tcuu)I
+\widetilde V D^{(v)}
+\alpha\widetilde QD^{(q)}.
\label{eq:ring-perturbation-Dtilde}
\end{align}
\end{subequations}

\subsection{Choice of the modulation functions}
We now choose the modulation functions so that the shock remains at $R=1$ in the modulated self-similar radial coordinate~\eqref{eq:modulated-radius}. This is equivalent to
\begin{equation}
\ss(t(\tau))=\Cr(\tau).
\label{eq:shock-freezing-gauge}
\end{equation}
Differentiating this identity in $\tau$ and using~\eqref{eq:scaling-factor-r} and~\eqref{eq:modulated-time} gives
\begin{equation}
\ds(t(\tau))
=\tfrac{\partial_\tau\Cr(\tau)}{\partial_\tau t(\tau)}
=-\crr(\tau)\Cu(\tau).
\label{eq:shock-freezing-speed}
\end{equation}
The next two modulation choices will simplify the Rankine--Hugoniot conditions for the variables $(\widetilde V,\widetilde Q,\widetilde K)$ derived below:
\begin{equation}
\tcrr\equiv0,
\qquad
\Cu=\Cb.
\label{eq:cr:0}
\end{equation}
This implies
\begin{equation}
\crr=\crrb,
\qquad
\cuu=\cbb=\crrb-1+\tcuu.
\label{eq:fixed-radial-equal-amplitude-consequences}
\end{equation}
From~\eqref{eq:cr:0} and~\eqref{eq:shock-freezing-speed} at $\tau=0$ we obtain
\begin{equation*}
C_u(0)=C_b(0)=-\tfrac{\ds(t(0))}{\crrb}.
\end{equation*}
Since $\ss(-1)=1$, we have
\begin{equation*}
C_r(0)=1.
\end{equation*}

\subsubsection{Rankine--Hugoniot conditions in self-similar variables}
\label{subsec:rh-modulated}
Using~\eqref{eq:cr:0}, the shock speed~\eqref{eq:shock-freezing-speed} becomes
\begin{equation}
\ds(t(\tau))=-\crrb\Cu(\tau).
\label{eq:fixed-shock-speed}
\end{equation}
At $R=1$, we have $U=V$, $\Sigma=Q$, and $B=e^K$ by~\eqref{eq:VQK-def}. Since $\Cu=\Cb$, the density formula~\eqref{eq:density-VQK} reduces to $ \rho=\left(\alpha Qe^{-K}\right)^{\frac{1}{\alpha}}. $ Using $p=\frac{\alpha^2}{\gamma}\rho\sigma^2$, we substitute \eqref{eq:modulated-ansatz} and~\eqref{eq:fixed-shock-speed} into the Rankine--Hugoniot jump conditions~\eqref{intro:RH}. Simplifying $\Cu$ gives the following jumps at $R=1$
\begin{subequations}
\label{eq:fixed-shock-RH-conservative}
\begin{align}
-\crrb\jump{\rho V}
&=\jump{\rho V^2+\tfrac{\alpha^2}{\gamma}\rho Q^2},\\
-\crrb\jump{\rho}
&=\jump{\rho V},\\
-\crrb\jump{\rho\left(V^2+\tfrac{\alpha}{\gamma}Q^2\right)}
&=\jump{\rho V\left(V^2+\alpha Q^2\right)}.
\end{align}
\end{subequations}
Fix now a constant $1\gg p_0 \ge 0$. We fix the state ahead of the shock (the upstream state) to be
\begin{equation}
\rho=1,
\qquad
u^r=0,
\qquad
p=p_0.
\label{eq:finite-pressure-interior-state}
\end{equation}
At $R=1^-$ (if $p_0>0$), in self-similar variables we obtain
\begin{equation*}
V(1^-,\tau)=0,
\qquad
\alpha Q(1^-,\tau)=\tfrac{\sqrt{\gamma p_0}}{\Cu(\tau)},
\qquad
K(1^-,\tau)=\log\left(\tfrac{\sqrt{\gamma p_0}}{\Cu(\tau)}\right).
\end{equation*}
As long as $ \gamma p_0<\crrb^2\Cu(\tau)^2, $ the shock satisfies the Lax inequalities~\eqref{intro:lax}. We can now compute the downstream values of the self-similar variables by using the Rankine--Hugoniot conditions~\eqref{eq:fixed-shock-RH-conservative} (equivalently, using~\eqref{intro:RH:re})
\begin{subequations}
\label{eq:finite-pressure-RH-traces}
\begin{align}
V(1^+,\tau)
&=-\tfrac{2\crrb}{\gamma+1}
+\tfrac{2\gamma p_0}
{(\gamma+1)\crrb\Cu(\tau)^2},
\label{eq:finite-pressure-RH-V}\\
\alpha Q(1^+,\tau)
&=\tfrac{\crrb}{\gamma+1}
\left[
\gamma\left(2-\tfrac{(\gamma-1)p_0}
{\crrb^2\Cu(\tau)^2}\right)
\left((\gamma-1)+\tfrac{2\gamma p_0}
{\crrb^2\Cu(\tau)^2}\right)
\right]^{\frac{1}{2}},
\label{eq:finite-pressure-RH-Q}\\
\rho(1^+,\tau)
&=\tfrac{\gamma+1}
{(\gamma-1)+\tfrac{2\gamma p_0}
{\crrb^2\Cu(\tau)^2}},
\label{eq:finite-pressure-RH-density}\\
K(1^+,\tau)
&=\log\bigl(\alpha Q(1^+,\tau)\bigr)
-\alpha\log\rho(1^+,\tau).
\label{eq:finite-pressure-RH-K}
\end{align}
\end{subequations}
Subtracting the Guderley traces~\eqref{eq:guderley:profile-traces} gives the Rankine--Hugoniot conditions for the perturbations:
\begin{subequations}
\label{eq:finite-pressure-RH-tilded-traces}
\begin{align}
\widetilde V(1^+,\tau)
&=\tfrac{2\gamma p_0}
{(\gamma+1)\crrb\Cu(\tau)^2},
\label{eq:finite-pressure-RH-tilded-V}\\
\widetilde Q(1^+,\tau)
&=\Qbar(1^+)\left[
\sqrt{
\left(1-\tfrac{\alpha p_0}{\crrb^2\Cu(\tau)^2}\right)
\left(1+\tfrac{\gamma p_0}{\alpha\crrb^2\Cu(\tau)^2}\right)
}-1
\right],
\label{eq:finite-pressure-RH-tilded-Q}\\
\widetilde K(1^+,\tau)
&=\tfrac{1}{2}\log\left(
1-\tfrac{\alpha p_0}{\crrb^2\Cu(\tau)^2}
\right)
+\tfrac{\gamma}{2}\log\left(
1+\tfrac{\gamma p_0}{\alpha\crrb^2\Cu(\tau)^2}
\right).
\label{eq:finite-pressure-RH-tilded-K}
\end{align}
\end{subequations}
In particular, we observe that\footnote{Here, the implicit constant in the inequality $\lesssim$ depends exclusively on $\gamma$, $\crrb$ and the Guderley profiles $(\bar V, \bar Q, \bar K)$.}
\begin{equation}
\label{eq:finite-pressure-RH-trace-bound}
|\widetilde V(1^+,\tau)|
+|\widetilde Q(1^+,\tau)|
+|\widetilde K(1^+,\tau)|
\lesssim\tfrac{p_0}{\Cu(\tau)^2}.
\end{equation}
Differentiating~\eqref{eq:finite-pressure-RH-tilded-traces} in $\tau$ and using $\partial_\tau\Cu=-\cuu\Cu$, then simplifying with~\eqref{eq:finite-pressure-RH-Q}, gives
\begin{subequations}
\label{eq:finite-pressure-RH-tilded-time-derivatives}
\begin{align}
\partial_\tau\widetilde V(1^+,\tau)
&=\tfrac{4\gamma p_0\cuu(\tau)}
{(\gamma+1)\crrb\Cu(\tau)^2},\\
\partial_\tau\widetilde Q(1^+,\tau)
&=\tfrac{\gamma p_0\cuu(\tau)}
{\alpha^2(1+\alpha)^2\Cu(\tau)^2Q(1^+,\tau)}
\left(
\gamma-\alpha^2
-\tfrac{2\alpha\gamma p_0}{\crrb^2\Cu(\tau)^2}
\right),\\
\partial_\tau\widetilde K(1^+,\tau)
&=\tfrac{\gamma p_0\cuu(\tau)}
{\alpha^2(1+\alpha)\Cu(\tau)^2Q(1^+,\tau)^2}
\left(
1+3\alpha
-\tfrac{2\alpha\gamma p_0}{\crrb^2\Cu(\tau)^2}
\right).
\end{align}
\end{subequations}
Using the equations~\eqref{eq:VQK-system} and the inverse formulas~\eqref{eq:ring-inverse}, we obtain
\begin{subequations}
\label{eq:finite-pressure-RH-normal-system}
\begin{align}
&\bigl(\crrb+V\bigr)\tA
=\tfrac{\alpha}{\gamma}Q
\left(-\partial_\tau\widetilde K+\tcuu-\widetilde V\right)
-\tfrac{\alpha}{\gamma}(1+\Vbar)\widetilde Q
-\widetilde V\,\overline{\rA},
\label{eq:finite-pressure-RH-normal-K}\\
&\bigl(\crrb+V\bigr)
\left(-\tfrac{1}{2}\tZ+\tfrac{1}{\alpha}\tA+\tfrac{1}{2}\tW\right)
+\tfrac{\alpha}{2}Q(\tZ+\tW)
=-\partial_\tau\widetilde Q+\tcuu Q
-(1+\alpha d)\widetilde VQ
\notag\\
&\qquad
-\bigl(1+(1+\alpha d)\Vbar\bigr)\widetilde Q
-\widetilde V
\left(
-\tfrac{1}{2}\overline{\rZ}
+\tfrac{1}{\alpha}\overline{\rA}
+\tfrac{1}{2}\overline{\rW}
\right)
-\tfrac{\alpha}{2}\widetilde Q
\left(\overline{\rZ}+\overline{\rW}\right),
\label{eq:finite-pressure-RH-normal-Q}\\
&\tfrac{1}{2}\bigl(\crrb+V-\alpha Q\bigr)\tZ
+\tfrac{1}{2}\bigl(\crrb+V+\alpha Q\bigr)\tW
=-\partial_\tau\widetilde V+\tcuu V
-(1+2\Vbar)\widetilde V-(\widetilde V)^2
\notag\\
&\qquad
-\tfrac{4\alpha^2}{\gamma}\Qbar\widetilde Q
-\tfrac{2\alpha^2}{\gamma}(\widetilde Q)^2
-\tfrac{1}{2}\bigl(\widetilde V-\alpha\widetilde Q\bigr)
\overline{\rZ}
-\tfrac{1}{2}\bigl(\widetilde V+\alpha\widetilde Q\bigr)
\overline{\rW}.
\label{eq:finite-pressure-RH-normal-V}
\end{align}
\end{subequations}
Evaluating the  equations~\eqref{eq:finite-pressure-RH-normal-system} at the shock location $R=1$, using the Rankine--Hugoniot jump conditions~\eqref{eq:finite-pressure-RH-tilded-traces}, and the associated time derivatives~\eqref{eq:finite-pressure-RH-tilded-time-derivatives}, leads to
\begin{equation}
\mathsf R_{\mathrm{RH}}
\begin{pmatrix}
\tZ\\[2pt]
\tA\\[2pt]
\tW
\end{pmatrix}(1^+,\tau)
=
\begin{pmatrix}
\tfrac{\gamma+1}{\crrb(\gamma-1)}\\[4pt]
\tfrac{\crrb\sqrt{2\gamma(\gamma-1)}}{\alpha(\gamma+1)}\\[4pt]
-\tfrac{2\crrb}{\gamma+1}
\end{pmatrix}
\tcuu(\tau)
+\mathcal P_{p_0}\bigl(\tau,\tcuu(\tau),\Ytil(1^+,\tau)\bigr),
\label{eq:finite-pressure-RH-ring-system}
\end{equation}
where 
\begin{equation*}
\mathsf R_{\mathrm{RH}}
:=
\begin{pmatrix}
0
&\tfrac{\gamma(\gamma+1)}
{\crrb\sqrt{2\gamma(\gamma-1)}}
&0
\\[4pt]
\tfrac{\crrb}{2(\gamma+1)}
\bigl(\sqrt{2\gamma(\gamma-1)}-\gamma+1\bigr)
&\tfrac{2\crrb}{\gamma+1}
&\tfrac{\crrb}{2(\gamma+1)}
\bigl(\sqrt{2\gamma(\gamma-1)}+\gamma-1\bigr)
\\[4pt]
\tfrac{\crrb}{2(\gamma+1)}
\bigl(\gamma-1-\sqrt{2\gamma(\gamma-1)}\bigr)
&0
&\tfrac{\crrb}{2(\gamma+1)}
\bigl(\gamma-1+\sqrt{2\gamma(\gamma-1)}\bigr)
\end{pmatrix},
\end{equation*}
while $\mathcal{P}_{p_0}$ is given by
\begin{subequations}
\label{eq:finite-pressure-RH-remainder-components}
\begin{align}
\bigl(\mathcal P_{p_0}\bigr)_1
&=\tfrac{1}{(\crrb+\Vbar)\Qbar}
\Bigl\{
\widetilde Q\tcuu(\tau)-Q\widetilde V
-(1+\Vbar)\widetilde Q
-\tfrac{\gamma}{\alpha}\widetilde V
\bigl(\overline{\rA}+\tA\bigr)
\notag\\
&\qquad\qquad
-\tfrac{\gamma p_0\cuu(\tau)}
{\alpha^2(1+\alpha)\Cu(\tau)^2Q}
\left(1+3\alpha
-\tfrac{2\alpha\gamma p_0}{\crrb^2\Cu(\tau)^2}\right)
\Bigr\},
\label{eq:finite-pressure-RH-remainder-A}\\
\bigl(\mathcal P_{p_0}\bigr)_2
&=-\tfrac{\gamma p_0\cuu(\tau)}
{\alpha^2(1+\alpha)^2\Cu(\tau)^2Q}
\left(\gamma-\alpha^2
-\tfrac{2\alpha\gamma p_0}{\crrb^2\Cu(\tau)^2}\right)
+\tcuu(\tau)\widetilde Q
-(1+\alpha d)\widetilde VQ
\notag\\
&\quad
-\bigl(1+(1+\alpha d)\Vbar\bigr)\widetilde Q
-\widetilde V
\left[
-\tfrac{1}{2}\bigl(\overline{\rZ}+\tZ\bigr)
+\tfrac{1}{\alpha}\bigl(\overline{\rA}+\tA\bigr)
+\tfrac{1}{2}\bigl(\overline{\rW}+\tW\bigr)
\right]
\notag\\
&\quad
-\tfrac{\alpha}{2}\widetilde Q
\bigl(\overline{\rZ}+\overline{\rW}+\tZ+\tW\bigr),
\label{eq:finite-pressure-RH-remainder-Q}\\
\bigl(\mathcal P_{p_0}\bigr)_3
&=-\tfrac{4\gamma p_0\cuu(\tau)}
{(\gamma+1)\crrb\Cu(\tau)^2}
+\tcuu(\tau)\widetilde V
-(1+2\Vbar)\widetilde V-(\widetilde V)^2
-\tfrac{4\alpha^2}{\gamma}\Qbar\widetilde Q
-\tfrac{2\alpha^2}{\gamma}(\widetilde Q)^2
\notag\\
&\quad
-\tfrac{1}{2}\bigl(\widetilde V-\alpha\widetilde Q\bigr)
\bigl(\overline{\rZ}+\tZ\bigr)
-\tfrac{1}{2}\bigl(\widetilde V+\alpha\widetilde Q\bigr)
\bigl(\overline{\rW}+\tW\bigr).
\label{eq:finite-pressure-RH-remainder-V}
\end{align}
\end{subequations}
In particular, combining these formulas with the bound~\eqref{eq:finite-pressure-RH-trace-bound}, we obtain
\begin{equation}
\mathcal P_0=0,
\qquad
\left|\mathcal P_{p_0}(\tau,\tcuu(\tau),\Ytil)\right|
\lesssim
\tfrac{p_0}{\Cu(\tau)^2}
\bigl(1+|\tcuu(\tau)|+|\Ytil|\bigr).
\label{eq:finite-pressure-RH-pressure-remainder}
\end{equation}
Solving \eqref{eq:finite-pressure-RH-ring-system} for $(\tcuu,\tW,\tA)$ as functions of $\tZ$ (since $\mathsf R_{\mathrm{RH}}$ is invertible and $0<\kappa_\gamma<\infty$), using the bound~\eqref{eq:finite-pressure-RH-pressure-remainder}, and using $\frac{p_0}{\Cu(\tau)^2}\ll1$ we obtain\footnote{In the cold case limit $p_0=0$, it is immediate to obtain \eqref{eq:finite-pressure-RH-solved-feedback}. For $p_0>0$, since in \eqref{eq:finite-pressure-RH-ring-system} the term $\mathcal{P}$ depends on the unknown $\Ytil$, to obtain the conclusion we need to apply the inverse function theorem, which can be applied easily under smallness assumptions for $\frac{p_0}{C_u^2}$.  }
\begin{subequations}
\label{eq:finite-pressure-RH-solved-feedback}
\begin{align}
\tcuu(\tau)
&=\kappa_\gamma\tZ(1^+,\tau)
+\Ocal\left(
\tfrac{p_0}{\Cu(\tau)^2}
\bigl(1+|\tZ(1^+,\tau)|\bigr)
\right),
\label{eq:finite-pressure-RH-solved-feedback-cu}\\
\tW(1^+,\tau)
&=b_W(\gamma)\tZ(1^+,\tau)
+\Ocal\left(
\tfrac{p_0}{\Cu(\tau)^2}
\bigl(1+|\tZ(1^+,\tau)|\bigr)
\right),
\label{eq:finite-pressure-RH-solved-feedback-W}\\
\tA(1^+,\tau)
&=b_A(\gamma)\tZ(1^+,\tau)
+\Ocal\left(
\tfrac{p_0}{\Cu(\tau)^2}
\bigl(1+|\tZ(1^+,\tau)|\bigr)
\right).
\label{eq:finite-pressure-RH-solved-feedback-A}
\end{align}
\end{subequations}
Here the coefficients in~\eqref{eq:profile-shock-feedback} are given by
\begin{subequations}
\label{eq:finite-pressure-RH-feedback-coefficients}
\begin{align}
\kappa_\gamma
&:=\tfrac{\gamma
\bigl(\sqrt{2\gamma(\gamma-1)}-\gamma+1\bigr)}
{2\bigl(\gamma+\sqrt{2\gamma(\gamma-1)}\bigr)},
\label{eq:finite-pressure-RH-feedback-coefficient-cu}\\
b_W(\gamma)
&:=\tfrac{
\bigl(\sqrt{2\gamma(\gamma-1)}-\gamma+1\bigr)
\bigl(\sqrt{2\gamma(\gamma-1)}-\gamma\bigr)}
{\bigl(\sqrt{2\gamma(\gamma-1)}+\gamma-1\bigr)
\bigl(\sqrt{2\gamma(\gamma-1)}+\gamma\bigr)},
\label{eq:finite-pressure-RH-feedback-coefficient-W}\\
b_A(\gamma)
&:=\tfrac{2\gamma-\sqrt{2\gamma(\gamma-1)}}
{2\bigl(\gamma+\sqrt{2\gamma(\gamma-1)}\bigr)}.
\label{eq:finite-pressure-RH-feedback-coefficient-A}
\end{align}
\end{subequations}
We also observe that the error term $\mathcal{P}_{p_0}$ vanishes in the cold gas limit, when $p_0=0$.

\subsection{Initial data in self-similar coordinates}

The initial data are prescribed at the self-similar initial time $\tau=0$ (which corresponds to the physical time $t=-1$), with a shock located at the self-similar coordinate $R=1$. In physical coordinates this corresponds to the shock being fixed at $r=\Cr(0)=1$.

Ahead of the shock we take the quiescent constant state
\begin{equation*}
\rho=1,
\qquad
u^r=0,
\qquad
p=p_0,
\qquad 0\leq R<1.
\end{equation*}
Equivalently, in the modulated variables,
\begin{equation*}
\Sigma(R,0)=\Sigma_0,
\qquad
U(R,0)=0,
\qquad
B(R,0)=\alpha\Sigma_0,
\qquad 0\leq R<1,
\end{equation*}
where
\begin{equation}
\Sigma_0:=\tfrac{\sqrt{\gamma p_0}}{\alpha\Cu(0)}.
\label{eq:initial-interior-sound-speed}
\end{equation}
Behind the shock, we take an appropriate perturbation of the Guderley profile. Therefore, the complete piecewise initial data are
\begin{equation}
\begin{pmatrix}
\Sigma\\
U\\
B
\end{pmatrix}(R,0)
=
\begin{cases}
\begin{pmatrix}
\Sigma_0\\[3pt]
0\\[3pt]
\alpha\Sigma_0
\end{pmatrix},
&0\leq R<1,\\[2em]
\begin{pmatrix}
\bar\Sigma\\
\bar U\\
\bar B
\end{pmatrix}(R)
+
\begin{pmatrix}
\widetilde\Sigma_{\rm in}\\
\widetilde U_{\rm in}\\
\widetilde B_{\rm in}
\end{pmatrix}(R),
&R>1.
\end{cases}
\label{eq:initial-data-self-similar}
\end{equation}
The exterior perturbation is assumed to be $C^1$ and sufficiently small in the appropriate weighted $C^1$ norm, as specified in Theorem~\ref{thm:main-self-similar}. In particular, we have $\Sigma(R,0)>0$ and $B(R,0)>0$ for $R>1$.

We observe that the traces at $R=1$ are constrained by the nonlinear Rankine--Hugoniot conditions. By the modulation choice~\eqref{eq:shock-freezing-speed}, the initial shock speed is
\begin{equation}
\ds(-1)=-\crrb\Cu(0).
\label{eq:initial-shock-speed}
\end{equation}
Specializing~\eqref{eq:finite-pressure-RH-traces} to $\tau=0$ and using~\eqref{eq:initial-interior-sound-speed} gives
\begin{subequations}
\label{eq:initial-finite-pressure-shock-traces}
\begin{align}
V(1^+,0)
&=-\tfrac{2\crrb}{\gamma+1}
+\tfrac{2\alpha^2\Sigma_0^2}{(\gamma+1)\crrb},
\label{eq:initial-finite-pressure-V-trace}\\
\alpha Q(1^+,0)
&=\tfrac{\crrb}{\gamma+1}
\left[
\gamma\left(2-\tfrac{(\gamma-1)\alpha^2\Sigma_0^2}
{\gamma\crrb^2}\right)
\left((\gamma-1)+\tfrac{2\alpha^2\Sigma_0^2}
{\crrb^2}\right)
\right]^{\frac{1}{2}},
\label{eq:initial-finite-pressure-Q-trace}\\
\rho(1^+,0)
&=\tfrac{\gamma+1}
{(\gamma-1)+\tfrac{2\alpha^2\Sigma_0^2}{\crrb^2}},
\label{eq:initial-finite-pressure-density-trace}\\
K(1^+,0)
&=\log\bigl(\alpha Q(1^+,0)\bigr)
-\alpha\log\rho(1^+,0).
\label{eq:initial-finite-pressure-K-trace}
\end{align}
\end{subequations}
These conditions determine the  trace of the exterior perturbation in~\eqref{eq:initial-data-self-similar}. When $p_0=0$, they reduce to the cold-gas Guderley traces~\eqref{eq:guderley:profile-traces}. These are the zeroth-order compatibility conditions at the shock. Since we carry out our analysis in $C^1$, we also require first-order compatibility conditions. That is, the traces $(\tZ,\tA,\tW)(1^+,0)$, together with $\tcuu(0)$, need to satisfy \eqref{eq:finite-pressure-RH-ring-system} at $\tau=0$. Explicitly, this means
\begin{subequations}
\label{eq:initial-differentiated-compatibility}
\begin{align}
\tcuu(0)
&=\kappa_\gamma\tZ(1^+,0)
+\tfrac{\crrb\sqrt{2\gamma(\gamma-1)}}
{(\gamma+1)\bigl(\gamma+\sqrt{2\gamma(\gamma-1)}\bigr)}
\bigl(\mathcal P_{p_0}\bigr)_1
\notag\\
&\quad
+\tfrac{\gamma(\gamma+1)}
{2\crrb\bigl(\gamma+\sqrt{2\gamma(\gamma-1)}\bigr)}
\left(\bigl(\mathcal P_{p_0}\bigr)_3
-\bigl(\mathcal P_{p_0}\bigr)_2\right),
\label{eq:initial-differentiated-compatibility-cu}\\
\tW(1^+,0)
&=b_W(\gamma)\tZ(1^+,0)
+\tfrac{2(\gamma+1)}
{\crrb\bigl(\sqrt{2\gamma(\gamma-1)}+\gamma-1\bigr)
\bigl(\gamma+\sqrt{2\gamma(\gamma-1)}\bigr)}
\notag\\
&\qquad\cdot\left(
\gamma\bigl(\mathcal P_{p_0}\bigr)_2
+\sqrt{2\gamma(\gamma-1)}\bigl(\mathcal P_{p_0}\bigr)_3
-\tfrac{2\crrb^2\sqrt{2\gamma(\gamma-1)}}{(\gamma+1)^2}
\bigl(\mathcal P_{p_0}\bigr)_1
\right),
\label{eq:initial-differentiated-compatibility-W}\\
\tA(1^+,0)
&=b_A(\gamma)\tZ(1^+,0)
+\tfrac{\crrb\bigl(2\gamma+\sqrt{2\gamma(\gamma-1)}\bigr)}
{(\gamma+1)\bigl(\gamma+\sqrt{2\gamma(\gamma-1)}\bigr)}
\bigl(\mathcal P_{p_0}\bigr)_1
\notag\\
&\quad
+\tfrac{\gamma(\gamma+1)}
{\crrb\sqrt{2\gamma(\gamma-1)}
\bigl(\gamma+\sqrt{2\gamma(\gamma-1)}\bigr)}
\left(\bigl(\mathcal P_{p_0}\bigr)_3
-\bigl(\mathcal P_{p_0}\bigr)_2\right).
\label{eq:initial-differentiated-compatibility-A}
\end{align}
\end{subequations}

\begin{remark}
The initial shock speed is the unique free parameter in the choice of the (exterior) undifferentiated traces of the initial data at the shock. After fixing the interior state $(\rho,u^r,p)=(1,0,p_0)$ and the initial shock speed $\ds(-1)$, the inverted Rankine--Hugoniot conditions~\eqref{intro:RH:re} then determine the three exterior physical traces. For $p_0>0$, these relations also show that we may choose as a free parameter any of the exterior traces $(\usp,\rsp,\ssp)$ at $t=-1$. The choice of shock speed also fixes $\Cu(0)$ through~\eqref{eq:initial-shock-speed}. Once this normalization is chosen, the self-similar traces are determined by~\eqref{eq:initial-finite-pressure-shock-traces}. Thus, the one-dimensional freedom to choose the undifferentiated traces of the initial data is incorporated into $\Cu(0)$.
\end{remark}

\section{Properties of the Guderley profile used in the stability analysis}
\label{sec:guderley-profile-properties}

In this section we record the three profile inequalities used in the stability analysis. We start by analyzing the linearization of~\eqref{eq:ring-perturbation-system}.

Under the modulation choice~\eqref{eq:cr:0}, the linearization around the Guderley profile can be written as 
\begin{equation}
\partial_\tau\Ytil+\overline{\Tmat}R\partial_R\Ytil
=\mathcal L_{\mathrm{bg}}\Ytil
+\tcuu\Ybar
+\widetilde V\mathcal F_V
+\alpha\widetilde Q\mathcal F_Q.
\label{eq:profile-linearized-system}
\end{equation}
Here $\overline{\Tmat}$ is the background value of the transport matrix \eqref{eq:ring-transport-matrix}, and the local linear operator is
\begin{equation}
\mathcal L_{\mathrm{bg}}Y'
:=-\overline{\Dmat}Y'
-\Ncal(\Ybar,Y')
-\Ncal(Y',\Ybar),
\label{eq:profile-background-linearization}
\end{equation}
with
\begin{equation*}
\overline{\Dmat}:=I+\Vbar D^{(v)}+\alpha\Qbar D^{(q)}.
\end{equation*}
In particular, we write its entries as
\begin{equation*}
\mathcal L_{\mathrm{bg}}
=
\begin{pmatrix}
\ell_{ZZ}&\ell_{ZA}&\ell_{ZW}\\
\ell_{AZ}&\ell_{AA}&\ell_{AW}\\
\ell_{WZ}&\ell_{WA}&\ell_{WW}
\end{pmatrix}.
\end{equation*}
For $d=3$, direct substitution of~\eqref{eq:ring-Dv}, \eqref{eq:ring-Dq}, and~\eqref{eq:ring-quadratic-form} into \eqref{eq:profile-background-linearization} gives
\begin{subequations}
\label{eq:profile-background-coefficients}
\begin{align}
\ell_{ZZ}
&=-1+\tfrac{1}{2}\overline{\rA}
+\tfrac{5\alpha}{2}\Qbar
-\tfrac{3(1+\alpha)}{2}\Vbar
-(1+\alpha)\overline{\rZ}
+\tfrac{\alpha-1}{2}\overline{\rW},
\label{eq:profile-background-coefficient-ZZ}\\
\ell_{ZA}=\ell_{WA}
&=\tfrac{1}{2}\bigl(\overline{\rZ}-\overline{\rW}\bigr)
-\tfrac{4\alpha}{\gamma}\Qbar,
\label{eq:profile-background-coefficient-ZA-WA}\\
\ell_{ZW}
&=-\tfrac{1}{2}\overline{\rA}
+\tfrac{\alpha(2\alpha+5)}{2\gamma}\Qbar
+\tfrac{3\alpha-1}{2}\Vbar
+\tfrac{\alpha-1}{2}\overline{\rZ},
\label{eq:profile-background-coefficient-ZW}\\
\ell_{AZ}=\ell_{AW}
&=-\tfrac{1+\alpha}{2}\overline{\rA}
-\tfrac{\alpha}{2\gamma}\Qbar,
\label{eq:profile-background-coefficient-AZ-AW}\\
\ell_{AA}
&=-1-(1+3\alpha)\Vbar
-\tfrac{1+\alpha}{2}
\bigl(\overline{\rZ}+\overline{\rW}\bigr),
\label{eq:profile-background-coefficient-AA}\\
\ell_{WZ}
&=\tfrac{1}{2}\overline{\rA}
-\tfrac{\alpha(2\alpha+5)}{2\gamma}\Qbar
+\tfrac{3\alpha-1}{2}\Vbar
+\tfrac{\alpha-1}{2}\overline{\rW},
\label{eq:profile-background-coefficient-WZ}\\
\ell_{WW}
&=-1-\tfrac{1}{2}\overline{\rA}
-\tfrac{5\alpha}{2}\Qbar
-\tfrac{3(1+\alpha)}{2}\Vbar
+\tfrac{\alpha-1}{2}\overline{\rZ}
-(1+\alpha)\overline{\rW}.
\label{eq:profile-background-coefficient-WW}
\end{align}
\end{subequations}
The terms $\mathcal{F}_V$ and $\mathcal{F}_Q$ are given by
\begin{equation}\label{eq:profile-forcing-vectors}\mathcal F_V:=-\left(R\partial_R\Ybar+D^{(v)}\Ybar\right),\qquad\mathcal F_Q:=-\left(\diag(-1,0,1)R\partial_R\Ybar+D^{(q)}\Ybar\right).\end{equation}

At the linear level, if we neglect the terms depending on $p_0$, the system \eqref{eq:profile-linearized-system} is coupled at $R=1$ to the boundary conditions obtained from~\eqref{eq:finite-pressure-RH-ring-system}:
\begin{subequations}
\label{eq:profile-shock-feedback}
\begin{align}
\tA(1^+,\tau)
&=b_A(\gamma)\tZ(1^+,\tau),
\label{eq:profile-shock-feedback-A}\\
\tW(1^+,\tau)
&=b_W(\gamma)\tZ(1^+,\tau),
\label{eq:profile-shock-feedback-W}\\
\tcuu(\tau)
&=\kappa_\gamma\tZ(1^+,\tau).
\label{eq:profile-shock-feedback-cu}
\end{align}
\end{subequations}
The coefficients $b_A(\gamma)$, $b_W(\gamma)$, and $\kappa_\gamma$ are given by~\eqref{eq:finite-pressure-RH-feedback-coefficients}. In the same approximation, $\widetilde V(1^+,\tau)=\widetilde Q(1^+,\tau)=0$. In particular $\tilde{V}$ and $\tilde{Q}$ are uniquely recovered from $\Ytil$ by integrating the first two identities in~\eqref{eq:ring-inverse} from $R=1$.

For $i\in\{Z,A,W\}$, let
\begin{equation}
\mathcal J_i
:=\tfrac{1}{2}\left|(\mathcal F_V)_i
-\alpha(\mathcal F_Q)_i\right|
+(1-\alpha)\left|(\mathcal F_Q)_i\right|
+\tfrac{11}{20}\left|(\mathcal F_V)_i
+\alpha(\mathcal F_Q)_i\right|.
\label{eq:profile-joint-Volterra-coefficient}
\end{equation}

\begin{lemma}\label{lem:profile-signs} 
Fix $d=3$ and $1<\gamma\leq\frac53$. Consider the corresponding Guderley profile $(\bar V, \bar Q, \bar K)$ and the associated coefficients defined in~\eqref{eq:profile-background-coefficients}. Then, for $1 \le R \le R_s$ we have
\begin{equation*}
 \ell_{ZA}=\ell_{WA}<0,\quad \ell_{ZW}>0,\quad \ell_{AZ}=\ell_{AW}>0,\quad \ell_{WZ}<0,\quad \overline{\rZ}>0,\quad \overline{\rA}<0,\quad \overline{\rW}>0.\end{equation*}
\end{lemma}

\begin{proof} We prove $\overline{\rA}<0$ and $\ell_{AZ}=\ell_{AW}>0$ by hand. By~\eqref{eq:guderley:exponents} and Theorem~\ref{lem:guderley-profile-monotonicity-decay}, we have $0<\crrb<1$, $\Vbar\geq-\frac{\crrb}{1+\alpha}$, and $\Qbar>0$. In particular, $\crrb+\Vbar\geq\frac{\alpha\crrb}{1+\alpha}>0$ and $1+\Vbar>0$. Equations~\eqref{eq:guderley:B-elimination}, \eqref{eq:guderley-VQK-background}, and~\eqref{eq:ring-def} imply
\begin{equation}\label{eq:profile-sign-A} R\partial_R\Kbar=-\tfrac{1+\Vbar}{\crrb+\Vbar},\qquad \overline{\rA}=-\tfrac{\alpha\Qbar(1+\Vbar)}{\gamma(\crrb+\Vbar)}<0.\end{equation}
Substituting~\eqref{eq:profile-sign-A} into~\eqref{eq:profile-background-coefficient-AZ-AW}, we obtain
\begin{equation*}
 \ell_{AZ}=\ell_{AW}=\tfrac{\alpha\Qbar}{2\gamma(\crrb+\Vbar)}\left((1+\alpha)(1+\Vbar)-(\crrb+\Vbar)\right)>0,\end{equation*}
because
\begin{equation*} (1+\alpha)(1+\Vbar)-(\crrb+\Vbar)\geq 1+\alpha-\tfrac{(1+2\alpha)\crrb}{1+\alpha}=\tfrac{\alpha^2+(1+2\alpha)(1-\crrb)}{1+\alpha}>0.\end{equation*}
Proposition~\ref{prop:ca-background-signs} in Appendix~\ref{app:three-profile-interval-proof} proves the other five signs using a computer-assisted argument. We refer the reader to Appendix~\ref{app:three-profile-computation} for details of the computer-assisted proof.
\end{proof}

Define now the three following quantities:
\begin{subequations}
\label{eq:profile-row-quantities}
\begin{align}
\mathcal R_A
&:=\tfrac{21}{10}\ell_{AZ}-\kappa_\gamma\overline{\rA}
+\log R\,\mathcal J_A,
\label{eq:profile-row-A}\\
\mathcal R_W
&:=-\ell_{WZ}-(1-\alpha)\ell_{WA}
+\kappa_\gamma\overline{\rW}
+\log R\,\mathcal J_W,
\label{eq:profile-row-W}\\
\Psi_Z
&:=\ell_{ZZ}-(1-\alpha)\ell_{ZA}+\tfrac{11}{10}\ell_{ZW}
+\kappa_\gamma\overline{\rZ}
+\log R\,\mathcal J_Z\notag\\ &=-1-2\Vbar
+\tfrac{(\gamma+8)(\gamma-1)}{2\gamma}\Qbar
+(\kappa_\gamma-2)\overline{\rZ}
+\tfrac{1}{10}\ell_{ZW}
+\log R\,\mathcal J_Z.
\label{eq:profile-row-Z}
\end{align}
\end{subequations}

For $1 \le R \le R_s$ we define the two functions $G_A(R)$ and $G_W(R)$ as solutions to the ODEs
\begin{subequations}
\label{eq:profile-gain-equations}
\begin{align}
\LA R\partial_R G_A
&=\ell_{AA}G_A+\mathcal R_A,
&
G_A(1)&=b_A(\gamma),
\label{eq:profile-gain-A}\\
\LW R\partial_R G_W
&=\ell_{WW}G_W+\mathcal R_W,
&
G_W(1)&=-b_W(\gamma) = | b_W(\gamma)|.
\label{eq:profile-gain-W}
\end{align}
\end{subequations}
Explicitly, these are given by
\begin{subequations}\label{eq:profile-gain-formulas}\begin{align} G_A(R)&=b_A(\gamma)e^{\int_1^R\frac{\ell_{AA}(\Xi)}{\Xi\LA(\Xi)}\,\dd\Xi}+\int_1^R\tfrac{\mathcal R_A(R')}{R'\LA(R')}e^{\int_{R'}^R\frac{\ell_{AA}(\Xi)}{\Xi\LA(\Xi)}\,\dd\Xi}\,\dd R',
\notag\\ G_W(R)&=-b_W(\gamma)e^{\int_1^R\frac{\ell_{WW}(\Xi)}{\Xi\LW(\Xi)}\,\dd\Xi}+\int_1^R\tfrac{\mathcal R_W(R')}{R'\LW(R')}e^{\int_{R'}^R\frac{\ell_{WW}(\Xi)}{\Xi\LW(\Xi)}\,\dd\Xi}\,\dd R'.\label{eq:profile-gain-W-formula}\end{align}\end{subequations}

The function $\Psi_Z$ arises in the differential inequality for the $\tZ$-component, while $G_A$ and $G_W$ arise by integrating the corresponding inequalities for $\tA$ and $\tW$ along the associated characteristics, starting from the shock. These quantities use constant weights for $\tA$ and $\tW$ by $\frac{1}{1-\alpha}$ and $\frac{10}{11}$, respectively, relative to $\tZ$, as we will use in the bootstrap bounds~\eqref{eq:nonlinear-bootstrap-exterior-assumption}.
\begin{lemma}\label{lem:three-profile-inequalities}
 Fix $d=3$ and $1<\gamma\leq\frac53$. Consider the corresponding Guderley profile $(\bar V,\bar Q,\bar K)$ and the functions $G_A$, $G_W$, and $\Psi_Z$ defined in~\eqref{eq:profile-gain-equations} and~\eqref{eq:profile-row-quantities}. Then
\begin{subequations}
\label{eq:three-profile-inequalities}
\begin{align}
\sup_{1\leq R\leq\Rs}G_A(R)&<1-\alpha,
\label{eq:profile-inequality-A}\\
\sup_{1\leq R\leq\Rs}G_W(R)&<\tfrac{11}{10},
\label{eq:profile-inequality-W}\\
\sup_{1\leq R\leq\Rs}\Psi_Z(R)&<0.
\label{eq:profile-inequality-Z}
\end{align}
\end{subequations}
\end{lemma}

\begin{proof} Lemma~\ref{lem:ca-common-stability} in Appendix~\ref{app:three-profile-interval-proof} proves the three inequalities using a computer-assisted argument.  We refer the reader to Section~\ref{subsec:ca-computational-proofs} of Appendix~\ref{app:three-profile-computation} for the details of the computer-assisted implementation.\end{proof}

\section{The nonlinear analysis}\label{sec:main-result}\label{sec:nonlinear-bootstrap}

We now carry out the proof of nonlinear stability. The precise form of Theorem~\ref{thm:intro-main-physical} is Corollary~\ref{thm:main-physical} below, which follows from the following theorem in modulated self-similar variables.

\begin{theorem}[Nonlinear stability in modulated self-similar variables]\label{thm:main-self-similar} Fix $d=3$, $1<\gamma\leq\frac{5}{3}$, and let $\alpha=\frac{\gamma-1}{2}$. Let $(\Vbar,\Qbar,\Kbar)$ be the corresponding Guderley profile in~\eqref{eq:guderley-VQK-background}, with similarity exponents $(\crrb,\crrb-1,\crrb-1)$. There exist constants $\varepsilon>0$, $\lambda>0$, $\mathsf{C}\geq1$, and a positive $C^1$ weight $\psi$ such that the following holds.

Consider initial data of the form~\eqref{eq:initial-data-self-similar}, with $p_0 \ge 0$, $\Cr(0)=1$, and $\Cu(0)=\Cb(0)>0$. Denote the exterior initial data in the variables~\eqref{eq:VQK-def} by $(V,Q,K)(\cdot,0)$, and define their perturbations and differentiated perturbations as in~\eqref{eq:VQK-perturbations} and~\eqref{eq:ring-perturbation}, respectively. Assume that:

\begin{enumerate}[label=(\roman*),leftmargin=2em] \item \emph{(Regularity).} The initial data are $C^1$ regular for $R\ge1$, with $Q(R,0)>0$ and $B(R,0)>0$ for $R>1$.

\item \emph{(Compatibility at the shock).} The exterior traces at $R=1^+$ satisfy the zeroth-order conditions~\eqref{eq:initial-finite-pressure-shock-traces} and the first-order conditions~\eqref{eq:initial-differentiated-compatibility}.

\begin{samepage}\item \emph{(Smallness).} The initial perturbations and pressure ratio satisfy
\begin{equation}
\sup_{R\geq1}R^{\frac{1}{\crrb}}\left(|\widetilde V(R,0)|+\alpha|\widetilde Q(R,0)|+\tfrac{|\widetilde B(R,0)|}{R}\right)\leq\varepsilon^2.
\label{eq:nonlinear-bootstrap-initial-far-field}
\end{equation}
\begin{subequations}
\label{eq:nonlinear-bootstrap-initial-ring-bounds}
\begin{equation}
\max\left\{\|\psi\tZ(0)\|_{L^\infty([1,\infty))},\tfrac{1}{1-\alpha}\|\psi\tA(0)\|_{L^\infty([1,\infty))},\tfrac{10}{11}\|\psi\tW(0)\|_{L^\infty([1,\infty))}\right\}\leq\varepsilon^2.
\label{eq:nonlinear-bootstrap-initial-Z-bound}
\end{equation}
\begin{equation}
\left\|R^{\frac{1}{\crrb}}\Ytil(0)\right\|_{L^\infty((1,\infty))}\leq\varepsilon^2,\qquad \tfrac{p_0}{\Cu(0)^2}\leq\varepsilon^2.
\label{eq:nonlinear-bootstrap-initial-sharp-ring-bound}
\end{equation}
\end{subequations}
\end{samepage}
\end{enumerate}

Under these assumptions, there exists a unique classical solution $(V,Q,K)$ of~\eqref{eq:VQK-system} for all $R\ge1$ and $\tau\geq0$, with the prescribed initial data, that satisfies the Rankine--Hugoniot jump conditions~\eqref{eq:finite-pressure-RH-traces} at $R=1$, and the modulation choice~\eqref{eq:cr:0}. In particular, the shock remains fixed at $R=1$, and the following estimates hold for every $\tau\geq0$:

\begin{enumerate}[label=(\alph*),ref=\alph*,leftmargin=2em] \item\label{item:main-self-similar-decay} \emph{(Exponential decay of the perturbation).} The differentiated and undifferentiated perturbations satisfy
\begin{subequations}
\label{eq:main-self-similar-perturbation-decay}
\begin{align}
\max\left\{\|\psi\tZ(\tau)\|_{L^\infty([1,\infty))},\tfrac{1}{1-\alpha}\|\psi\tA(\tau)\|_{L^\infty([1,\infty))},\tfrac{10}{11}\|\psi\tW(\tau)\|_{L^\infty([1,\infty))}\right\}
&\leq \varepsilon e^{-\lambda\tau},
\label{eq:main-self-similar-ring-decay}\\
\sup_{R\geq1}R^{\frac{1}{2\crrb}}\left(|\widetilde V|+\alpha|\widetilde Q|+R|\partial_R\widetilde V|+\alpha R|\partial_R\widetilde Q|+\tfrac{|\widetilde B|+R|\partial_R\widetilde B|}{R}\right)(R,\tau)
&\leq \mathsf{C}\varepsilon e^{-\lambda\tau}.
\label{eq:main-self-similar-VQ-decay}
\end{align}
\end{subequations}

\item\label{item:main-self-similar-modulation} \emph{(Exponential decay of the modulation).} The modulation perturbations satisfy
\begin{equation}
|\tcuu(\tau)|+|\tcbb(\tau)|\leq \mathsf{C}\varepsilon e^{-\lambda\tau}.
\label{eq:main-self-similar-modulation-decay}
\end{equation}
In particular, $(\crr,\cuu,\cbb)$ converges exponentially to $(\crrb,\crrb-1,\crrb-1)$.

\item\label{item:main-self-similar-sharp-bounds} \emph{(Sharp spatial bounds and positivity).} Uniformly for $R\geq1$ and $\tau\geq0$,
\begin{subequations}
\label{eq:main-self-similar-sharp-bounds}
\begin{align}
\sup_{R\geq1}R^{\frac{1}{\crrb}}\left(|\Ytil|+|\widetilde V|+\alpha|\widetilde Q|\right)(R,\tau)
&\leq \mathsf{C}\varepsilon,
\label{eq:main-self-similar-sharp-VQ-bound}\\
(1-\mathsf{C}\varepsilon)\Qbar(R)\leq Q(R,\tau)
&\leq(1+\mathsf{C}\varepsilon)\Qbar(R),
\label{eq:main-self-similar-Q-comparison}\\
\|\widetilde K(\tau)\|_{L^\infty((1,\infty))}
&\leq \mathsf{C}\varepsilon.
\label{eq:main-self-similar-global-K-bound}
\end{align}
\end{subequations}

\item\label{item:main-self-similar-convergence} \emph{(Convergence of undifferentiated variables).} For every $0<\eta<\frac{1}{2\crrb}$ we have
\begin{equation*}
\sup_{R\geq1}R^{\frac{1}{\crrb}-\eta}
\left(|\widetilde V|+\alpha|\widetilde Q|+\tfrac{|\widetilde B|}{R}\right)(R,\tau)
\leq \mathsf{C}\varepsilon e^{-2\crrb\eta\lambda\tau}.
\end{equation*}
\end{enumerate}\end{theorem}

We express Theorem~\ref{thm:main-self-similar} in physical coordinates.

\begin{corollary}[Nonlinear stability in physical variables]
\label{thm:main-physical}
Fix $d=3$, $1<\gamma\leq\frac{5}{3}$, and the associated Guderley profile $(\bar U,\bar\Sigma,\bar B)$, with similarity exponent $\crrb$. There exists $\varepsilon>0$ such that the following hold. Consider radial initial data $(\rho,u^r,p)(\cdot,-1)$ for~\eqref{eq:euler:cl} with a single shock at $r=1$. Assume that:

\begin{enumerate}[label=(\roman*),leftmargin=2em]
\item \emph{(Regularity and interior state).} The data are $C^1$ on each side of the shock, with $\rho(r,-1)>0$ and $p(r,-1)>0$ for $r>1$, and
\begin{equation*}
(\rho,u^r,p)(r,-1)=(1,0,p_0),\qquad 0\leq r<1,
\end{equation*}
where $p_0\geq0$.

\item \emph{(Compatibility at the shock).} The exterior traces satisfy the Rankine--Hugoniot conditions with initial shock speed $\ds(-1)<0$. More precisely, with the normalization
\begin{equation*}
\Cr(0)=1,\qquad \Cu(0)=\Cb(0)=-\tfrac{\ds(-1)}{\crrb},
\end{equation*}
the corresponding self-similar traces satisfy~\eqref{eq:initial-finite-pressure-shock-traces} and~\eqref{eq:initial-differentiated-compatibility}.

\item \emph{(Weighted perturbation bounds).} Expressed in the variables~\eqref{eq:modulated-ansatz} and~\eqref{eq:VQK-def} at $\tau=0$, the exterior data satisfy the undifferentiated weighted bound~\eqref{eq:nonlinear-bootstrap-initial-far-field}, the differentiated bounds~\eqref{eq:nonlinear-bootstrap-initial-ring-bounds}, and the pressure smallness~\eqref{eq:nonlinear-bootstrap-initial-sharp-ring-bound}.

\item \emph{(Physical velocity smallness).} The initial velocity satisfies
\begin{equation*}
\|u^r(\cdot,-1)-\bar u(\cdot,-1)\|_{L^\infty((1,\infty))}\leq\varepsilon^2.
\end{equation*}
\end{enumerate}

Under these assumptions, there exist a finite time $t_\ast>-1$ and a unique radial classical shock solution of~\eqref{eq:euler:cl} on $-1\leq t<t_\ast$, with the prescribed initial data. The following conclusions hold:

\begin{enumerate}[label=(\alph*),ref=\alph*,leftmargin=2em]

\item\label{item:main-physical-focusing} \emph{(Shock focusing).} As $t\uparrow t_\ast$, the shock radius $\ss(t)$ converges to the origin with the following rates
\begin{equation}
\s(t)\simeq(t_\ast-t)^{\crrb},\qquad -\ds(t)\simeq(t_\ast-t)^{\crrb-1}.
\label{eq:main-physical-focusing-rates}
\end{equation}
Moreover,
\begin{equation}
|t_\ast|\leq \mathsf{C}\varepsilon.
\label{eq:main-physical-focusing-time}
\end{equation}

\item \label{item:main:rh} \emph{(Rankine--Hugoniot jump conditions and Lax inequality)} For all $-1\le t < t_*$, the solution satisfies the Rankine--Hugoniot jump condition~\eqref{intro:RH} and Lax inequalities~\eqref{intro:lax} along the shock $\ss(t)$.

\item\label{item:main-physical-blowup} \emph{(Blowup rates and density bounds).} The velocity and sound speed satisfy
\begin{equation}
\|u^r(\cdot,t)\|_{L^\infty((0,\infty))}\simeq(t_\ast-t)^{\crrb-1},\qquad
\|\sigma(\cdot,t)\|_{L^\infty((0,\infty))}\simeq(t_\ast-t)^{\crrb-1}
\label{eq:main-physical-blowup-rates}
\end{equation}
as $t\uparrow t_\ast$, while the density remains uniformly bounded.

\item\label{item:main-physical-interior} \emph{(Interior state).} Ahead of the shock, 
\begin{equation}
(\rho,u^r,p)(r,t)=(1,0,p_0),\qquad 0\leq r<\s(t),\quad -1\leq t<t_\ast.
\label{eq:main-physical-interior-state}
\end{equation}

\item\label{item:main-physical-convergence} \emph{(Asymptotic self-similarity).} In self-similar coordinates (normalized by the shock coordinate), we have
\begin{equation}
\tfrac{t_\ast-t}{\s(t)}(u^r,\sigma,b)\bigl(\s(t)R,t\bigr)
-(\bar U,\bar\Sigma,\bar B)(R)\longrightarrow0
\label{eq:main-physical-profile-convergence}
\end{equation}
as $t\uparrow t_\ast$, locally uniformly for $R\in[0,\infty)\setminus\{1\}$.
\end{enumerate}
The constants in these estimates depend only on $\gamma$ and the fixed profiles.
\end{corollary}

The rest of this section proves Theorem~\ref{thm:main-self-similar} by a nonlinear bootstrap argument, using the three strict profile inequalities~\eqref{eq:three-profile-inequalities} from Lemma~\ref{lem:three-profile-inequalities}. Throughout the proof we use the modulation choice~\eqref{eq:cr:0}. 

\subsection{The perturbation equations}
We start by writing the perturbation equations satisfied by $\Ytil$, defined in~\eqref{eq:ring-perturbation}. Subtracting the stationary equation for $\Ybar$ from~\eqref{eq:ring-evolution} gives
\begin{equation}
\partial_\tau\Ytil+\Tmat R\partial_R\Ytil=\mathcal L_{\mathrm{bg}}\Ytil+\tcuu\Ybar+\widetilde V\mathcal F_V+\alpha\widetilde Q\mathcal F_Q-\bigl(\bigl(-\tcuu I+\widetilde V D^{(v)}+\alpha\widetilde QD^{(q)}\bigr)\Ytil+\Ncal(\Ytil,\Ytil)\bigr), \label{eq:nonlinear-bootstrap-perturbation-system}
\end{equation}
where we recall that the transport matrix $\Tmat$ is defined in~\eqref{eq:ring-transport-matrix}, the constant matrices $D^{(v)}$ and $D^{(q)}$ in~\eqref{eq:ring-constant-matrices}, and the bilinear form $\Ncal$ in~\eqref{eq:ring-quadratic-form}, the linear operator $\mathcal L_{\mathrm{bg}}$ in~\eqref{eq:profile-background-linearization}, and the forcing vectors $\mathcal F_V$ and $\mathcal F_Q$ in~\eqref{eq:profile-forcing-vectors}. This PDE is coupled at $R=1$ to the differentiated Rankine--Hugoniot boundary conditions
\begin{equation*}
\mathsf R_{\mathrm{RH}}
\begin{pmatrix}
\tZ\\[2pt]
\tA\\[2pt]
\tW
\end{pmatrix}(1^+,\tau)
=
\begin{pmatrix}
\tfrac{\gamma+1}{\crrb(\gamma-1)}\\[4pt]
\tfrac{\crrb\sqrt{2\gamma(\gamma-1)}}{\alpha(\gamma+1)}\\[4pt]
-\tfrac{2\crrb}{\gamma+1}
\end{pmatrix}
\tcuu(\tau)
+\mathcal P_{p_0}\bigl(\tau,\tcuu(\tau),\Ytil(1^+,\tau)\bigr),
\end{equation*}
where $\mathcal P_{p_0}$ is defined in~\eqref{eq:finite-pressure-RH-remainder-components}. The undifferentiated Rankine--Hugoniot conditions~\eqref{eq:finite-pressure-RH-tilded-traces} are
\begin{equation*}
\begin{aligned}
\widetilde V(1^+,\tau)&=\tfrac{2\gamma p_0}{(\gamma+1)\crrb\Cu(\tau)^2},\\ \widetilde Q(1^+,\tau)&=\Qbar(1^+)\left(\sqrt{\left(1-\tfrac{\alpha p_0}{\crrb^2\Cu(\tau)^2}\right)\left(1+\tfrac{\gamma p_0}{\alpha\crrb^2\Cu(\tau)^2}\right)}-1\right),\\ \widetilde K(1^+,\tau)&=\tfrac{1}{2}\log\left(1-\tfrac{\alpha p_0}{\crrb^2\Cu(\tau)^2}\right)+\tfrac{\gamma}{2}\log\left(1+\tfrac{\gamma p_0}{\alpha\crrb^2\Cu(\tau)^2}\right).
\end{aligned}
\end{equation*}

\subsection{Bootstrap assumptions}

We will use several constants and parameters throughout the proof, with the following order of dependence:
\begin{equation*}
R_2\longrightarrow R_1\longrightarrow\lambda\longrightarrow\mathsf c\longrightarrow\mathsf{M}\longrightarrow R_3\longrightarrow\mathsf{C}\longrightarrow\varepsilon.
\end{equation*}
Each parameter may depend on $\gamma$, the profiles, and the preceding constants. Schematically, $\lambda$ and $\varepsilon$ are chosen small enough; $\mathsf M$,  $\mathsf C$, $R_3$ are large constants; $ \mathsf{c}$ is chosen close to but smaller than $1$; $R_1$ and $R_2$ are chosen sufficiently close to the sonic point $R_s$.

By continuity across the sonic point, we start by fixing
\begin{equation*}
\Rs<R_1<R_2
\end{equation*}
and constants $0<\mathsf c<1$ and $0<\lambda \ll 1$ such that\footnote{We extend $G_A$ and $G_W$ beyond $R_s$ using their defining ODEs.}
\begin{equation}
\sup_{1\leq R\leq R_2}G_A(R)
<\mathsf c(1-\alpha),
\qquad
\sup_{1\leq R\leq R_2}G_W(R)<\tfrac{11}{10}\mathsf c,
\qquad
\sup_{1\leq R\leq R_2}\Psi_Z(R)<-4\lambda,
\label{eq:nonlinear-bootstrap-extended-profile-bounds}
\end{equation}
where we take $R_2$ close enough to $\Rs$ that the signs in Lemma~\ref{lem:profile-signs} also hold throughout $1\leq R\leq R_2$. We will also require
\begin{equation}
\lambda<1-\crrb.
\label{eq:nonlinear-bootstrap-lambda-choice}
\end{equation}
We now choose a positive $C^1$ weight $\psi$, with $\psi(R_2)=1$ and $\psi=\psi(R_1)$ on $[1,R_1]$. On $[R_1,R_2]$, we require
\begin{equation}
R\partial_R\log\psi(R)\leq0,
\qquad
\mathsf c\,\psi(R_1)<\tfrac{1+\mathsf c}{2}.
\label{eq:nonlinear-bootstrap-exterior-weight}
\end{equation}
We then introduce $R_3 >R_2$ (to be chosen later). On $[R_2,R_3]$, we choose $\psi$ sufficiently decreasing such that
\begin{equation}
\min_{1\leq i\leq3}\left\{(\overline{\Dmat})_{ii}-(\overline{\Tmat})_{ii}R\partial_R\log\psi\right\}
 \geq\tfrac12+2\mathsf{M}(1+\log R)R^{-\frac{1}{\crrb}},
\qquad R_2\leq R\leq R_3.
\label{eq:nonlinear-bootstrap-weight-damping}
\end{equation}
For $R\geq R_3$, we require
\begin{equation}
\begin{aligned}
0\leq R\partial_R\log\psi(R)&\leq\tfrac{3}{4\crrb}\quad\text{for }R\geq R_3,\\
R\partial_R\log\psi(R)&=\tfrac{1}{2\crrb},
\qquad
\psi(R)\simeq\psi(R_3)\left(\tfrac{R}{R_3}\right)^{\frac{1}{2\crrb}}
\quad\text{for }R\geq2R_3.
\label{eq:nonlinear-bootstrap-exterior-weight-derivative}
\end{aligned}
\end{equation}
In~\eqref{eq:nonlinear-bootstrap-weight-damping}, $\mathsf{M}$ is chosen large enough to bound the sum of the off-diagonal terms of $\overline{\Dmat}$, the mixed bilinear terms involving $\Ybar$, and the reconstruction and modulation terms in~\eqref{eq:ring-perturbation-system}. The weight choice is possible because the profile speeds are uniformly positive for $R\geq R_2$, by~\eqref{eq:nonlinear-bootstrap-background-outgoing-speeds}.  We choose $R_3$ large enough that $R_3>2R_2$ and~\eqref{eq:nonlinear-bootstrap-weight-damping} holds with $R\partial_R\log\psi=0$ on $[R_3/2,R_3]$. This is possible since $\overline{\Dmat}\to I$ as $R\to+\infty$. This allows us to connect $\psi$ to the far-field behavior as in~\eqref{eq:nonlinear-bootstrap-exterior-weight-derivative}. We also observe that
\begin{equation}
\int_1^\infty\tfrac{1}{\psi(R)}\tfrac{\dd R}{R}<\infty.
\label{eq:nonlinear-bootstrap-integrable-weight}
\end{equation}
To recapitulate, the weight $\psi$ is chosen as follows:
\begin{itemize}
\item It is constant, equal to $\psi(R_1)$, on $[1,R_1]$.
\item It is \emph{slightly decreasing} on $[R_1,R_2]$, so that $\psi(R_1)$ remains close to $1$, as quantified in~\eqref{eq:nonlinear-bootstrap-exterior-weight}, while matching the prescribed negative derivative at $R_2$.
\item It is nonincreasing on $[R_2,R_3]$ and satisfies~\eqref{eq:nonlinear-bootstrap-weight-damping}.
\item On $[R_3,\infty)$, where the dynamics are governed by the decay of the profile coefficients, it grows like $R^{\frac{1}{2\crrb}}$ at infinity.
\end{itemize}

\subsubsection{Bootstrap assumption}
Assume a classical solution of~\eqref{eq:VQK-system}, coupled to the shock conditions at $R=1$, exists for $0\leq\tau\leq T$.

We assume the following global weighted bootstrap on the differentiated variables:
\begin{equation}
\max\left\{\|\psi\tZ(\tau)\|_{L^\infty([1,\infty))},\tfrac{1}{1-\alpha}\|\psi\tA(\tau)\|_{L^\infty([1,\infty))},\tfrac{10}{11}\|\psi\tW(\tau)\|_{L^\infty([1,\infty))}\right\}
\leq\varepsilon e^{-\lambda\tau},
\label{eq:nonlinear-bootstrap-exterior-assumption}
\end{equation}
where $\varepsilon\ll1$ is chosen sufficiently small. Since $\psi(R)\ge 1$ on $[1,R_2]$, in particular we obtain the unweighted bounds
\begin{equation*}
\max\left\{\|\tZ(\tau)\|_{L^\infty([1,R_2])},\tfrac{1}{1-\alpha}\|\tA(\tau)\|_{L^\infty([1,R_2])},\tfrac{10}{11}\|\tW(\tau)\|_{L^\infty([1,R_2])}\right\}
\leq\varepsilon e^{-\lambda\tau}.
\end{equation*}
Finally, we assume the modulation bootstrap
\begin{equation}
|\tcuu(\tau)|+|\tcbb(\tau)|
\leq \varepsilon^{\frac{2}{3}}e^{-\lambda\tau},
\label{eq:nonlinear-bootstrap-modulation-assumption}
\end{equation}
where $\tcbb:=\cbb-(\crrb-1)=\tcuu$ by~\eqref{eq:fixed-radial-equal-amplitude-consequences}.

We also recall the qualitative assumptions on the initial data of Theorem~\ref{thm:main-self-similar}. The data have the form~\eqref{eq:initial-data-self-similar}, with $p_0\geq0$, $\Cr(0)=1$, and $\Cu(0)=\Cb(0)>0$ is defined by~\eqref{eq:cr:0} and~\eqref{eq:shock-freezing-speed} at $\tau=0$. In the exterior region, the data are $C^1$, with $Q(R,0)>0$ and $B(R,0)>0$, and their shock traces satisfy the zeroth-order compatibility conditions~\eqref{eq:initial-finite-pressure-shock-traces} and the first-order compatibility conditions~\eqref{eq:initial-differentiated-compatibility}.

The initial weighted bound~\eqref{eq:nonlinear-bootstrap-initial-ring-bounds} gives
\begin{equation*}
\max\left\{\|\psi\tZ(0)\|_{L^\infty([1,\infty))},\tfrac{1}{1-\alpha}\|\psi\tA(0)\|_{L^\infty([1,\infty))},\tfrac{10}{11}\|\psi\tW(0)\|_{L^\infty([1,\infty))}\right\}\leq\varepsilon^2.
\end{equation*}
The bound~\eqref{eq:nonlinear-bootstrap-initial-sharp-ring-bound} also gives $|\tZ(R,0)|,|\tA(R,0)|,|\tW(R,0)|\leq\varepsilon^2$ on $[1,R_2]$.
We also recall the assumption~\eqref{eq:nonlinear-bootstrap-initial-sharp-ring-bound} on internal pressure
\begin{equation*}
\tfrac{p_0}{\Cu(0)^2}\leq\varepsilon^2.
\end{equation*}

\subsection{Improving the bootstrap.}
We are now ready to prove Theorem~\ref{thm:main-self-similar}. We will use a standard bootstrap argument. We assume the bounds~\eqref{eq:nonlinear-bootstrap-exterior-assumption}--\eqref{eq:nonlinear-bootstrap-modulation-assumption}, and prove the improved version of such bounds:
\begin{subequations}
\label{eq:nonlinear-bootstrap-improvement-targets}
\begin{align}
\max\left\{\|\psi\tZ(\tau)\|_{L^\infty([1,\infty))},\tfrac{1}{1-\alpha}\|\psi\tA(\tau)\|_{L^\infty([1,\infty))},\tfrac{10}{11}\|\psi\tW(\tau)\|_{L^\infty([1,\infty))}\right\}
&\leq\tfrac{1+\mathsf c}{2}\varepsilon e^{-\lambda\tau},
\label{eq:nonlinear-bootstrap-exterior-target}\\
|\tcuu(\tau)|+|\tcbb(\tau)|
&\leq\tfrac12\varepsilon^{\frac{2}{3}}e^{-\lambda\tau}.
\label{eq:nonlinear-bootstrap-modulation-target}
\end{align}
\end{subequations}
 The local well-posedness of the system~\eqref{eq:VQK-system} with the shock conditions~\eqref{eq:finite-pressure-RH-traces} was proved by Majda in~\cite{Maj83exist,Maj83stab} in Sobolev spaces; in the radial case local well-posedness and continuity of the norms in $L^\infty$-based spaces follow from standard 1d theory, see for instance Li~\cite{Li80} or Dafermos~\cite{Daf10}. The standard one-dimensional local well-posedness theory can be easily adapted to the weighted norms used here by using finite speed of propagation and the decay of the initial data, and applying the estimates on compact sets.

\subsubsection{Improving the bootstrap for the modulation functions}
\label{subsec:bootstrap-modulation}
We first use the bootstrap on the modulation functions~\eqref{eq:nonlinear-bootstrap-modulation-assumption} to prove exponential growth of $C_u$. We recall the choice $\lambda<1-\crrb$ in~\eqref{eq:nonlinear-bootstrap-lambda-choice}. Since $\partial_\tau\Cu=-\cuu\Cu$ and $\cuu=\crrb-1+\tcuu$, we have
\begin{equation}
\Cu(\tau)=\Cu(0)\exp\left( (1-\crrb)\tau-\int_0^\tau\tcuu(s)\,\dd s \right)\geq e^{-\frac{\varepsilon^{\frac23}}{\lambda}} \Cu(0)e^{(1-\crrb)\tau}. \label{eq:nonlinear-bootstrap-Cu-growth}
\end{equation}
Thus $\Cu$ grows exponentially. In particular, by~\eqref{eq:nonlinear-bootstrap-initial-sharp-ring-bound} and~\eqref{eq:nonlinear-bootstrap-lambda-choice}, taking $\varepsilon^{\frac23}\leq\frac{\lambda\log 2}{2}$ gives
\begin{equation}
\tfrac{p_0}{\Cu(\tau)^2}\leq e^{\frac{2\varepsilon^{\frac23}}{\lambda}}\varepsilon^2e^{-2(1-\crrb)\tau}\leq2\varepsilon^2e^{-2\lambda\tau}.
\label{eq:nonlinear-bootstrap-pressure-decay}
\end{equation}
Recalling~\eqref{eq:finite-pressure-RH-solved-feedback-cu}, we have
\begin{equation}
\tcuu(\tau)=\kappa_\gamma\tZ(1^+,\tau) +\Ocal\left( \tfrac{p_0}{\Cu(\tau)^2} \left(1+|\tZ(1^+,\tau)|\right) \right), \label{eq:nonlinear-bootstrap-perturbed-shock-feedback}
\end{equation}
and using the bootstrap bound~\eqref{eq:nonlinear-bootstrap-exterior-assumption} and the pressure decay~\eqref{eq:nonlinear-bootstrap-pressure-decay}, we obtain
\begin{equation}
|\tcuu(\tau)|\leq \kappa_\gamma\varepsilon e^{-\lambda\tau} +\varepsilon^2e^{-2\lambda\tau} \left(\mathsf{C}+\mathsf{C}\varepsilon e^{-\lambda\tau}\right)\leq\mathsf{C}\varepsilon e^{-\lambda\tau}, \label{eq:nonlinear-bootstrap-cu-improvement}
\end{equation}
for $\varepsilon$ sufficiently small. Since $\tcbb=\tcuu$, we therefore obtain
\begin{equation}
|\tcuu(\tau)|+|\tcbb(\tau)| \leq\mathsf{C}\varepsilon e^{-\lambda\tau}\leq\tfrac12\varepsilon^{\frac23}e^{-\lambda\tau}, \label{eq:nonlinear-bootstrap-pressure-and-cu-control}
\end{equation}
after taking $\mathsf{C}\varepsilon^{\frac13}\leq\frac12$. This strictly improves~\eqref{eq:nonlinear-bootstrap-modulation-assumption} and proves~\eqref{eq:nonlinear-bootstrap-modulation-target}.

\subsubsection{Reconstruction bounds for the nonlinear coefficients}
We record the reconstruction needed in the differentiated equations. Using~\eqref{eq:nonlinear-bootstrap-pressure-decay} and the Rankine--Hugoniot conditions~\eqref{eq:finite-pressure-RH-tilded-traces} at $R=1$, we obtain
\begin{equation}
|\widetilde V(1^+,\tau)|
+\alpha|\widetilde Q(1^+,\tau)|
+|\widetilde K(1^+,\tau)|
\le \mathsf{C} \varepsilon^2e^{-2\lambda\tau}.
\label{eq:nonlinear-bootstrap-undifferentiated-shock-traces}
\end{equation}
Integrating~\eqref{eq:ring-inverse} from the shock gives, for $R\geq1$,
\begin{subequations}
\label{eq:nonlinear-bootstrap-undifferentiated-reconstruction}
\begin{align}
\widetilde V(R,\tau)
&=\widetilde V(1^+,\tau)
+\tfrac{1}{2}\int_1^R(\tZ+\tW)(S,\tau)\tfrac{\dd S}{S},
\label{eq:nonlinear-bootstrap-V-reconstruction}\\
\alpha\widetilde Q(R,\tau)
&=\alpha\widetilde Q(1^+,\tau)
+\int_1^R
\left(-\tfrac{\alpha}{2}\tZ+\tA
+\tfrac{\alpha}{2}\tW\right)(S,\tau)
\tfrac{\dd S}{S}.
\label{eq:nonlinear-bootstrap-Q-reconstruction}
\end{align}
\end{subequations}
In particular, the reconstruction on $[1,R_2]$ has only a logarithmic factor. The shock-trace bound~\eqref{eq:finite-pressure-RH-trace-bound} and the preceding formulas give
\begin{subequations}
\label{eq:nonlinear-bootstrap-undifferentiated-bounds}
\begin{gather}
|\widetilde V(R,\tau)|+\alpha|\widetilde Q(R,\tau)|
\leq \mathsf{C}\tfrac{p_0}{\Cu(\tau)^2}
+\mathsf{C}\log R\,\|\Ytil(\tau)\|_{L^\infty([1,R_2])},
\qquad 1\leq R\leq R_2,
\label{eq:nonlinear-bootstrap-inner-VQ-bound}\\
\|\widetilde V(\tau)\|_{L^\infty([1,\infty))}
+\alpha\|\widetilde Q(\tau)\|_{L^\infty([1,\infty))}
\leq\mathsf{C}\varepsilon e^{-\lambda\tau}.
\label{eq:nonlinear-bootstrap-global-VQ-bound}
\end{gather}
\end{subequations}
For~\eqref{eq:nonlinear-bootstrap-global-VQ-bound}, we also use the pressure bound~\eqref{eq:nonlinear-bootstrap-pressure-decay}, the bootstrap~\eqref{eq:nonlinear-bootstrap-exterior-assumption}, and the integrability~\eqref{eq:nonlinear-bootstrap-integrable-weight}; the constant may depend on the fixed weight. Moreover, $\inf_{R\geq1}\psi(R)=\psi(R_3)>0$.

Substituting~\eqref{eq:nonlinear-bootstrap-exterior-assumption}, \eqref{eq:nonlinear-bootstrap-pressure-decay}, and~\eqref{eq:nonlinear-bootstrap-cu-improvement} into these reconstruction bounds gives the estimate for the nonlinear terms in~\eqref{eq:nonlinear-bootstrap-perturbation-system}:
\begin{equation}
\psi\left|
\left[-\tcuu I+\widetilde V D^{(v)}
+\alpha\widetilde QD^{(q)}\right]\Ytil
+\Ncal(\Ytil,\Ytil)
\right|
\leq \mathsf{C}\varepsilon^2e^{-2\lambda\tau},
\qquad R\geq1.
\label{eq:nonlinear-bootstrap-quadratic-bounds}
\end{equation}

\subsection{Improving the bootstrap for differentiated variables on $[1, R_2]$}

Given $R\geq1$ and $\tau\geq0$, let $s\mapsto\Phi_i(R,\tau;s)$, for $i\in\{A,W,Z\}$, denote the self-similar $i$-characteristic through $(R,\tau)$.  These characteristics solve
\begin{subequations}
\label{eq:nonlinear-bootstrap-inner-characteristics}
\begin{align}
\tfrac{\dd}{\dd s}\Phi_A(R,\tau;s)
&=\bigl(\crrb+V(\Phi_A(R,\tau;s),s)\bigr)\Phi_A(R,\tau;s),
&
\Phi_A(R,\tau;\tau)&=R,
\label{eq:nonlinear-bootstrap-inner-A-characteristic}\\
\tfrac{\dd}{\dd s}\Phi_W(R,\tau;s)
&=\bigl(\crrb+V(\Phi_W(R,\tau;s),s)
+\alpha Q(\Phi_W(R,\tau;s),s)\bigr)\Phi_W(R,\tau;s),
&
\Phi_W(R,\tau;\tau)&=R,
\label{eq:nonlinear-bootstrap-inner-W-characteristic}\\
\tfrac{\dd}{\dd s}\Phi_Z(R,\tau;s)
&=\bigl(\crrb+V(\Phi_Z(R,\tau;s),s)
-\alpha Q(\Phi_Z(R,\tau;s),s)\bigr)\Phi_Z(R,\tau;s),
&
\Phi_Z(R,\tau;\tau)&=R.
\label{eq:nonlinear-bootstrap-inner-Z-characteristic}
\end{align}
\end{subequations}
We first justify the dichotomy used below.  The differentiated bootstrap~\eqref{eq:nonlinear-bootstrap-exterior-assumption}, the shock-trace estimate~\eqref{eq:nonlinear-bootstrap-undifferentiated-shock-traces}, and the reconstruction formulas~\eqref{eq:nonlinear-bootstrap-undifferentiated-reconstruction} give~\eqref{eq:nonlinear-bootstrap-global-VQ-bound}.  By Theorem~\ref{lem:guderley-profile-monotonicity-decay}, the profile $\Vbar$ is increasing and $\Qbar>0$ on $[1,\infty)$.  Thus the background value of~\eqref{eq:ring-transport-matrix} and the shock value~\eqref{eq:guderley:VQ-shock-data} give
\begin{equation*}
\LW(R)\geq\LA(R)=\crrb+\Vbar(R)\geq\crrb+\Vbar(1^+)=\tfrac{\alpha\crrb}{1+\alpha}>0,\qquad R\geq1.
\end{equation*}

Consequently, after taking $\varepsilon$ sufficiently small, for every $R\geq1$ and $0\leq s\leq T$ we obtain
\begin{subequations}
\label{eq:nonlinear-bootstrap-positive-inner-AW-speeds}
\begin{align}
\crrb+V(R,s)
&=\LA(R)+\widetilde V(R,s)
\geq\tfrac{\alpha\crrb}{1+\alpha}-\mathsf{C}\varepsilon^{\frac23}e^{-\lambda s}\geq\tfrac{\alpha\crrb}{2(1+\alpha)}>0,
\label{eq:nonlinear-bootstrap-positive-inner-A-speed}\\
\crrb+V(R,s)+\alpha Q(R,s)
&=\LW(R)+\widetilde V(R,s)+\alpha\widetilde Q(R,s)
\geq\tfrac{\alpha\crrb}{2(1+\alpha)}>0.
\notag
\end{align}
\end{subequations}

It follows from~\eqref{eq:nonlinear-bootstrap-inner-characteristics} that $s\mapsto\Phi_A(R,\tau;s)$ and $s\mapsto\Phi_W(R,\tau;s)$ are strictly increasing.  Thus, when followed backward from $(R,\tau)$, each characteristic has exactly two possibilities: it meets the shock $R=1$ at some $\tau'\in[0,\tau]$, or it remains strictly larger than $1$ and reaches $s=0$. 

\subsubsection{The $\rA$-component}
We start by writing the $\rA$-row of~\eqref{eq:nonlinear-bootstrap-perturbation-system}; using~\eqref{eq:ring-constant-matrices} and~\eqref{eq:ring-quadratic-form}, this is
\begin{subequations}
\begin{align}
\left(\partial_\tau+(\crrb+V)R\partial_R\right)\tA
={}&\ell_{AA}\tA+\ell_{AZ}\tZ+\ell_{AW}\tW+\tcuu\overline{\rA}+\widetilde V(\mathcal F_V)_A+\alpha\widetilde Q(\mathcal F_Q)_A\notag\\
&+\left(\tcuu-(1+\alpha d)\widetilde V\right)\tA-\left(\tfrac{1+\alpha}{2}\tA+\tfrac{\alpha}{2\gamma}\widetilde Q\right)(\tZ+\tW).
\label{eq:nonlinear-bootstrap-A-equation}
\end{align}
By using~\eqref{eq:nonlinear-bootstrap-undifferentiated-reconstruction}, together with the bounds~\eqref{eq:finite-pressure-RH-trace-bound} and~\eqref{eq:nonlinear-bootstrap-pressure-decay}, we obtain
\begin{equation*}
\begin{aligned}
\widetilde V(\mathcal F_V)_A+\alpha\widetilde Q(\mathcal F_Q)_A
={}&\int_1^R\Bigl[\tfrac{1}{2}\bigl((\mathcal F_V)_A-\alpha(\mathcal F_Q)_A\bigr)\tZ(S,\tau)
+(\mathcal F_Q)_A\tA(S,\tau)\\
&\qquad\qquad+\tfrac{1}{2}\bigl((\mathcal F_V)_A+\alpha(\mathcal F_Q)_A\bigr)\tW(S,\tau)\Bigr]\tfrac{\dd S}{S}
+\Ocal(\varepsilon^2e^{-2\lambda\tau}).
\end{aligned}
\end{equation*}
Then, using~\eqref{eq:nonlinear-bootstrap-exterior-assumption} and $\psi\geq1$ on $[1,R_2]$, we get
\begin{equation*}
\begin{aligned}
&\left|\widetilde V(\mathcal F_V)_A+\alpha\widetilde Q(\mathcal F_Q)_A\right|\\
&\quad\leq\log R\left[\tfrac{1}{2}\left|(\mathcal F_V)_A-\alpha(\mathcal F_Q)_A\right|+(1-\alpha)\left|(\mathcal F_Q)_A\right|+\tfrac{11}{20}\left|(\mathcal F_V)_A+\alpha(\mathcal F_Q)_A\right|\right]\varepsilon e^{-\lambda\tau}
+\mathsf{C}\varepsilon^2e^{-2\lambda\tau}\\
&\quad=\log R\,\mathcal J_A(R)\varepsilon e^{-\lambda\tau}+\mathsf{C}\varepsilon^2e^{-2\lambda\tau},
\end{aligned}
\end{equation*}
where the last equality is the definition of $\mathcal J_A$ in~\eqref{eq:profile-joint-Volterra-coefficient}.  By Lemma~\ref{lem:profile-signs} and the choice of $R_2$ following~\eqref{eq:nonlinear-bootstrap-extended-profile-bounds}, we have $\ell_{AZ}=\ell_{AW}>0$ and $\overline{\rA}<0$ throughout $[1,R_2]$.  Applying the same bootstrap bounds to the off-diagonal terms, using~\eqref{eq:nonlinear-bootstrap-cu-improvement} for the modulation term and~\eqref{eq:nonlinear-bootstrap-quadratic-bounds} for the quadratic terms, gives
\begin{align}
\left|\left(\partial_\tau+(\crrb+V)R\partial_R\right)\tA-\ell_{AA}\tA\right|
&\leq\left(\ell_{AZ}+\tfrac{11}{10}\ell_{AW}-\kappa_\gamma\overline{\rA}+\log R\,\mathcal J_A\right)\varepsilon e^{-\lambda\tau}+\mathsf{C}\varepsilon^2e^{-\lambda\tau}\notag\\
&=\mathcal R_A(R)\varepsilon e^{-\lambda\tau}+\mathsf{C}\varepsilon^2e^{-\lambda\tau},
\label{eq:nonlinear-bootstrap-A-source-bound}
\end{align}
\end{subequations}
 where we used the definition of $\mathcal R_A$ from~\eqref{eq:profile-row-A}.  We now show how to use the bound~\eqref{eq:nonlinear-bootstrap-A-source-bound}, the equation for $\rA$ just derived~\eqref{eq:nonlinear-bootstrap-A-equation} and the Rankine--Hugoniot jump condition~\eqref{eq:finite-pressure-RH-solved-feedback-A} in order to improve the $\rA$ bootstrap inequality~\eqref{eq:nonlinear-bootstrap-exterior-assumption}. We will proceed by integrating along the associated characteristic equation $\Phi_A$; since as we observed in~\eqref{eq:nonlinear-bootstrap-inner-A-characteristic} and~\eqref{eq:nonlinear-bootstrap-positive-inner-A-speed} we have $\partial_s\Phi_A>0$, if we follow the associated characteristic for a given $(R,\tau)$ backward in time, there are two possible cases
 \begin{enumerate}
     \item there exists $\tau'\in[0,\tau]$ such that $\Phi_A(1,\tau';\tau)=R$,
     \item there exists $R'>1$ such that $\Phi_A(R',0;\tau)=R$.
 \end{enumerate}
We analyze each case separately.  

\smallskip \noindent\emph{Case 1: the characteristic meets the shock.}\par Assume that there exists $\tau'\in[0,\tau]$ such that
\begin{equation*}
\Phi_A(R,\tau;\tau')=1,\qquad\text{or equivalently}\qquad\Phi_A(1,\tau';\tau)=R.
\end{equation*}
By uniqueness of the characteristic equation, $\Phi_A(R,\tau;s)=\Phi_A(1,\tau';s)$ for $\tau'\leq s\leq\tau$.  We therefore integrate $s\mapsto\Phi_A(1,\tau';s)$ forward from the shock at time $\tau'$ to $(R,\tau)$.  The Rankine--Hugoniot condition  \eqref{eq:finite-pressure-RH-solved-feedback-A}, the bound~\eqref{eq:nonlinear-bootstrap-pressure-decay}, and the bootstrap bound for $\tZ$ \eqref{eq:nonlinear-bootstrap-exterior-assumption} give
\begin{equation*}
e^{\lambda\tau'} |\tA(1^+,\tau')| \leq \varepsilon\left(b_A(\gamma)+\mathsf{C}\varepsilon\right).
\end{equation*}
Integrating~\eqref{eq:nonlinear-bootstrap-A-equation} along the $\Phi_A$ characteristic we get
\begin{equation}
\begin{aligned}
&e^{\lambda\tau}|\tA(R,\tau)| \leq \varepsilon\left(b_A(\gamma)+\mathsf{C}\varepsilon\right)e^{\int_{\tau'}^\tau\left(\ell_{AA}(\Phi_A(1,\tau';s))+\lambda\right)\,\dd s}\\
&\qquad+\varepsilon\int_{\tau'}^\tau e^{\int_s^\tau\left(\ell_{AA}(\Phi_A(1,\tau';r))+\lambda\right)\,\dd r}\left(\mathcal R_A(\Phi_A(1,\tau';s))+\mathsf{C}\varepsilon\right)\,\dd s.
\label{eq:nonlinear-bootstrap-A-shock-case}
\end{aligned}
\end{equation}
To identify $G_A(R)$ in this expression, we will execute a change of variables.  First we observe that the characteristic equation~\eqref{eq:nonlinear-bootstrap-inner-A-characteristic} and the lower bound~\eqref{eq:nonlinear-bootstrap-positive-inner-A-speed} yield
\begin{equation*}
0\leq\tau-\tau'\leq\tfrac{2(1+\alpha)}{\alpha\crrb}\log R\leq\tfrac{2(1+\alpha)}{\alpha\crrb}\log R_2.
\end{equation*}
Thus the time interval of integration is bounded independently of $\tau$ and $\tau'$.  Moreover, $\LA\geq\frac{\alpha\crrb}{1+\alpha}$ and~\eqref{eq:nonlinear-bootstrap-inner-VQ-bound} imply
\begin{subequations}
\begin{align}
\left\|\tfrac{1}{\LA+\widetilde V}-\tfrac{1}{\LA}\right\|_{L^\infty([1,R_2])}
&=\left\|\tfrac{\widetilde V}{\LA(\LA+\widetilde V)}\right\|_{L^\infty([1,R_2])}\leq \mathsf{C}\varepsilon,
\label{eq:nonlinear-bootstrap-A-inverse-speed-error}
\end{align}
\end{subequations}
In particular, the estimate~\eqref{eq:nonlinear-bootstrap-A-inverse-speed-error} and the change of variables $\Theta=\Phi_A(1,\tau';r)$ (observe that $\Phi_A(1, \tau'; r)$ is increasing in $r$) give, for every $s\in[\tau',\tau]$,
\begin{equation*}
\begin{aligned}
&\left|\int_s^\tau\bigl(\ell_{AA}(\Phi_A(1,\tau';r))+\lambda\bigr)\,\dd r-\int_{\Phi_A(1,\tau';s)}^R\tfrac{\ell_{AA}(\Theta)}{\Theta\LA(\Theta)}\,\dd\Theta\right|\leq \mathsf{C}(\lambda+\varepsilon)\log\tfrac{R}{\Phi_A(1,\tau';s)}\leq \mathsf{C}(\lambda+\varepsilon),\\
&\left|\int_{\tau'}^\tau e^{\int_s^\tau(\ell_{AA}(\Phi_A(1,\tau';r))+\lambda)\,\dd r}\mathcal R_A(\Phi_A(1,\tau';s))\,\dd s-\int_1^R\tfrac{\mathcal R_A(\Theta)}{\Theta\LA(\Theta)}e^{\int_\Theta^R\frac{\ell_{AA}(\Xi)}{\Xi\LA(\Xi)}\,\dd\Xi}\,\dd\Theta\right|\leq \mathsf{C}(\lambda+\varepsilon).
\end{aligned}
\end{equation*}
Consequently, from~\eqref{eq:nonlinear-bootstrap-A-shock-case} we obtain
\begin{equation*}
\begin{aligned}
e^{\lambda\tau}\!|\tA(R,\tau)|&\leq \varepsilon\left[b_A(\gamma)e^{\int_1^R\frac{\ell_{AA}(\Xi)}{\Xi\LA(\Xi)}\,\dd\Xi}\!+\!\int_1^R\tfrac{\mathcal R_A(S)}{S\LA(S)}e^{\int_S^R\frac{\ell_{AA}(\Xi)}{\Xi\LA(\Xi)}\,\dd\Xi}\,\dd S\!+\!\mathsf{C}(\lambda\!+\!\varepsilon)\right]\\
&=\varepsilon\left(G_A(R)\!+\!\mathsf{C}(\lambda\!+\!\varepsilon)\right).
\end{aligned}
\end{equation*}
Then, applying the bound~\eqref{eq:nonlinear-bootstrap-extended-profile-bounds} and choosing $\lambda$ and $\varepsilon$ small enough we obtain
\begin{equation*}
|\tA(R,\tau)| \leq \left(G_A(R)+\mathsf{C}\left(\lambda+\varepsilon\right)\right) \varepsilon e^{-\lambda\tau} \le \mathsf{c} (1-\alpha)\varepsilon e^{-\lambda\tau}.
\end{equation*}

\smallskip \noindent\emph{Case 2: the characteristic does not meet the shock.}\par Assume that there is no $\tau'\in[0,\tau]$ such that $\Phi_A(R,\tau;\tau')=1$.  Then $R'=\Phi_A(R,\tau;0)>1$ and $\Phi_A(R',0;\tau)=R$.  We integrate $s\mapsto\Phi_A(R',0;s)$ forward from $(R',0)$ to $(R,\tau)$. Then, the argument proceeds exactly as in the previous case. The only difference is that the initial data here are given at time $\tau=0$ and satisfy the bound~\eqref{eq:nonlinear-bootstrap-initial-Z-bound}. Repeating the argument verbatim, we obtain 
\begin{equation*}
\begin{aligned}
e^{\lambda\tau}|\tA(R,\tau)|
&\leq\varepsilon^2e^{\int_{R'}^R\frac{\ell_{AA}(\Xi)}{\Xi\LA(\Xi)}\,\dd\Xi}+\varepsilon\int_{R'}^R\tfrac{\mathcal R_A(S)}{S\LA(S)}e^{\int_S^R\frac{\ell_{AA}(\Xi)}{\Xi\LA(\Xi)}\,\dd\Xi}\,\dd S+\mathsf{C}\varepsilon(\lambda+\varepsilon)\notag\\
&\leq\varepsilon\left(G_A(R)+\mathsf{C}(\lambda+\varepsilon)\right).
\end{aligned}
\end{equation*}
Then, applying the bound~\eqref{eq:nonlinear-bootstrap-extended-profile-bounds} and choosing $\lambda$ and $\varepsilon$ small enough we obtain
\begin{equation}
|\tA(R,\tau)|\leq\left(G_A(R)+\mathsf{C}(\lambda+\varepsilon)\right)\varepsilon e^{-\lambda\tau}\leq\mathsf c(1-\alpha)\varepsilon e^{-\lambda\tau}.
\label{eq:nonlinear-bootstrap-A-improvement-preliminary}
\end{equation}

\subsubsection{The $\rW$-component}

Improving the bootstrap for $\rW$ follows the same strategy as for $\rA$. We start by writing the equations obeyed by $\rW$; taking the $\rW$-row of~\eqref{eq:nonlinear-bootstrap-perturbation-system} and using~\eqref{eq:ring-constant-matrices} and~\eqref{eq:ring-quadratic-form}, we obtain
\begin{subequations}
\begin{align}
\left(\partial_\tau+(\crrb+V+\alpha Q)R\partial_R\right)\tW
={}&\ell_{WW}\tW+\ell_{WZ}\tZ+\ell_{WA}\tA+\tcuu\overline{\rW}+\widetilde V(\mathcal F_V)_W+\alpha\widetilde Q(\mathcal F_Q)_W\notag\\
&+\left(\tcuu-\tfrac{3+\alpha d}{2}\widetilde V-\tfrac{\alpha(d+2)}{2}\widetilde Q\right)\tW-\left(\tfrac{1-\alpha d}{2}\widetilde V+\tfrac{\alpha((d+2)+2\alpha(d-2))}{2\gamma}\widetilde Q\right)\tZ\notag\\
&-\tfrac{4\alpha}{\gamma}\widetilde Q\tA-\left(\tfrac{1-\alpha}{2}\tZ+\tfrac{1+\alpha}{2}\tW\right)\tW+\tfrac{1}{2}\tA(\tZ-\tW).
\label{eq:nonlinear-bootstrap-W-equation}
\end{align}
By using~\eqref{eq:nonlinear-bootstrap-undifferentiated-reconstruction}, together with the bounds~\eqref{eq:finite-pressure-RH-trace-bound} and~\eqref{eq:nonlinear-bootstrap-pressure-decay}, we obtain
\begin{equation*}
\begin{aligned}
\widetilde V(\mathcal F_V)_W+\alpha\widetilde Q(\mathcal F_Q)_W
={}&\int_1^R\Bigl[\tfrac{1}{2}\bigl((\mathcal F_V)_W-\alpha(\mathcal F_Q)_W\bigr)\tZ(S,\tau)
+(\mathcal F_Q)_W\tA(S,\tau)\\
&\qquad\qquad+\tfrac{1}{2}\bigl((\mathcal F_V)_W+\alpha(\mathcal F_Q)_W\bigr)\tW(S,\tau)\Bigr]\tfrac{\dd S}{S}
+\Ocal(\varepsilon^2e^{-2\lambda\tau}).
\end{aligned}
\end{equation*}
Then, using~\eqref{eq:nonlinear-bootstrap-exterior-assumption} and $\psi\geq1$ on $[1,R_2]$, we get
\begin{equation*}
\begin{aligned}
&\left|\widetilde V(\mathcal F_V)_W+\alpha\widetilde Q(\mathcal F_Q)_W\right|\\
&\quad\leq\log R\left[\tfrac{1}{2}\left|(\mathcal F_V)_W-\alpha(\mathcal F_Q)_W\right|+(1-\alpha)\left|(\mathcal F_Q)_W\right|+\tfrac{11}{20}\left|(\mathcal F_V)_W+\alpha(\mathcal F_Q)_W\right|\right]\varepsilon e^{-\lambda\tau}\\
&\qquad+\mathsf{C}\varepsilon^2e^{-2\lambda\tau}\\
&\quad=\log R\,\mathcal J_W(R)\varepsilon e^{-\lambda\tau}+\mathsf{C}\varepsilon^2e^{-2\lambda\tau},
\end{aligned}
\end{equation*}
where the last equality is the definition of $\mathcal J_W$ in~\eqref{eq:profile-joint-Volterra-coefficient}. By Lemma~\ref{lem:profile-signs} and the choice of $R_2$ following~\eqref{eq:nonlinear-bootstrap-extended-profile-bounds}, we have $\ell_{WZ}<0$, $\ell_{WA}<0$, and $\overline{\rW}>0$ throughout $[1,R_2]$. Applying the same bootstrap bounds to the off-diagonal terms, using~\eqref{eq:nonlinear-bootstrap-cu-improvement} for the modulation term and~\eqref{eq:nonlinear-bootstrap-quadratic-bounds} for the quadratic terms, gives
\begin{align}
&\left|\left(\partial_\tau+(\crrb+V+\alpha Q)R\partial_R\right)\tW-\ell_{WW}\tW\right|
\leq\left(-\ell_{WZ}-(1-\alpha)\ell_{WA}+\kappa_\gamma\overline{\rW}+\log R\,\mathcal J_W\right)\varepsilon e^{-\lambda\tau}\notag\\
&\qquad+\mathsf{C}\varepsilon^2e^{-\lambda\tau}=\mathcal R_W(R)\varepsilon e^{-\lambda\tau}+\mathsf{C}\varepsilon^2e^{-\lambda\tau},
\label{eq:nonlinear-bootstrap-W-source-bound}
\end{align}
\end{subequations}
where we used the definition of $\mathcal R_W$ from~\eqref{eq:profile-row-W}. Here Lemma~\ref{lem:profile-signs} gives $\mathcal R_W>0$, and $-b_W(\gamma)>0$, so $G_W>0$ on $[1,R_2]$. We again distinguish two cases.

\smallskip \noindent\emph{Case 1: the characteristic meets the shock.}\par Assume that $\Phi_W(1,\tau';\tau)=R$ for some $\tau'\in[0,\tau]$. The argument is the same as in Case 1 for $\rA$: we integrate~\eqref{eq:nonlinear-bootstrap-W-equation} using~\eqref{eq:nonlinear-bootstrap-W-source-bound}, the boundary condition~\eqref{eq:finite-pressure-RH-solved-feedback-W}, and~\eqref{eq:nonlinear-bootstrap-pressure-decay}, and then make the same change of variables. Recalling~\eqref{eq:profile-gain-W-formula}, we obtain
\begin{equation*}
\begin{aligned}
e^{\lambda\tau}|\tW(R,\tau)|
&\leq\varepsilon\left[-b_W(\gamma)e^{\int_1^R\frac{\ell_{WW}(\Xi)}{\Xi\LW(\Xi)}\,\dd\Xi}+\int_1^R\tfrac{\mathcal R_W(S)}{S\LW(S)}e^{\int_S^R\frac{\ell_{WW}(\Xi)}{\Xi\LW(\Xi)}\,\dd\Xi}\,\dd S+\mathsf{C}(\lambda+\varepsilon)\right]\notag\\
&=\varepsilon\left(G_W(R)+\mathsf{C}(\lambda+\varepsilon)\right).
\end{aligned}
\end{equation*}
For $\lambda$ and $\varepsilon$ sufficiently small,~\eqref{eq:nonlinear-bootstrap-extended-profile-bounds} then gives
\begin{equation*}
|\tW(R,\tau)|\leq\left(G_W(R)+\mathsf{C}(\lambda+\varepsilon)\right)\varepsilon e^{-\lambda\tau}\leq\tfrac{11}{10}\mathsf c\varepsilon e^{-\lambda\tau}.
\end{equation*}

\smallskip \noindent\emph{Case 2: the characteristic does not meet the shock.}\par If no such $\tau'$ exists, let $R'=\Phi_W(R,\tau;0)>1$, so that $\Phi_W(R',0;\tau)=R$. We use the same argument as in Case 2 for $\rA$, with the initial bound~\eqref{eq:nonlinear-bootstrap-initial-sharp-ring-bound}. Since $\mathcal R_W>0$ and $-b_W(\gamma)>0$, the comparison with~\eqref{eq:profile-gain-W-formula} yields
\begin{equation*}
\begin{aligned}
e^{\lambda\tau}|\tW(R,\tau)|
&\leq\varepsilon^2e^{\int_{R'}^R\frac{\ell_{WW}(\Xi)}{\Xi\LW(\Xi)}\,\dd\Xi}+\varepsilon\int_{R'}^R\tfrac{\mathcal R_W(S)}{S\LW(S)}e^{\int_S^R\frac{\ell_{WW}(\Xi)}{\Xi\LW(\Xi)}\,\dd\Xi}\,\dd S+\mathsf{C}\varepsilon(\lambda+\varepsilon)\notag\\
&\leq\varepsilon\left(G_W(R)+\mathsf{C}(\lambda+\varepsilon)\right).
\end{aligned}
\end{equation*}
Using~\eqref{eq:nonlinear-bootstrap-extended-profile-bounds} and the same choice of $\lambda$ and $\varepsilon$, we conclude
\begin{equation}
|\tW(R,\tau)|\leq\left(G_W(R)+\mathsf{C}(\lambda+\varepsilon)\right)\varepsilon e^{-\lambda\tau}\leq\tfrac{11}{10}\mathsf c\varepsilon e^{-\lambda\tau}.
\label{eq:nonlinear-bootstrap-W-improvement-preliminary}
\end{equation}

\subsubsection{The $\rZ$-component}

The $Z$-characteristic has a different geometry compared to the previous two characteristics.  By~\eqref{eq:guderley:first-characteristic-monotonicity}, \eqref{eq:guderley:unique-sonic-point}, and $R_2>\Rs$, one has $\LZ(1^+)<0<\LZ(R_2)$.  Since~\eqref{eq:nonlinear-bootstrap-inner-VQ-bound} gives
\begin{equation*}
\left|
\bigl(\crrb+V(R,s)-\alpha Q(R,s)\bigr)-\LZ(R)
\right|
\leq \mathsf{C}\varepsilon e^{-\lambda s},
\qquad 1\leq R\leq R_2,
\end{equation*}
by taking $\varepsilon$ small enough we have
\begin{equation}
\crrb+V(1^+,s)-\alpha Q(1^+,s)<0,
\qquad
\crrb+V(R_2,s)-\alpha Q(R_2,s)>0.
\label{eq:nonlinear-bootstrap-Z-boundary-directions}
\end{equation}
Thus the characteristic $\Phi_Z$ is outgoing both at the shock at $R=1$ and at $R=R_2$.

Consequently, for every $R\in[1,R_2]$ and $0<\tau\leq T$, the backward characteristic remains in $[1,R_2]$ until time $0$: this implies there exists $R'\in(1,R_2)$ such that
\begin{equation}
R'=\Phi_Z(R,\tau;0) \text{ or equivalently } \Phi_Z(R',0;\tau)=R.
\label{eq:nonlinear-bootstrap-Z-initial-characteristic}
\end{equation}
We therefore integrate $s\mapsto\Phi_Z(R',0;s)$ forward from $(R',0)$ to $(R,\tau)$.

We first write the $\rZ$-row of~\eqref{eq:nonlinear-bootstrap-perturbation-system}. Using~\eqref{eq:ring-constant-matrices} and~\eqref{eq:ring-quadratic-form}, we obtain
\begin{align}
\left(\partial_\tau+(\crrb+V-\alpha Q)R\partial_R\right)\tZ
={}&\ell_{ZZ}\tZ+\ell_{ZA}\tA+\ell_{ZW}\tW+\tcuu\overline{\rZ}+\widetilde V(\mathcal F_V)_Z+\alpha\widetilde Q(\mathcal F_Q)_Z\notag\\
&+\left(\tcuu-\tfrac{3+\alpha d}{2}\widetilde V+\tfrac{\alpha(d+2)}{2}\widetilde Q\right)\tZ-\left(\tfrac{1-\alpha d}{2}\widetilde V-\tfrac{\alpha((d+2)+2\alpha(d-2))}{2\gamma}\widetilde Q\right)\tW\notag\\
&-\tfrac{4\alpha}{\gamma}\widetilde Q\tA-\left(\tfrac{1+\alpha}{2}\tZ+\tfrac{1-\alpha}{2}\tW\right)\tZ+\tfrac{1}{2}\tA(\tZ-\tW).
\label{eq:nonlinear-bootstrap-Z-equation}
\end{align}
To integrate~\eqref{eq:nonlinear-bootstrap-Z-equation}, we use the bootstrap bounds~\eqref{eq:nonlinear-bootstrap-exterior-assumption} and the definition of $\Psi_Z$ in~\eqref{eq:profile-row-Z}. The reconstruction formulas~\eqref{eq:nonlinear-bootstrap-undifferentiated-reconstruction}, the modulation identity~\eqref{eq:nonlinear-bootstrap-perturbed-shock-feedback}, the bounds~\eqref{eq:nonlinear-bootstrap-pressure-decay} and~\eqref{eq:nonlinear-bootstrap-quadratic-bounds}, together with Lemma~\ref{lem:profile-signs}, give
\begin{equation}
\left(\partial_\tau+(\crrb+V-\alpha Q)R\partial_R\right)\left(e^{\lambda\tau}|\tZ(R,\tau)|\right)\leq(\ell_{ZZ}+\lambda)e^{\lambda\tau}|\tZ(R,\tau)|+(\Psi_Z-\ell_{ZZ})\varepsilon+\mathsf{C}\varepsilon^2e^{-\lambda\tau}.
\label{eq:nonlinear-bootstrap-Z-first-contact}
\end{equation}
We have $\ell_{ZZ}<0$ and $\ell_{ZZ}+\lambda<-3\lambda$ by Lemma~\ref{lem:profile-signs}, \eqref{eq:profile-row-Z}, and~\eqref{eq:nonlinear-bootstrap-extended-profile-bounds}. Since $\sup_{[1,R_2]}\Psi_Z<-4\lambda$ and $\ell_{ZZ}$ is bounded on $[1,R_2]$, we choose $\mathsf c<1$ sufficiently close to $1$ and then $\varepsilon$ sufficiently small so that, for $R\in[1,R_2]$ and $\tau\geq0$,
\begin{equation*}
0\leq\Psi_Z-\ell_{ZZ}+\mathsf{C}\varepsilon e^{-\lambda\tau}\leq-\mathsf c(\ell_{ZZ}+\lambda).
\end{equation*}
We integrate~\eqref{eq:nonlinear-bootstrap-Z-first-contact} along the characteristic $\Phi_Z(R',0;s)$ from~\eqref{eq:nonlinear-bootstrap-Z-initial-characteristic}. Using the preceding bound and the bound on the initial data~\eqref{eq:nonlinear-bootstrap-initial-Z-bound}, we obtain
\begin{equation*}
e^{\lambda\tau}|\tZ(R,\tau)|\leq\varepsilon^2e^{\int_0^\tau(\ell_{ZZ}(\Phi_Z(R',0;s))+\lambda)\,\dd s}+\mathsf c\varepsilon\left(1-e^{\int_0^\tau(\ell_{ZZ}(\Phi_Z(R',0;s))+\lambda)\,\dd s}\right)\leq\mathsf c\varepsilon,
\end{equation*}
where the last inequality follows from $0<e^{\int_0^\tau(\ell_{ZZ}(\Phi_Z(R',0;s))+\lambda)\,\dd s}\leq1$ and $\varepsilon\leq\mathsf c$. Hence
\begin{equation}
\|\tZ(\tau)\|_{L^\infty([1,R_2])}\leq\mathsf c\varepsilon e^{-\lambda\tau}.
\label{eq:nonlinear-bootstrap-Z-improvement}
\end{equation}
Estimates~\eqref{eq:nonlinear-bootstrap-A-improvement-preliminary}, \eqref{eq:nonlinear-bootstrap-W-improvement-preliminary}, and~\eqref{eq:nonlinear-bootstrap-Z-improvement}, combined with $\psi\leq\psi(R_1)$ on $[1,R_2]$ and~\eqref{eq:nonlinear-bootstrap-exterior-weight}, give the improved weighted bootstrap bounds
\begin{equation*}
\max\left\{\|\psi\tZ(\tau)\|_{L^\infty([1,R_2])},\tfrac{1}{1-\alpha}\|\psi\tA(\tau)\|_{L^\infty([1,R_2])},\tfrac{10}{11}\|\psi\tW(\tau)\|_{L^\infty([1,R_2])}\right\}
<\tfrac{1+\mathsf c}{2}\varepsilon e^{-\lambda\tau}.
\end{equation*}

\subsection{Improving the bootstrap for the differentiated variables in the outgoing region $R>R_2$.}
We start by observing that the estimates~\eqref{eq:nonlinear-bootstrap-A-improvement-preliminary}, \eqref{eq:nonlinear-bootstrap-W-improvement-preliminary}, and~\eqref{eq:nonlinear-bootstrap-Z-improvement} give us the bound
\begin{equation*}
e^{\lambda\tau}\max\left\{|\tZ|,\tfrac{1}{1-\alpha}|\tA|,\tfrac{10}{11}|\tW|\right\}(R_2,\tau)\leq\mathsf c\varepsilon.
\end{equation*}
Together with the modulation estimate~\eqref{eq:nonlinear-bootstrap-pressure-and-cu-control}, these inner estimates also give
\begin{equation*}
\|\Ytil(\tau)\|_{L^\infty([1,R_2])}+|\tcuu(\tau)|+|\tcbb(\tau)|\leq\mathsf{C}\varepsilon e^{-\lambda\tau}.
\end{equation*}
Before showing the weighted decay in the region $R\geq R_2$, we remind the reader that we have a positive lower bound for all three characteristic speeds on $R \ge R_2$.  Moreover, from~\eqref{eq:guderley:first-characteristic-monotonicity} and~\eqref{eq:guderley:unique-sonic-point}, setting $\delta:=\frac{1}{2}\LZ(R_2)>0$ we have
\begin{equation}
\inf_{R\geq R_2}\min\{\LZ(R),\LA(R),\LW(R)\}=\LZ(R_2)=2\delta>0.
\label{eq:nonlinear-bootstrap-background-outgoing-speeds}
\end{equation}
By~\eqref{eq:nonlinear-bootstrap-global-VQ-bound}, taking $\varepsilon$ sufficiently small therefore gives
\begin{equation}
\crrb+V-\alpha Q\geq\delta,
\qquad
\crrb+V\geq\delta,
\qquad
\crrb+V+\alpha Q\geq\delta,
\qquad R\geq R_2.
\label{eq:nonlinear-bootstrap-nonlinear-outgoing-speeds}
\end{equation}
We now treat the $[R_2, R_3]$ region. 

\subsubsection{Improving the bootstrap on $[R_2, R_3]$.}
We multiply~\eqref{eq:ring-perturbation-system} by $\psi$ and use the bilinearity of $\Ncal$ to obtain
\begin{equation}
\begin{aligned}
&\partial_\tau(\psi\Ytil)+\Tmat R\partial_R(\psi\Ytil)
+\bigl(\Dmat-\Tmat R\partial_R\log\psi\bigr)(\psi\Ytil)\\
&\qquad=-\psi\widetilde{\Tmat} R\partial_R\Ybar-\psi\widetilde{\Dmat}\Ybar
-\Ncal(\Ybar,\psi\Ytil)-\Ncal(\psi\Ytil,\Ybar)-\Ncal(\Ytil,\psi\Ytil).
\label{eq:nonlinear-bootstrap-weighted-perturbation-system}
\end{aligned}
\end{equation}
On the compact interval $[R_2,R_3]$, we have the bound
\begin{equation*}
\begin{aligned}
\left|\mathsf{RHS}_{\eqref{eq:nonlinear-bootstrap-weighted-perturbation-system}}\right|
&\leq\mathsf{M}(1+\log R)R^{-\frac{1}{\crrb}}\varepsilon e^{-\lambda\tau}
+\mathsf{C}\varepsilon^2e^{-2\lambda\tau},
\qquad R_2\leq R\leq R_3.
\end{aligned}
\end{equation*}
Indeed, since $\psi$ is nonincreasing on $[1,R_3]$ and $\psi(R_2)=1$, we have $\psi(R)\leq1$ and $\psi(R)/\psi(S)\leq1$ for $1\leq S\leq R$ and $R_2\leq R\leq R_3$. The reconstruction formulas~\eqref{eq:nonlinear-bootstrap-undifferentiated-reconstruction}, the shock-trace bound~\eqref{eq:nonlinear-bootstrap-undifferentiated-shock-traces}, and the global bootstrap~\eqref{eq:nonlinear-bootstrap-exterior-assumption} therefore give
\begin{equation*}
\begin{aligned}
\psi(R)\bigl(|\widetilde V(R,\tau)|+\alpha|\widetilde Q(R,\tau)|\bigr)
&\leq\mathsf{C}\varepsilon^2e^{-2\lambda\tau}
+\mathsf{C}\varepsilon e^{-\lambda\tau}\int_1^R\tfrac{\psi(R)}{\psi(S)}\tfrac{\dd S}{S}\\
&\leq\mathsf{C}\varepsilon^2e^{-2\lambda\tau}
+\mathsf{C}\varepsilon(\log R)e^{-\lambda\tau}.
\end{aligned}
\end{equation*}
The modulation estimate~\eqref{eq:nonlinear-bootstrap-cu-improvement} also gives $\psi(R)|\tcuu(\tau)|\leq\mathsf{C}\varepsilon e^{-\lambda\tau}$, since $\psi(R)\leq1$ on $[R_2,R_3]$. Together with the boundedness of the profiles and bilinearity of $\Ncal$, the preceding bounds control the linear terms on the right-hand side of~\eqref{eq:nonlinear-bootstrap-weighted-perturbation-system}. The quadratic term is bounded by $\mathsf{C}\varepsilon^2e^{-2\lambda\tau}$ using~\eqref{eq:nonlinear-bootstrap-exterior-assumption} and $\inf\psi=\psi(R_3)>0$.

We now use the damping supplied by the weight. Since all three characteristic speeds are positive on $[R_2,R_3]$, choosing $\psi$ sufficiently decreasing as in~\eqref{eq:nonlinear-bootstrap-weight-damping}, and multiplying by $e^{\lambda\tau}$, gives
\begin{equation*}
\begin{aligned}
\min_{1\leq i\leq3}\left\{(\Dmat)_{ii}-(\Tmat)_{ii}R\partial_R\log\psi\right\}-\lambda
&\geq\tfrac12+2\mathsf{M}(1+\log R)R^{-\frac{1}{\crrb}}-\lambda-\mathsf{C}\varepsilon\\
&\geq\tfrac14+2\mathsf{M}(1+\log R)R^{-\frac{1}{\crrb}},
\qquad R_2\leq R\leq R_3.
\end{aligned}
\end{equation*}
The preceding source and damping bounds show that, for sufficiently small $\varepsilon$, the diagonal terms dominate the remaining terms. We therefore argue as in the proof of~\eqref{eq:nonlinear-bootstrap-Z-improvement}, integrating along the characteristics~\eqref{eq:nonlinear-bootstrap-inner-characteristics} from $\tau=0$ or $R=R_2$. This yields the improved bootstrap bound
\begin{equation}
\max\left\{\|\psi\tZ(\tau)\|_{L^\infty([R_2,R_3])},\tfrac{1}{1-\alpha}\|\psi\tA(\tau)\|_{L^\infty([R_2,R_3])},\tfrac{10}{11}\|\psi\tW(\tau)\|_{L^\infty([R_2,R_3])}\right\}
\leq\tfrac{1+\mathsf c}{2}\varepsilon e^{-\lambda\tau}.
\label{eq:nonlinear-bootstrap-compact-exterior-bound}
\end{equation}

\subsubsection{Improving the bootstrap in the far-field region $R \ge R_3$.}

For $R\geq R_3$, we use the same weighted equation~\eqref{eq:nonlinear-bootstrap-weighted-perturbation-system}, now with the damping supplied by the decay of the profiles. The decay in~\eqref{eq:guderley-profile-decay}, the stationary equations, and the definitions of the differentiated variables imply
\begin{equation}
\begin{aligned}
&\Ybar,\ R\partial_R\Ybar,\ \mathcal F_V,\ \mathcal F_Q
=\mathcal O\left(R^{-\frac{1}{\crrb}}\right),\\
&\overline{\Dmat}=I+\mathcal O\left(R^{-\frac{1}{\crrb}}\right),
\qquad
\overline{\Tmat}=\crrb I+\mathcal O\left(R^{-\frac{1}{\crrb}}\right).
\label{eq:nonlinear-bootstrap-far-field-coefficients}
\end{aligned}
\end{equation}
By~\eqref{eq:nonlinear-bootstrap-exterior-weight-derivative}, the diagonal entries of $\overline{\Dmat}-\overline{\Tmat}R\partial_R\log\psi$ converge to $\frac12$ at infinity. In particular, using the properties of $\psi$ in~\eqref{eq:nonlinear-bootstrap-exterior-weight-derivative} we can obtain the inequality
\begin{equation}
\min_{1\leq i\leq3}\left\{(\overline{\Dmat})_{ii}-(\overline{\Tmat})_{ii}R\partial_R\log\psi\right\}\geq\tfrac18,
\qquad R\geq R_3.
\label{eq:nonlinear-bootstrap-far-field-row-bounds}
\end{equation}
The reconstruction formulas~\eqref{eq:nonlinear-bootstrap-undifferentiated-reconstruction}, the shock-trace bound~\eqref{eq:nonlinear-bootstrap-undifferentiated-shock-traces}, and the modulation estimate~\eqref{eq:nonlinear-bootstrap-cu-improvement} give
\begin{equation*}
e^{\lambda\tau}\left(|\tcuu(\tau)|+|\widetilde V(R,\tau)|+\alpha|\widetilde Q(R,\tau)|\right)
\leq\mathsf{C}\varepsilon\left(1+\int_1^R\tfrac{1}{\psi(S)}\tfrac{\dd S}{S}\right).
\end{equation*}
By the properties of the weight~\eqref{eq:nonlinear-bootstrap-exterior-weight}--\eqref{eq:nonlinear-bootstrap-exterior-weight-derivative},
\begin{equation*}
\sup_{R\geq R_3}\psi(R)R^{-\frac{1}{\crrb}}\leq\mathsf{C}\psi(R_3)R_3^{-\frac{1}{\crrb}},
\qquad
\int_1^\infty\tfrac{1}{\psi(S)}\tfrac{\dd S}{S}\leq\tfrac{\log R_3+\mathsf{C}}{\psi(R_3)}.
\end{equation*}
Together with~\eqref{eq:nonlinear-bootstrap-far-field-coefficients} and the bootstrap~\eqref{eq:nonlinear-bootstrap-exterior-assumption}, this gives
\begin{equation*}
\sup_{R\geq R_3}e^{\lambda\tau}\left|\mathsf{RHS}_{\eqref{eq:nonlinear-bootstrap-weighted-perturbation-system}}\right|
\leq\mathsf{M}(1+\log R_3)R_3^{-\frac{1}{\crrb}}\varepsilon+\mathsf{C}\varepsilon^2.
\end{equation*}
In particular, if $R_3$ is large enough we can dominate $\mathsf{RHS}_{\eqref{eq:nonlinear-bootstrap-weighted-perturbation-system}}$ using the diagonal terms~\eqref{eq:nonlinear-bootstrap-far-field-row-bounds}. We repeat the characteristic argument on $[R_3,\infty)$, integrating from $\tau=0$ or $R=R_3$ and using~\eqref{eq:nonlinear-bootstrap-compact-exterior-bound} at $R=R_3$. Combining this estimate with the bounds already proved on $[1,R_3]$, we obtain
\begin{equation}
\max\left\{
\|\psi\tZ(\tau)\|_{L^\infty([1,\infty))},
\tfrac{1}{1-\alpha}\|\psi\tA(\tau)\|_{L^\infty([1,\infty))},
\tfrac{10}{11}\|\psi\tW(\tau)\|_{L^\infty([1,\infty))}
\right\}
\leq\tfrac{1+\mathsf c}{2}\varepsilon e^{-\lambda\tau}.
\label{eq:nonlinear-bootstrap-exterior-improvement}
\end{equation}
This proves~\eqref{eq:nonlinear-bootstrap-exterior-target}.

\subsection{Weighted bounds for the undifferentiated variables}

We conclude by proving bounds for the undifferentiated variables, deducing them from the bounds for the differentiated variables.

\emph{1. Reconstruction from infinity.} By~\eqref{eq:nonlinear-bootstrap-exterior-improvement} and~\eqref{eq:nonlinear-bootstrap-exterior-weight-derivative}, $\|R^{\frac{1}{2\crrb}}\Ytil(\tau)\|_{L^\infty}\leq\mathsf{C}\varepsilon e^{-\lambda\tau}$. Using~\eqref{eq:nonlinear-bootstrap-global-VQ-bound} and the preceding differentiated bound, the equations~\eqref{eq:Q-eq}--\eqref{eq:V-eq} and Gr\"onwall's inequality at fixed $R$ propagate the initial vanishing at infinity; since the background also vanishes there, $\widetilde V,\widetilde Q\to0$ as $R\to\infty$. Integrating~\eqref{eq:ring-inverse} from infinity then gives
\begin{subequations}
\label{eq:nonlinear-bootstrap-tail-reconstruction}
\begin{align}
R^{\frac{1}{2\crrb}}|\widetilde V(R,\tau)|&\leq\tfrac12 R^{\frac{1}{2\crrb}}\int_R^\infty|\tZ+\tW|(S,\tau)\tfrac{\dd S}{S}\leq\mathsf{C}\left\|R^{\frac{1}{2\crrb}}\Ytil(\tau)\right\|_{L^\infty},
\label{eq:nonlinear-bootstrap-V-tail-reconstruction}\\
\alpha R^{\frac{1}{2\crrb}}|\widetilde Q(R,\tau)|&\leq R^{\frac{1}{2\crrb}}\int_R^\infty\left|-\tfrac\alpha2\tZ+\tA+\tfrac\alpha2\tW\right|(S,\tau)\tfrac{\dd S}{S}\leq\mathsf{C}\left\|R^{\frac{1}{2\crrb}}\Ytil(\tau)\right\|_{L^\infty}.
\label{eq:nonlinear-bootstrap-Q-tail-reconstruction}
\end{align}
\end{subequations}
Consequently,
\begin{equation}
\sup_{R\geq1}R^{\frac{1}{2\crrb}}\left(|\Ytil|+|\widetilde V|+\alpha|\widetilde Q|+R|\partial_R\widetilde V|+\alpha R|\partial_R\widetilde Q|\right)(R,\tau)\leq\mathsf{C}\varepsilon e^{-\lambda\tau}.
\label{eq:nonlinear-bootstrap-weighted-VQ-bound}
\end{equation}

\emph{2. Sharp spatial bounds.} By Theorem~\ref{lem:guderley-profile-monotonicity-decay} and the stationary identities~\eqref{eq:guderley:VQ-system-Q} and~\eqref{eq:guderley:B-elimination}, we have
\begin{equation}
\Qbar(R)\simeq R^{-\frac{1}{\crrb}},\qquad \Bbar(R)\simeq R^{1-\frac{1}{\crrb}},\qquad R\geq1.
\label{eq:nonlinear-bootstrap-background-two-sided-bounds}
\end{equation}
We use the cancellation
\begin{equation}
\mathcal L_{\mathrm{bg}}+\tfrac{1}{\crrb}\Tmat
=\mathcal O\left(R^{-\frac{1}{\crrb}}\right)+\tfrac{1}{\crrb}\left(\widetilde V I+\alpha\widetilde Q\diag(-1,0,1)\right).
\label{eq:nonlinear-bootstrap-sharp-weight-cancellation}
\end{equation}
Using~\eqref{eq:nonlinear-bootstrap-sharp-weight-cancellation} and multiplying~\eqref{eq:nonlinear-bootstrap-perturbation-system} by $R^{\frac{1}{\crrb}}$, we obtain, for $R$ large enough,
\begin{equation*}
\partial_\tau\left(R^{\frac{1}{\crrb}}\Ytil\right)+\Tmat R\partial_R\left(R^{\frac{1}{\crrb}}\Ytil\right)=\mathcal O\left(|\Ytil|\right)+\mathcal O\left(\varepsilon e^{-\lambda\tau}\right)\left(1+R^{\frac{1}{\crrb}}|\Ytil|\right).
\end{equation*}
Here we used~\eqref{eq:nonlinear-bootstrap-far-field-coefficients}, \eqref{eq:nonlinear-bootstrap-pressure-and-cu-control}, and~\eqref{eq:nonlinear-bootstrap-weighted-VQ-bound}.

To integrate this equation, we first multiply~\eqref{eq:nonlinear-bootstrap-perturbation-system} by $\min\{R,L\}^{\frac{1}{\crrb}}$, $L\geq R_3$, and integrate from $\tau=0$ or $R=R_2$. The initial bound~\eqref{eq:nonlinear-bootstrap-initial-sharp-ring-bound}, the inner estimate, \eqref{eq:nonlinear-bootstrap-far-field-coefficients}, \eqref{eq:nonlinear-bootstrap-pressure-and-cu-control}, and~\eqref{eq:nonlinear-bootstrap-weighted-VQ-bound} give
\begin{equation*}
\left\|\min\{R,L\}^{\frac{1}{\crrb}}\Ytil(\tau)\right\|_{L^\infty([R_2,\infty))}\leq\mathsf{C}\varepsilon+\mathsf{C}\varepsilon\int_0^\tau e^{-\lambda s}\left(1+\left\|\min\{R,L\}^{\frac{1}{\crrb}}\Ytil(s)\right\|_{L^\infty([R_2,\infty))}\right)\,\dd s.
\end{equation*}
Gr\"onwall's inequality in time, $L\to\infty$, and~\eqref{eq:nonlinear-bootstrap-tail-reconstruction} therefore imply
\begin{equation}
\sup_{R\geq1}R^{\frac{1}{\crrb}}\left(|\Ytil|+|\widetilde V|+\alpha|\widetilde Q|+R|\partial_R\widetilde V|+\alpha R|\partial_R\widetilde Q|\right)(R,\tau)\leq\mathsf{C}\varepsilon.
\label{eq:nonlinear-bootstrap-sharp-VQ-bound}
\end{equation}
In particular we deduce
\begin{equation}
(1-\mathsf{C}\varepsilon)\Qbar(R)\leq Q(R,\tau)\leq(1+\mathsf{C}\varepsilon)\Qbar(R).
\label{eq:nonlinear-bootstrap-global-Q-positivity}
\end{equation}

\emph{3. Entropy reconstruction.} Since $\widetilde K=\log(B/\Bbar)$, \eqref{eq:nonlinear-bootstrap-initial-far-field} and~\eqref{eq:nonlinear-bootstrap-background-two-sided-bounds} give $\|\widetilde K(0)\|_{L^\infty}\leq\mathsf{C}\varepsilon^2$. Subtracting the stationary equation from~\eqref{eq:K-eq} gives
\begin{equation*}
(\partial_\tau+(\crrb+V)R\partial_R)\widetilde K=\tcuu-\widetilde V(1+R\partial_R\Kbar).
\end{equation*}
 Integration along the $A$-characteristics from $\tau=0$ or the shock, using~\eqref{eq:nonlinear-bootstrap-undifferentiated-shock-traces}, \eqref{eq:nonlinear-bootstrap-pressure-and-cu-control}, and~\eqref{eq:nonlinear-bootstrap-weighted-VQ-bound}, yields
\begin{equation}
\|\widetilde K(\tau)\|_{L^\infty}\leq\mathsf{C}\varepsilon^2+\mathsf{C}\varepsilon\int_0^\tau e^{-\lambda s}\,\dd s\leq\mathsf{C}\varepsilon.
\label{eq:nonlinear-bootstrap-global-K-bound}
\end{equation}
Using the last identity in~\eqref{eq:ring-inverse}, the bounds~\eqref{eq:nonlinear-bootstrap-undifferentiated-shock-traces}, \eqref{eq:nonlinear-bootstrap-weighted-VQ-bound}, \eqref{eq:nonlinear-bootstrap-background-two-sided-bounds}, and~\eqref{eq:nonlinear-bootstrap-global-Q-positivity}, and integrating from the shock, we obtain
\begin{equation}
\sup_{R\geq1}R^{-\frac{1}{2\crrb}}\left(|\widetilde K|+R|\partial_R\widetilde K|\right)(R,\tau)\leq\mathsf{C}\varepsilon e^{-\lambda\tau},
\label{eq:nonlinear-bootstrap-weighted-K-bound}
\end{equation}
and
\begin{equation}
\sup_{R\geq1}R^{\frac{1}{2\crrb}-1}\left(|\widetilde B|+R|\partial_R\widetilde B|\right)(R,\tau)\leq\mathsf{C}\varepsilon e^{-\lambda\tau}.
\label{eq:nonlinear-bootstrap-weighted-B-bound}
\end{equation}

\emph{4. Convergence to the profile power.} Interpolating~\eqref{eq:nonlinear-bootstrap-weighted-VQ-bound} with~\eqref{eq:nonlinear-bootstrap-sharp-VQ-bound}, and~\eqref{eq:nonlinear-bootstrap-weighted-K-bound} with~\eqref{eq:nonlinear-bootstrap-global-K-bound}, gives, for $0<\eta<\frac{1}{2\crrb}$,
\begin{equation}
\sup_{R\geq1}R^{\frac{1}{\crrb}-\eta}\left(|\widetilde V|+\alpha|\widetilde Q|+\tfrac{|\widetilde B|}{R}\right)(R,\tau)\leq\mathsf{C}\varepsilon e^{-2\crrb\eta\lambda\tau}.
\label{eq:nonlinear-bootstrap-near-profile-convergence}
\end{equation}
Here we use $|e^{\widetilde K}-1|\leq e^{\|\widetilde K\|_{L^\infty}}|\widetilde K|$ and $\mathsf{C}\varepsilon\leq1$.

\subsection{Proof of Theorem~\ref{thm:main-self-similar}}

We now gather the conclusions and sum up the proof of Theorem~\ref{thm:main-self-similar}. 
\begin{proof}[Proof of Theorem~\ref{thm:main-self-similar}]
We assumed the bootstrap bounds~\eqref{eq:nonlinear-bootstrap-exterior-assumption} and~\eqref{eq:nonlinear-bootstrap-modulation-assumption} on $[0,T]$. The bounds~\eqref{eq:nonlinear-bootstrap-A-improvement-preliminary}, \eqref{eq:nonlinear-bootstrap-W-improvement-preliminary}, and~\eqref{eq:nonlinear-bootstrap-Z-improvement}, combined with~\eqref{eq:nonlinear-bootstrap-exterior-improvement} and~\eqref{eq:nonlinear-bootstrap-pressure-and-cu-control}, strictly improve these assumptions and close the bootstrap by continuity. In particular, since we control the $C^1$ norm of the solution, a continuity argument combined with classical local well-posedness estimates for the shock front problem~\cite{Maj83exist,Maj83stab,Li80,Daf10} proves that there exists a unique asymptotically self-similar shock solution to Euler for all $\tau \ge 0$. 

Part~(\ref{item:main-self-similar-decay}) follows now from the bounds~\eqref{eq:nonlinear-bootstrap-exterior-improvement}, \eqref{eq:nonlinear-bootstrap-weighted-VQ-bound}, and~\eqref{eq:nonlinear-bootstrap-weighted-B-bound}. Part~(\ref{item:main-self-similar-modulation}) is proved in~\eqref{eq:nonlinear-bootstrap-pressure-and-cu-control}. Part~(\ref{item:main-self-similar-sharp-bounds}) follows from~\eqref{eq:nonlinear-bootstrap-sharp-VQ-bound}, \eqref{eq:nonlinear-bootstrap-global-Q-positivity}, and~\eqref{eq:nonlinear-bootstrap-global-K-bound}. Part~(\ref{item:main-self-similar-convergence}) follows from~\eqref{eq:nonlinear-bootstrap-near-profile-convergence}.
\end{proof}
We now prove Corollary~\ref{thm:main-physical}.
\begin{proof}[Proof of Corollary~\ref{thm:main-physical}]
We apply the change of variables~\eqref{eq:modulated-coordinates}--\eqref{eq:modulated-ansatz} to the solution of Theorem~\ref{thm:main-self-similar}. At the shock, $u^r(1^+,-1)=\Cu(0)V(1^+,0)$ and $\bar u(1^+,-1)=\Vbar(1^+)=-\frac{2\crrb}{\gamma+1}$. By~\eqref{eq:nonlinear-bootstrap-initial-far-field}, the trace $V(1^+,0)$ is bounded away from zero. The physical velocity smallness therefore gives
\begin{equation*}
|\Cu(0)-1|
=\left|\tfrac{u^r(1^+,-1)-\bar u(1^+,-1)-\widetilde V(1^+,0)}{V(1^+,0)}\right|
\leq\mathsf{C}\varepsilon^2.
\end{equation*}
The modulation choice~\eqref{eq:cr:0} gives
\begin{subequations}
\label{eq:physical-proof-scaling}
\begin{align}
\Cr(\tau)&=e^{-\crrb\tau},
\label{eq:physical-proof-radial-scaling}\\
\Cu(\tau)=\Cb(\tau)&=\Cu(0)\exp\left((1-\crrb)\tau-\int_0^\tau\tcuu(s)\,\dd s\right).
\label{eq:physical-proof-amplitude-scaling}
\end{align}
\end{subequations}
For the blowup time $t_\ast$, we have the identity
\begin{equation*}
t_\ast=t(+\infty)=-1+\int_0^\infty\tfrac{\Cr(s)}{\Cu(s)}\,\dd s=-1+\Cu(0)^{-1}\int_0^\infty\exp\left(-s+\int_0^s\tcuu(v)\,\dd v\right)\,\dd s.
\end{equation*}
The modulation estimate~\eqref{eq:main-self-similar-modulation-decay} gives $\int_0^\infty|\tcuu(s)|\,\dd s\leq \mathsf{C}\varepsilon$. Hence the time $t_\ast:=\lim_{\tau\to\infty}t(\tau)$ is finite; more precisely we have the bound
\begin{equation*}
|\Cu(0)(t_\ast+1)-1|\leq\mathsf{C}\varepsilon,
\end{equation*}
which together with $|\Cu(0)-1|\leq\mathsf{C}\varepsilon^2$ proves~\eqref{eq:main-physical-focusing-time}.

\emph{Part~(\ref{item:main-physical-focusing}).} The bound~\eqref{eq:main-physical-focusing-time} is proved above. By~\eqref{eq:modulated-time}, we have
\begin{equation}
t_\ast-t(\tau)=\Cu(0)^{-1}\int_\tau^\infty\exp\left(-s+\int_0^s\tcuu(v)\,\dd v\right)\,\dd s.
\label{eq:physical-proof-remaining-time}
\end{equation}
The modulation bound~\eqref{eq:main-self-similar-modulation-decay} and $|\Cu(0)-1|\leq\mathsf{C}\varepsilon^2$ give $t_\ast-t(\tau)\simeq e^{-\tau}$. Together with~\eqref{eq:physical-proof-scaling} and the identities $\s(t(\tau))=\Cr(\tau)$ and $-\ds(t(\tau))=\crrb\Cu(\tau)$, we have~\eqref{eq:main-physical-focusing-rates}.

\emph{Part~(\ref{item:main:rh}).} The Rankine--Hugoniot conditions follow from Theorem~\ref{thm:main-self-similar}, while the Lax inequalities follow from~\eqref{eq:nonlinear-bootstrap-pressure-decay}, \eqref{eq:nonlinear-bootstrap-positive-inner-A-speed}, and~\eqref{eq:nonlinear-bootstrap-Z-boundary-directions}.

\emph{Part~(\ref{item:main-physical-blowup}).} The sharp bound~\eqref{eq:main-self-similar-sharp-VQ-bound}, the profile decay~\eqref{eq:guderley-profile-decay}, and $\crrb<1$ bound $U=RV$ and $\Sigma=RQ$ uniformly for $R\geq1$. Their exterior shock traces~\eqref{eq:finite-pressure-RH-traces} remain bounded away from zero by~\eqref{eq:nonlinear-bootstrap-pressure-decay}. Thus~\eqref{eq:modulated-ansatz}, \eqref{eq:physical-proof-amplitude-scaling}, and~\eqref{eq:physical-proof-remaining-time}, together with the interior state~\eqref{eq:finite-pressure-interior-state}, give the blowup rates~\eqref{eq:main-physical-blowup-rates}. For the density, $\Cb=\Cu$ and~\eqref{eq:density-VQK} give
\begin{equation*}
\rho(\Cr(\tau)R,t(\tau))=\overline{\mathsf R}(R)\left(\tfrac{Q(R,\tau)}{\Qbar(R)}e^{-\widetilde K(R,\tau)}\right)^{\frac{1}{\alpha}}.
\end{equation*}
The background density $\overline{\mathsf R}=(\frac{\alpha R\Qbar}{\Bbar})^{\frac{1}{\alpha}}$ is bounded by~\eqref{eq:nonlinear-bootstrap-background-two-sided-bounds}. The bounds~\eqref{eq:main-self-similar-Q-comparison} and~\eqref{eq:main-self-similar-global-K-bound} therefore bound the exterior density uniformly, while the interior density is $1$.

\emph{Part~(\ref{item:main-physical-interior}).} The constant state~\eqref{eq:finite-pressure-interior-state} solves~\eqref{eq:euler:cl} and agrees with the initial interior data. By uniqueness, it is preserved ahead of the shock, proving~\eqref{eq:main-physical-interior-state}.

\emph{Part~(\ref{item:main-physical-convergence}).} Using~\eqref{eq:physical-proof-scaling}, \eqref{eq:physical-proof-remaining-time}, and $\s(t(\tau))=\Cr(\tau)$, we obtain
\begin{equation*}
\tfrac{(t_\ast-t(\tau))\Cu(\tau)}{\s(t(\tau))}=\int_0^\infty\exp\left(-s+\int_\tau^{\tau+s}\tcuu(v)\,\dd v\right)\,\dd s\longrightarrow1.
\end{equation*}
The limit follows by the dominated convergence theorem. At $r=\s(t)R$, the ansatz~\eqref{eq:modulated-ansatz}, $\Cb=\Cu$, and the perturbation decay~\eqref{eq:main-self-similar-perturbation-decay} now imply~\eqref{eq:main-physical-profile-convergence} locally uniformly for $R>1$. On $0\leq R<1$, the factor $\frac{t_\ast-t}{\s(t)}$ tends to zero by~\eqref{eq:main-physical-focusing-rates} and $\crrb<1$. Therefore we obtain the convergence also for $R <1$.
\end{proof}

\begin{remark}
We now verify the three conditions defining an implosion in the introduction, with center $x_*=0$ and time $t_*$. By~\eqref{eq:finite-pressure-RH-density} and~\eqref{eq:nonlinear-bootstrap-pressure-decay}, the exterior density trace satisfies $\rho(\s(t)^+,t)\to\frac{\gamma+1}{\gamma-1}>0$, while~\eqref{eq:finite-pressure-RH-Q} and~\eqref{eq:physical-proof-amplitude-scaling} give $\sigma(\s(t)^+,t)\to+\infty$. Since $p=\frac{\alpha^2}{\gamma}\rho\sigma^2$ and $\s(t)\to0$, the pressure becomes unbounded in every neighborhood of the origin, proving the first condition.

For the inward-focusing condition, we take $\mathsf r(t)=\s(t)/2$. The velocity vanishes on this ball by~\eqref{eq:main-physical-interior-state}, so the required scalar product is zero.

Finally, the spatial bound~\eqref{eq:main-self-similar-sharp-VQ-bound} and the profile decay~\eqref{eq:guderley-profile-decay} give, for $r=\Cr(\tau)R>\s(t(\tau))$,
\begin{equation*}
|u^r(r,t(\tau))|+|\sigma(r,t(\tau))|
\leq\mathsf C\Cu(\tau)R^{1-\frac{1}{\crrb}}
\leq\mathsf C r^{1-\frac{1}{\crrb}}.
\end{equation*}
This implies the flow variables do not blow up away from $0$.
\end{remark}

\section*{Acknowledgements}
The author thanks his advisor, Vlad Vicol, for introducing him to this problem and for stimulating discussions. The author is also grateful to Jiajie Chen and Steve Shkoller for stimulating discussions on the stability of imploding shock waves. The work of the author was partially supported by the Collaborative NSF grant DMS-2307681 and by a Simons Dissertation Fellowship.

The author acknowledges the use of GPT 5.6 Sol and GPT 6 Astra for assistance with editing, checking calculations, and numerical exploration. In particular, the author identified the constant weights used in the stability argument in Section~\ref{sec:nonlinear-bootstrap} and the constants in the barrier constructions in Appendix~\ref{app:three-profile-interval-proof} through discussions with ChatGPT. These models also assisted with drafting and refining the \texttt{SageMath} code used for the computer-assisted verification in Appendix~\ref{app:three-profile-computation}.

\appendix
\section{Proof of the three profile inequalities}
\label{app:three-profile-interval-proof}
The goal of this Appendix is to provide the mathematical arguments used to prove Lemma~\ref{lem:three-profile-inequalities}. The proof of the Lemma is split into different parts, many of which require a computer-assisted argument. The details of the computer-assisted arguments are contained in Appendix~\ref{app:three-profile-computation}. 
 
The arguments used to prove the Lemma involve barriers in the phase portrait (see Figure~\ref{fig:guderley-profile-enclosure}). We start by discussing the construction of those.

\subsection{The barrier curves}

In order to treat the uniform limit as $\gamma \downarrow 1$, we introduce the parameter $\theta$
\begin{equation*}
 \theta:=\sqrt{\tfrac{2(\gamma-1)}\gamma},\qquad
 \gamma=\tfrac2{2-\theta^2},\qquad
 \alpha=\tfrac{\theta^2}{2(2-\theta^2)},\qquad
 0<\theta\leq\tfrac2{\sqrt5}.
\end{equation*}
It is convenient to normalize the variables $(\bar V,\bar Q)$ in the following way:
\begin{subequations}
\label{eq:ca-normalization}
\begin{align}
 X&:=\tfrac{\crrb+\Vbar}{\crrb\theta},& C&:=\tfrac{\alpha\Qbar}{\crrb\theta},\label{eq:ca-normalization-XC}\\
 z&:=\tfrac{\crrb^{-1}-1}{2\gamma},& \crrb&=\tfrac1{1+2\gamma z}.\label{eq:ca-normalization-z}
\end{align}
\end{subequations}
The first advantage of this notation is that the shock location is now independent of $\crrb$:
\begin{equation}
\label{eq:ca-shock-point}
 X_0=\tfrac\theta{4-\theta^2},\qquad C_0=\tfrac2{4-\theta^2}.
\end{equation}
The second advantage is that at the sonic point we have
\begin{equation*}
 (X,C)|_{P_6}=(X_6,X_6),
\end{equation*}
where
\begin{equation*}
X_6 = \tfrac{\crrb + \Vbar_6}{\crrb \theta}.
\end{equation*}
For later use, it will be convenient to use the normalized quantity $q$\footnote{It can be shown that as $\gamma \downarrow 1$,  $q \to \frac{1}{2}$.}
\begin{equation}
\label{eq:ca-background-endpoint-z}
 q:=\tfrac z\theta=\tfrac{X_6(1-\theta X_6)}{\gamma+(2-\gamma)\theta X_6},
\end{equation}
where the last identity follows from the definitions of $X_6$ and $\Vbar_6$ in~\eqref{eq:ca-normalization-XC} and~\eqref{eq:guderley:sonic-V-value}.
To normalize the shock to sonic interval uniformly in $\gamma$, we also introduce the rescaled coordinate
\begin{equation*}
 L:=X_6-X_0,\qquad x:=\tfrac{X_6-X}{L}.
\end{equation*}
In particular, $x=0$ corresponds to the sonic location, while $x=1$ corresponds to the shock location.

We now use $x$ to write the ODE that $C(x)$ solves
\begin{equation*}
 \tfrac{\dd C}{\dd x}=-\tfrac{L\theta C N_c[X,C]}{X N_x[X,C]}.
\end{equation*}
Here $N_c$ and $N_x$ are polynomials computed from the definitions of $\Delta_{\Vbar}$ and $\Delta_{\Qbar}$ in~\eqref{eq:guderley:Delta-V-def} and~\eqref{eq:guderley:Delta-Q-def}. In particular, they have the explicit formulas
\begin{subequations}
\label{eq:ca-scaled-phase-fields}
\begin{align}
 N_c[X,C]&:=C^2(X+2q)-\gamma X^3+\bigl(q\gamma(\gamma-3)+\tfrac{\gamma\theta}{2}\bigr)X^2-\tfrac{q\gamma^2\theta}{2}X,
 \label{eq:ca-scaled-Nf}\\
 N_x[X,C]&:=C^2(3\theta X-3+4\theta q)-(\theta X-1)X(X+2\gamma q),
 \label{eq:ca-scaled-FG}\\
 D(X,C)&:=X^2-C^2.
 \label{eq:ca-scaled-D}
\end{align}
\end{subequations}

 Let $T_3(x)=\sum_{j=0}^3c_jx^j$ be the Taylor polynomial of $C$ (as a function of $x$) at the sonic point. We already computed these coefficients in the variables $(\bar V, \bar Q)$ in~\eqref{eq:guderley:sonic-phase-slope} and~\eqref{eq:guderley:sonic-coefficient-recurrence}. In terms of $\eta_j$, we can easily compute the $c_j$ using the chain rule:
\begin{equation}
\label{eq:ca-common-sonic-Taylor}
 c_0=X_6,\qquad
 c_j=(-1)^j\alpha(\crrb\theta)^{j-1}\eta_jL^j,\quad 1\leq j\leq3.
\end{equation}
On $0\leq x\leq\frac12$, we will use the following barriers:
\begin{equation}
\label{eq:ca-common-sonic-barriers}
 C_6^\flat(x):=T_3(x)-\tfrac{|c_3|}{1000}x^3,\qquad
 C_6^\sharp(x):=T_3(x)+\tfrac52|c_3|x^3.
\end{equation}

In the shock region $\frac{1}{2} \le x \le 1$, we will have to be more careful in the construction of barriers since some of the natural Taylor coefficients degenerate in the limit $\gamma \downarrow 1$. We introduce the following quantities
\begin{equation}
\label{eq:ca-shock-factor}
 k:=\tfrac{2\theta q}{3-4\theta q},\qquad
 \mathcal B(X,C):=\tfrac{\theta N_c+k N_x}{X\theta N_x},\qquad
 C=\left(\tfrac{X_0}{X}\right)^k S.
\end{equation}
This choice is done so that we can rewrite
\begin{equation*}
 \tfrac{\dd C}{\dd X}=-\tfrac{k}{X}C+\theta C\mathcal B(X,C),
\end{equation*}
and then 
\begin{equation*}
 \tfrac{\dd S}{\dd X}=\theta S\mathcal B(X,C),\qquad S(X_0)=C_0.
\end{equation*}
We write $h:=X-X_0=L(1-x)$. By using the ODE solved by $S$ we deduce the Taylor series
\begin{equation*}
 S(X_0+h)=C_0+\theta a_1h+\theta a_2h^2+O(h^3)\qquad\text{as }h\to0.
\end{equation*}
More explicitly, with $\mathcal B_0:=\mathcal B(X_0,C_0)$, differentiation of the equation for $S$ gives
\begin{equation}
\label{eq:ca-shock-Taylor-coefficients}
 a_1:=C_0\mathcal B_0,\qquad
 a_2:=\tfrac{C_0}{2}\left(\theta\mathcal B_0^2+
 \partial_X\mathcal B(X_0,C_0)+\tfrac{\theta C_0 N_c[X_0,C_0]}{X_0N_x[X_0,C_0]}\partial_C\mathcal B(X_0,C_0)\right).
\end{equation}
Then, for $\frac12\leq x\leq1$, we define
\begin{subequations}
\label{eq:ca-common-shock-barriers}
\begin{align}
 C_1^\flat(x)&:=\left(\tfrac{X_0}{X_0+h}\right)^k
 \left[C_0+\theta\left(a_1h+(a_2+\varepsilon_\flat)h^2\right)\right],
 \label{eq:ca-common-shock-lower}\\
 C_1^\sharp(x)&:=\left(\tfrac{X_0}{X_0+h}\right)^k
 \left[C_0+\theta\left(a_1h+(a_2+\varepsilon_\sharp)h^2\right)\right],
 \label{eq:ca-common-shock-upper}\\
 \varepsilon_\flat&:=-\tfrac{\theta^2}{32}-\tfrac1{10000},\qquad
 \varepsilon_\sharp:=\tfrac{(1-\theta)^5}{2}+\tfrac1{1000}.
 \label{eq:ca-common-shock-errors}
\end{align}
\end{subequations}
That is, we slightly perturb the Taylor series for $S$, and multiply it back by $\left(\frac{X_0}{X_0+h}\right)^k$. Before proceeding we observe that the value of $z$ is not known a priori; part of the proof will consist in finding rigorous bounds for $z$. In particular, the functions $C_6^\flat$, $C_6^\sharp$, $C_1^\flat$, and $C_1^\sharp$ should not be thought of only as functions of $x$ and $\gamma$, but also of $z$.

In order to estimate the parameter $z(\gamma)$, we will use two functions $z_+$ and $z_-$ whose precise construction is detailed in Section~\ref{subsec:ca-shooting} of Appendix~\ref{app:three-profile-computation}. These functions have been obtained by numerical shooting at selected values of $\gamma$ and then interpolating in $\theta$.

Following Jang, Liu, and Schrecker~\cite{JLS25b}, we define
\begin{equation*}
 z_g(\gamma):=\tfrac{2(\gamma-1)}{\sqrt{(2\gamma^2-\gamma+1)^2+2\gamma(\gamma-1)(4\gamma(\gamma-1)+\tfrac83)}+2\gamma^2-\gamma+1},\qquad
 z_M(\gamma):=\tfrac1{\gamma+2+2\sqrt{2\gamma}}.
\end{equation*}

\begin{lemma}
\label{lem:ca-common-barriers}
Let $z_-(\gamma)<z_+(\gamma)$ be the functions constructed in Section~\ref{subsec:ca-shooting}. Then, for every $1<\gamma \le \frac53$ and $z_-(\gamma)\le z \le z_+(\gamma)$ the following hold.

\begin{enumerate}[label=(\roman*),ref=\roman*,leftmargin=2em]
\item\label{item:ca-common-admissibility} We have the inequalities
\begin{equation}
\label{eq:ca-common-admissible-z}
 z_g(\gamma)<z<z_M(\gamma).
\end{equation}
\item\label{item:ca-common-qualitative} The coefficients $c_1,c_2,c_3$ are well defined, with $L>0$, $c_1>0$ and $c_3\ne0$. The quantity $1+(\gamma-2)z-2\theta X_6$ and the slope discriminant,
\begin{equation*}
 \left.\left((\partial_X N_x-\theta\partial_C N_c)^2
 +4\theta\partial_C N_x\partial_X N_c\right)\right|_{(X,C)=(X_6,X_6)},
\end{equation*}
are positive. Moreover the functions $C_6^\flat$, $C_6^\sharp$, $C_1^\flat$, and $C_1^\sharp$ are positive on their respective intervals of definition.
\item\label{item:ca-common-field-signs} For $0<x\leq\frac12$ and $C_6^\flat(x)\leq C\leq C_6^\sharp(x)$, and for $\frac12\leq x\leq1$ and $C_1^\flat(x)\leq C\leq C_1^\sharp(x)$, we have
\begin{equation}
\label{eq:ca-common-field-signs}
 N_x[X,C]<0,\qquad D(X,C)<0.
\end{equation}
\item\label{item:ca-common-sonic-signs} For $0<x<\frac12$, we have the inequalities
\begin{equation}
\label{eq:ca-common-sonic-signs}
\begin{aligned}
 -X N_x\bigl[X,C_6^\flat(x)\bigr]\tfrac{\dd C_6^\flat}{\dd x}(x)
 -L\theta C_6^\flat(x)N_c\bigl[X,C_6^\flat(x)\bigr]&<0,\\
 -X N_x\bigl[X,C_6^\sharp(x)\bigr]\tfrac{\dd C_6^\sharp}{\dd x}(x)
 -L\theta C_6^\sharp(x)N_c\bigl[X,C_6^\sharp(x)\bigr]&>0.
\end{aligned}
\end{equation}
\item\label{item:ca-common-shock-signs} For $\frac12<x<1$, we have the inequalities
\begin{equation}
\label{eq:ca-common-shock-signs}
\begin{aligned}
 -X N_x\bigl[X,C_1^\flat(x)\bigr]\tfrac{\dd C_1^\flat}{\dd x}(x)
 -L\theta C_1^\flat(x)N_c\bigl[X,C_1^\flat(x)\bigr]&>0,\\
 -X N_x\bigl[X,C_1^\sharp(x)\bigr]\tfrac{\dd C_1^\sharp}{\dd x}(x)
 -L\theta C_1^\sharp(x)N_c\bigl[X,C_1^\sharp(x)\bigr]&<0.
\end{aligned}
\end{equation}
\end{enumerate}
We recall that $X=X_6-Lx$.
\end{lemma}

\begin{proof}
We prove the five assertions in order.
Part~(\ref{item:ca-common-admissibility}). The proof of the inequalities $z_g(\gamma)<z<z_M(\gamma)$ is computer-assisted; see Section~\ref{subsec:ca-computational-proofs} for details.

Part~(\ref{item:ca-common-qualitative}). The bounds~\eqref{eq:ca-common-admissible-z} imply $L>0$; see Remark~3.2 and Appendix~B in~\cite{JLS25b}. For the other sign, \eqref{eq:guderley:sonic-V-value} and~\eqref{eq:ca-normalization} give
\begin{equation*}
 1+(\gamma-2)z-2\theta X_6
 =\sqrt{\bigl[1-(\gamma+2+2\sqrt{2\gamma})z\bigr]\bigl[1-(\gamma+2-2\sqrt{2\gamma})z\bigr]}>0.
\end{equation*}
Both factors are positive by the upper bound in~\eqref{eq:ca-common-admissible-z}; this is equivalent to the positivity of the radical in equations~(2.8)--(2.9) and Remark~2.5 of~\cite{JLS25b}. Lemmas~3.3--3.5 of~\cite{JLS25b} give a positive discriminant, $\eta_1<0$, and well-defined Taylor coefficients. Thus~\eqref{eq:ca-common-sonic-Taylor} gives well-defined $c_1,c_2,c_3$ and $c_1=-\alpha\eta_1L>0$.

 The functions $C_6^\flat$, $C_6^\sharp$, $C_1^\flat$, and $C_1^\sharp$ are shown to be positive on their respective intervals of definition by a computer-assisted argument. 

Part~(\ref{item:ca-common-field-signs}). The inequalities~\eqref{eq:ca-common-field-signs} on $C_6^\flat$, $C_6^\sharp$, $C_1^\flat$, and $C_1^\sharp$ for $x>0$ are verified by a computer-assisted argument; see Sections~\ref{subsec:ca-sonic-cap} and~\ref{subsec:ca-computational-proofs}. By part~(\ref{item:ca-common-qualitative}), these curves are positive. Since both $N_x$ and $D$ are affine in $C^2$, the same signs hold throughout the whole region.

Part~(\ref{item:ca-common-sonic-signs}). The inequalities~\eqref{eq:ca-common-sonic-signs} are verified by a computer-assisted calculation; see Sections~\ref{subsec:ca-sonic-cap} and~\ref{subsec:ca-computational-proofs} for details.

Part~(\ref{item:ca-common-shock-signs}). We use the change of unknown $C=(\frac{X_0}{X})^kS$. Since $\frac{\dd X}{\dd x}=-L$, the definition of $\mathcal B$ gives
\begin{equation*}
 -X N_x\tfrac{\dd C}{\dd x}-L\theta C N_c
 =LX N_x\left(\tfrac{X_0}{X}\right)^k
 \left(\tfrac{\dd S}{\dd X}-\theta S\mathcal B(X,C)\right).
\end{equation*}
By part~(\ref{item:ca-common-field-signs}), $LX N_x(\frac{X_0}{X})^k$ is negative. Thus, with $S^\flat,S^\sharp$ denoting the bracketed polynomials in~\eqref{eq:ca-common-shock-barriers}, the inequalities~\eqref{eq:ca-common-shock-signs} are equivalent to
\begin{equation*}
 \tfrac{\dd S^\flat}{\dd X}-\theta S^\flat\mathcal B(X,C_1^\flat)<0,
 \qquad
 \tfrac{\dd S^\sharp}{\dd X}-\theta S^\sharp\mathcal B(X,C_1^\sharp)>0.
\end{equation*}
These inequalities are verified by a computer-assisted calculation; see Sections~\ref{subsec:ca-sonic-cap} and~\ref{subsec:ca-computational-proofs} for details.

\end{proof}

From the barrier inequalities and a simple continuity argument, we are able to deduce lower and upper bounds for the Guderley exponent $\crrb$.

\begin{lemma}
\label{lem:ca-common-parameter}
For every $1<\gamma\leq\frac53$, the following hold.
\begin{enumerate}[label=(\roman*),ref=\roman*,leftmargin=2em]
\item\label{item:ca-common-ordering} For every $z_-(\gamma)<z<z_+(\gamma)$, the barriers satisfy $C_6^\flat(x)\leq C_6^\sharp(x)$ and $C_1^\flat(x)\leq C_1^\sharp(x)$ on their respective intervals of definition. The inequalities attain equality, respectively, only at the sonic and shock point.
\item\label{item:ca-common-midpoint-signs} At $z=z_\pm(\gamma)$ we have the inequalities
\begin{equation}
\label{eq:ca-common-midpoint-signs}
 C_1^\flat(\tfrac12;z_-)>C_6^\sharp(\tfrac12;z_-),\qquad
 C_6^\flat(\tfrac12;z_+)>C_1^\sharp(\tfrac12;z_+).
\end{equation}

\item\label{item:ca-common-exponent} The Guderley exponent $\crrb$ satisfies
\begin{equation}
\label{eq:ca-cr-enclosure}
 \tfrac1{1+2\gamma z_+(\gamma)}<\crrb<
 \tfrac1{1+2\gamma z_-(\gamma)},\qquad 1<\gamma\leq\tfrac53.
\end{equation}
\end{enumerate}
\end{lemma}

\begin{proof}
Part~(\ref{item:ca-common-ordering}). Subtracting the explicit formulas~\eqref{eq:ca-common-sonic-barriers} and~\eqref{eq:ca-common-shock-barriers} gives
\begin{equation*}
\begin{aligned}
 C_6^\sharp(x)-C_6^\flat(x)
 &=\tfrac{2501}{1000}|c_3|x^3,
\\
 C_1^\sharp(x)-C_1^\flat(x)
 &=\left(\tfrac{X_0}{X_0+h}\right)^k\theta h^2
 \left[\tfrac{(1-\theta)^5}{2}+\tfrac{\theta^2}{32}+\tfrac{11}{10000}\right],
 \qquad h=L(1-x).
\end{aligned}
\end{equation*}
By Lemma~\ref{lem:ca-common-barriers}, $c_3\ne0$ and $L>0$. Moreover, $0<\theta\leq\frac2{\sqrt5}<1$ and $X_0>0$. Hence both differences have the claimed signs, with equality precisely at $x=0$ and $x=1$, respectively.

Part~(\ref{item:ca-common-midpoint-signs}). The inequalities~\eqref{eq:ca-common-midpoint-signs} are verified by a computer-assisted argument; see Section~\ref{subsec:ca-computational-proofs} for details.

Part~(\ref{item:ca-common-exponent}). Fix $\gamma$ and pick $z$ in $[z_-(\gamma),z_+(\gamma)]$. For each such value, let $\calC_1(x;z)$ denote the solution starting at the shock and $\calC_6(x;z)$ the analytic solution starting at the sonic point with slope $c_1$. The latter exists and depends continuously on $z$ by Theorem~3.8 and Remark~3.9 in~\cite{JLS25b}.

 By~\eqref{eq:ca-common-sonic-Taylor} and~\eqref{eq:ca-common-sonic-barriers}, and since $c_3\ne0$ by Lemma~\ref{lem:ca-common-barriers}(\ref{item:ca-common-qualitative}), we have $\calC_6-C_6^\flat>0$ and $C_6^\sharp-\calC_6>0$ for all sufficiently small $x>0$. Suppose that $\calC_6$ first touches $C_6^\flat$ at a point in $(0,\frac12)$. At this point, the equations and~\eqref{eq:ca-common-sonic-signs} give
\begin{equation*}
0\geq\tfrac{\dd}{\dd x}(\calC_6-C_6^\flat)=-\tfrac{L\theta C_6^\flat N_c[X,C_6^\flat]}{X N_x[X,C_6^\flat]}-\tfrac{\dd C_6^\flat}{\dd x}>0,
\end{equation*}
a contradiction. The same argument, with opposite sign, applies to $C_6^\sharp$.

For $\calC_1$, we argue in the same way using~\eqref{eq:ca-shock-Taylor-coefficients}, \eqref{eq:ca-common-shock-barriers}, and~\eqref{eq:ca-common-shock-signs}. We have thus proved
\begin{equation*}
 C_6^\flat(x;z)\leq\calC_6(x;z)\leq C_6^\sharp(x;z),
 \qquad 0\leq x\leq\tfrac12,
\end{equation*}
\begin{equation*}
 C_1^\flat(x;z)\leq\calC_1(x;z)\leq C_1^\sharp(x;z),
 \qquad \tfrac12\leq x\leq1.
\end{equation*}
At $z=z_-(\gamma)$, the barrier inequality and the first inequality in~\eqref{eq:ca-common-midpoint-signs} give
\begin{equation*}
 \calC_1(\tfrac12;z_-)
 \geq C_1^\flat(\tfrac12;z_-)
 > C_6^\sharp(\tfrac12;z_-)
 \geq\calC_6(\tfrac12;z_-),
\end{equation*}
whereas at $z=z_+(\gamma)$, the second one gives
\begin{equation*}
 \calC_1(\tfrac12;z_+)
 \leq C_1^\sharp(\tfrac12;z_+)
 < C_6^\flat(\tfrac12;z_+)
 \leq\calC_6(\tfrac12;z_+).
\end{equation*}
Since the values $\calC_1(\frac12;z)$ and $\calC_6(\frac12;z)$ depend continuously on $z$ (see Remark~3.9 and the argument following equation~(4.3) in~\cite{JLS25b}), we have that the continuous function $z\mapsto\calC_1(\frac12;z)-\calC_6(\frac12;z)$ changes sign. The intermediate value theorem, as in the argument following equation~(4.3) in~\cite{JLS25b}, gives a parameter $z\in(z_-(\gamma),z_+(\gamma))$ for which
\begin{equation*}
 \calC_1(\tfrac12;z)=\calC_6(\tfrac12;z).
\end{equation*}

The uniqueness of $\crrb$ (and hence of $z$) for the shock-to-$P_6$ connection, proved in Lemma~5.4 of~\cite{JLS25b}, identifies this parameter with that of the Guderley profile. Thus
\begin{equation}
\label{eq:ca-z-enclosure}
 z_-(\gamma)<z<z_+(\gamma).
\end{equation}
The identity $z=(\crrb^{-1}-1)/(2\gamma)$, solved for $\crrb$, gives~\eqref{eq:ca-cr-enclosure}.
\end{proof}

Now we combine the barrier inequalities and the bounds just obtained on $z$ to enclose the Guderley profile $C(x)$.

\begin{lemma}
\label{lem:ca-common-profile}
For every $1<\gamma\leq\frac53$ and the corresponding Guderley similarity exponent $\crrb$, we have the inequalities
\begin{subequations}
\label{eq:ca-profile-enclosure}
\begin{align}
 C_6^\flat(x)&\leq C(x)\leq C_6^\sharp(x),&&0\leq x\leq\tfrac12,
 \label{eq:ca-profile-sonic-half}\\
 C_1^\flat(x)&\leq C(x)\leq C_1^\sharp(x),&&\tfrac12\leq x\leq1.
 \label{eq:ca-profile-shock-half}
\end{align}
\end{subequations}
\end{lemma}

\begin{proof}
By Lemma~\ref{lem:ca-common-parameter}, the Guderley exponent lies in the associated interval. Then, the barrier inequalities in Lemma~\ref{lem:ca-common-barriers}(\ref{item:ca-common-field-signs})--(\ref{item:ca-common-shock-signs}) immediately imply the conclusion of the Lemma by using the argument used in the proof of Lemma~\ref{lem:ca-common-parameter}.
\end{proof}

\begin{figure}[ht]
\centering
\includegraphics{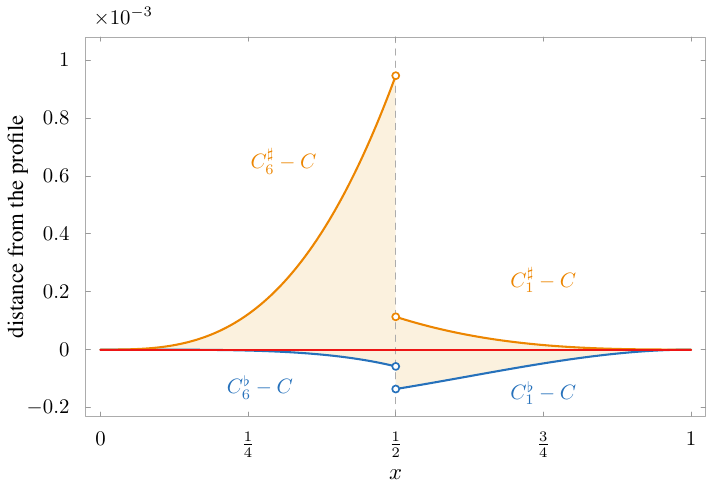}
\caption{\textbf{The Guderley profile enclosure for $d=3$ and $\gamma=\frac53$.} The horizontal axis represents $x$, and the vertical axis represents the barrier difference from the Guderley profile $C$. The red line corresponds to zero distance. The blue and orange curves show the lower and upper barrier differences $C_6^\flat-C$ and $C_6^\sharp-C$ on $0\leq x\leq\frac12$, and $C_1^\flat-C$ and $C_1^\sharp-C$ on $\frac12\leq x\leq1$, obtained from~\eqref{eq:ca-common-sonic-barriers} and~\eqref{eq:ca-common-shock-barriers}. The dashed line marks the midpoint $x=\frac{1}{2}$.}
\label{fig:guderley-profile-enclosure}
\end{figure}

\begin{remark}
This lemma constitutes the backbone of the computer-assisted argument, since it enables us to verify inequalities along the Guderley profile. Indeed, for each fixed $\gamma$, to prove an inequality of the form $F(z,x,C(x))>0$, it suffices to verify $F(z,x,C)>0$ for every $z_-<z<z_+$ and every $(x,C)$ satisfying
\begin{equation*}
\begin{aligned}
 C_6^\flat(x)\leq C\leq C_6^\sharp(x),
 &\qquad 0\leq x\leq\tfrac12,\\
 C_1^\flat(x)\leq C\leq C_1^\sharp(x),
 &\qquad \tfrac12\leq x\leq1.
\end{aligned}
\end{equation*}
In particular, the barriers and the $z_\pm$ are explicit functions that can be bounded using interval arithmetic.
\end{remark}

The next Lemma records explicit bounds, in terms of $x$, of the logarithm of the radius $\log R$, which appears in the expressions~\eqref{eq:profile-row-quantities}.

\begin{lemma}
\label{lem:ca-common-radius}
For every $1<\gamma\leq\frac53$ the following hold.

For the Guderley similarity exponent, let $R^\sharp$ be the function defined in Section~\ref{subsec:ca-radius-formulas}. Then
\begin{equation}
\label{eq:ca-logR-bound}
 0\leq\log R\leq R^\sharp(X),\qquad 1\leq R\leq\Rs.
\end{equation}
\end{lemma}

\begin{proof}

By Lemmas~\ref{lem:ca-common-barriers}(\ref{item:ca-common-field-signs}) and~\ref{lem:ca-common-profile}, we have $N_x<0$ and $D<0$ along the Guderley profile before the sonic point. Hence the equation~\eqref{eq:guderley:V-autonomous} gives
\begin{equation*}
 \tfrac{\dd\log R}{\dd X}
 =\theta\tfrac{D(X,C(X))}{N_x[X,C(X)]}>0,
 \qquad X_0\leq X<X_6.
\end{equation*}
By a computer-assisted argument (see Sections~\ref{subsec:ca-radius-formulas} and~\ref{subsec:ca-computational-proofs}), we verify the inequality
\begin{equation}
\label{eq:ca-S-margin}
 \tfrac1\theta\tfrac{\dd R^\sharp}{\dd X}
 -\tfrac{D(X,C(X))}{N_x[X,C(X)]}>0,
 \qquad X_0\leq X<X_6.
\end{equation}
 Since $\theta>0$, the inequality implies $0<\frac{\dd\log R}{\dd X}<\frac{\dd R^\sharp}{\dd X}$. At the shock, $R=1$ and $R^\sharp(X_0)=0$; hence integration from $X_0$ gives~\eqref{eq:ca-logR-bound} for $1 \le R \le \Rs$.
\end{proof}

\begin{proposition}
\label{prop:ca-background-signs}
For every $1<\gamma\leq\frac53$, the regular $P_6$ Guderley profile in dimension $d=3$ satisfies
\begin{equation}
\label{eq:ca-background-signs}
 \overline{\rZ}>0,\qquad \overline{\rW}>0,\qquad
 -\ell_{ZA}>0,\qquad \ell_{ZW}>0,\qquad -\ell_{WZ}>0,
 \qquad 1\leq R\leq\Rs.
\end{equation}
\end{proposition}

\begin{proof}
We use the phase portrait equations to write the quantities in~\eqref{eq:ca-background-signs} explicitly in terms of $x$ and $C$. We then use the parameter and profile bounds of Lemmas~\ref{lem:ca-common-parameter} and~\ref{lem:ca-common-profile} to verify their signs by a computer-assisted argument; see Sections~\ref{subsec:ca-reconstruction} and~\ref{subsec:ca-computational-proofs} for details.
\end{proof}

Finally, we can prove Lemma~\ref{lem:three-profile-inequalities}. 

\begin{lemma}
\label{lem:ca-common-stability}
For every $1<\gamma\leq\frac53$, consider the regular $P_6$ Guderley profile in dimension $d=3$. The functions $G_A$, $G_W$, and $\Psi_Z$ defined in~\eqref{eq:profile-gain-equations} and~\eqref{eq:profile-row-quantities} satisfy
\begin{equation*}
 \sup_{1\leq R\leq\Rs}G_A(R)<1-\alpha,\qquad
 \sup_{1\leq R\leq\Rs}G_W(R)<\tfrac{11}{10},\qquad
 \sup_{1\leq R\leq\Rs}\Psi_Z(R)<0.
\end{equation*}
\end{lemma}

\begin{proof}
We prove the three bounds in the order stated. We denote by $\widetilde{\mathcal R}_A$, $\widetilde{\mathcal R}_W$, and $\widetilde{\Psi}_Z$ the expressions obtained from $\mathcal R_A$, $\mathcal R_W$, and $\Psi_Z$ in~\eqref{eq:profile-row-quantities} by replacing $\log R$ with $R^\sharp$.

1. The $A$ inequality. We use the constant function $1-\alpha$ as an upper barrier for $G_A$. By Lemma~\ref{lem:ca-common-radius} and the nonnegativity of $\mathcal J_A$, we have $\mathcal R_A\leq\widetilde{\mathcal R}_A$. It therefore suffices to verify
\begin{subequations}
\begin{align}
 \ell_{AA}(1-\alpha)+\widetilde{\mathcal R}_A&<0,
 \label{eq:ca-A-margin}\\
 b_A(\gamma)&<1-\alpha.
 \label{eq:ca-A-shock-margin}
\end{align}
\end{subequations}
Indeed, subtracting the constant $1-\alpha$ in~\eqref{eq:profile-gain-A} and using $\LA>0$ from~\eqref{eq:out:strat:aw}, we obtain
\begin{equation*}
 \tfrac{\dd}{\dd R}\bigl(G_A(R)-(1-\alpha)\bigr)
 <\tfrac{\ell_{AA}(R)}{R\LA(R)}\bigl(G_A(R)-(1-\alpha)\bigr).
\end{equation*}
Multiplying by the integrating factor $e^{-\int_1^R\frac{\ell_{AA}(s)}{s\LA(s)}\,\dd s}$ and integrating from the shock gives
\begin{equation*}
 G_A(R)-(1-\alpha)
 \leq\bigl(b_A(\gamma)-(1-\alpha)\bigr)
 e^{\int_1^R\frac{\ell_{AA}(s)}{s\LA(s)}\,\dd s}<0,
 \qquad 1\leq R\leq\Rs.
\end{equation*}
The inequality~\eqref{eq:ca-A-margin} is verified by a computer-assisted argument; see Sections~\ref{subsec:ca-reconstruction} and~\ref{subsec:ca-computational-proofs}. The shock inequality~\eqref{eq:ca-A-shock-margin} follows from~\eqref{eq:finite-pressure-RH-feedback-coefficient-A}. Indeed,
\begin{equation*}
 (1-\alpha)-b_A(\gamma)
 =\tfrac{\theta(6-\theta-4\theta^2)}{2(1+\theta)(2-\theta^2)}>0,
 \qquad 0<\theta\leq\tfrac2{\sqrt5}.
\end{equation*}

2. The $W$ inequality. Similarly to the argument for the $A$-inequality, we will use a barrier function for the $G_W$ inequality. It suffices to find a function $H\in C^1([1,\Rs])$ satisfying
\begin{equation*}
 H(1)\geq-b_W(\gamma),\qquad
 H(R)\leq\tfrac{549}{500},\qquad 1\leq R\leq\Rs,
\end{equation*}
and the differential inequality
\begin{equation}
\label{eq:ca-common-W-residual}
 \LW(R)R\tfrac{\dd H}{\dd R}(R)
 -\ell_{WW}(R)H(R)-\widetilde{\mathcal R}_W(R)\geq0.
\end{equation}
To justify this comparison, Lemma~\ref{lem:ca-common-radius} and $\mathcal J_W\geq0$ give $\mathcal R_W\leq\widetilde{\mathcal R}_W$. Subtracting~\eqref{eq:ca-common-W-residual} from~\eqref{eq:profile-gain-W}, we obtain, since $\LW>0$,
\begin{equation*}
 \tfrac{\dd}{\dd R}\bigl(G_W(R)-H(R)\bigr)
 \leq\tfrac{\ell_{WW}(R)}{R\LW(R)}
 \bigl(G_W(R)-H(R)\bigr).
\end{equation*}
Multiplying by the integrating factor $e^{-\int_1^R\frac{\ell_{WW}(s)}{s\LW(s)}\,\dd s}$ and integrating from $R=1$ yields
\begin{equation*}
 G_W(R)-H(R)
 \leq\bigl(-b_W(\gamma)-H(1)\bigr)
 e^{\int_1^R\frac{\ell_{WW}(s)}{s\LW(s)}\,\dd s}\leq0.
\end{equation*}
Thus $G_W(R)\leq H(R)$. The definition of $H$ is specified in Section~\ref{subsec:ca-reconstruction}. The inequality~\eqref{eq:ca-common-W-residual} is verified by a computer-assisted argument; see Sections~\ref{subsec:ca-reconstruction} and~\ref{subsec:ca-computational-proofs}. Consequently,
\begin{equation*}
 \sup_{1\leq R\leq\Rs}G_W(R)\leq\tfrac{549}{500}<\tfrac{11}{10}.
\end{equation*}
3. The $Z$ inequality. Lemma~\ref{lem:ca-common-radius} and $\mathcal J_Z\geq0$ give
\begin{equation*}
 \Psi_Z(R)\leq\widetilde{\Psi}_Z<0,
 \qquad 1\leq R\leq\Rs,
\end{equation*}
where the last inequality is verified by a computer-assisted argument; see Sections~\ref{subsec:ca-reconstruction} and~\ref{subsec:ca-computational-proofs}. 
\end{proof}

\section{Computer-assisted verification}
\label{app:three-profile-computation}

We describe the computer-assisted verification of the inequalities in Appendix~\ref{app:three-profile-interval-proof}. The programs are written in Python and are run in a SageMath environment~\cite{SageMath}. The organization of the code and the structure of this appendix are inspired by the computer-assisted proofs in~\cite{CCSV26,CGP26,CGS25,BCG25}. All source code and data mentioned in this appendix are available at \url{https://github.com/giorgiocialdea/stability_guderley_ca}. 

\subsection{Supplementary data}
\label{subsec:ca-shooting}

The parameter-$z$ table in \path{data/parameters-z} contains $453$ rational nodes $\gamma_0<\cdots<\gamma_{452}$, with $\gamma_0=1$ and $\gamma_{452}=\frac53$, together with associated prescribed values $r_j^-,r_j^+$ at each node.  From these data we define, for $1\leq j\leq452$,
\begin{equation*}
 \theta_j:=\sqrt{2-\tfrac2{\gamma_j}},\qquad
 X_6^\pm(\gamma_j):=\tfrac12+\tfrac{\theta_j\log\theta_j}{6}+\theta_j r_j^\pm.
\end{equation*}
With $\theta_0=0$, we interpolate the prescribed values $r_j^\pm$ linearly in $\theta$. For $\gamma\in[\gamma_j,\gamma_{j+1}]$, $0\leq j\leq451$, we set
\begin{equation}
\label{eq:ca-common-face-interpolation}
 X_6^\pm(\gamma):=\tfrac12+\tfrac{\theta\log\theta}{6}+\theta\left[r_j^\pm+
 \tfrac{\theta-\theta_j}{\theta_{j+1}-\theta_j}(r_{j+1}^\pm-r_j^\pm)\right],
\end{equation}
and define the functions $z_\pm$ by
\begin{equation*}
 z_\pm(\gamma):=\tfrac{\theta X_6^\pm(\gamma)(1-\theta X_6^\pm(\gamma))}{\gamma+(2-\gamma)\theta X_6^\pm(\gamma)}.
\end{equation*}
The prescribed values satisfy $r_j^-<r_j^+$ at every node; interpolation and~\eqref{eq:ca-common-parameter-monotonicity} therefore give $z_-(\gamma)<z_+(\gamma)$ for $\gamma>1$.

\subsection{Objects, methods and routines}
\label{subsec:ca-methods}

We describe the formulas and then the shared interval routine. The phase and barrier formulas are in \path{profiles.py} and \path{barriers.py}; the radius and stability expressions are in \path{fields.py}, \path{radius.py} and \path{inequalities.py}.

{\emergencystretch=1em Common arithmetic is in \path{expressions.py}, and the actual interval evaluator is in \path{interval_checks.py}. The fixed parameters are in \path{parameters.py}.\par}

\subsubsection{Phase equations and parameter bounds}

Substitution of~\eqref{eq:ca-normalization} into~\eqref{eq:guderley:VQ-autonomous} gives the equations
\begin{subequations}
\label{eq:ca-phase-odes}
\begin{align}
 R\partial_RX&=\tfrac{N_x}{\theta D},\label{eq:ca-phase-radius}\\
 R\partial_RC&=\tfrac{C N_c}{XD},\label{eq:ca-phase-F}\\
 \tfrac{\dd X}{\dd C}&=\tfrac{X N_x}{\theta C N_c},
\notag\\
 \tfrac{\dd\log R}{\dd X}&=\theta\tfrac{D}{N_x}.
\notag
\end{align}
\end{subequations}
These identities and~\eqref{eq:ca-background-endpoint-z} are computed symbolically in \path{profiles.py}. Using
\begin{equation*}
 w:=\sqrt{1-2(\gamma+2)z+(\gamma-2)^2z^2},
\end{equation*}
one has
\begin{equation*}
 X_6=C_6=\tfrac{1+(\gamma-2)z-w}{2\theta}.
\end{equation*}
The slope $\alpha\eta_1$ (defined in~\eqref{eq:guderley:sonic-phase-slope}) is the negative root of the quadratic equation
\begin{equation*}
 (\partial_C N_x)\eta^2+(\partial_X N_x-\theta\partial_C N_c)\eta-\theta\partial_X N_c=0\quad\text{at }(X_6,C_6).
\end{equation*}
The positive discriminant in Lemma~\ref{lem:ca-common-barriers}(\ref{item:ca-common-qualitative}) ensures that the square root in~\eqref{eq:guderley:sonic-phase-slope} is real. The higher coefficients are computed using the recurrence~\eqref{eq:guderley:sonic-coefficient-recurrence} with $\ell=2,3$, followed by the rescaling~\eqref{eq:ca-common-sonic-Taylor}.  

For $0\leq\lambda\leq1$, we use 
\begin{equation}
\label{eq:ca-common-parameter-substitution}
 X_6=X_6^-(\gamma)+\lambda\bigl(X_6^+(\gamma)-X_6^-(\gamma)\bigr),\qquad
 z=\tfrac{\theta X_6(1-\theta X_6)}{\gamma+(2-\gamma)\theta X_6}.
\end{equation}
We check~\eqref{eq:ca-common-admissible-z} for these interpolated parameters. The quadratic relation gives
\begin{equation*}
 \bigl(1+(\gamma-2)z-2\theta X_6\bigr)^2
 =1-2(\gamma+2)z+(\gamma-2)^2z^2>0.
\end{equation*}
The expression $1+(\gamma-2)z-2\theta X_6$ is continuous across the interpolation intervals and tends to $1$ as $\gamma\downarrow1$, so it is positive throughout. The identity
\begin{equation}
\label{eq:ca-common-parameter-monotonicity}
 \partial_{X_6}z
 =\theta\tfrac{1+(\gamma-2)z-2\theta X_6}{\gamma+(2-\gamma)\theta X_6}>0
\end{equation}
shows that we can parametrize the interval $z_-<z<z_+$ in terms of $X_6^-<X_6<X_6^+$.

\subsubsection{Comparison functions}
\label{subsec:ca-sonic-cap}

For a polynomial $C(x)$, we set
\begin{equation}
\label{eq:ca-common-sonic-defect}
 \mathfrak I[C]:=-X N_x[X,C]\tfrac{\dd C}{\dd x}-L\theta C N_c[X,C],\qquad X=X_6-Lx.
\end{equation}
We now normalize the expressions of $N_x$ and $D$. By definition, we have that $N_x$ and $D$ vanish at $x=0$, while by construction we have that $\mathfrak I[C^\flat]$ and $\mathfrak I[C^\sharp]$ vanish cubically at $x=0$. Moreover, both the sonic barrier quantities contain a factor of $\theta$, which we simplify before interval evaluation. We hence compute symbolically the normalized expressions
\begin{equation}
\label{eq:ca-common-sonic-deflation}
 -\tfrac{N_x}{x},\qquad -\tfrac{D}{x},\qquad
 -\tfrac{\mathfrak I[C_6^\flat]}{\theta x^3},\qquad
 \tfrac{\mathfrak I[C_6^\sharp]}{\theta x^3}.
\end{equation}
For the shock barrier we write $S=C_0+\theta(a_1h+ch^2)$, with $c=a_2+\varepsilon$, and set $\zeta(h)=S_X-\theta S\mathcal B(X,C)$. Since the Taylor identities give $\zeta(0)=0$, we use the expression
\begin{equation}
\label{eq:ca-common-shock-deflation}
 \tfrac{\zeta}{\theta h}
 =2c-\theta(a_1+ch)\mathcal B(X,C)
 -C_0\left[\tfrac{\mathcal B(X,C_0)-\mathcal B(X_0,C_0)}{h}
 +\tfrac{\mathcal B(X,C)-\mathcal B(X,C_0)}{C-C_0}\tfrac{C-C_0}{h}\right].
\end{equation}
 The power difference is evaluated through
\begin{equation*}
 \psi(s,k):=\int_0^1(1+sr)^{-k-1}\,\dd r
 =\tfrac{1-(1+s)^{-k}}{ks},\qquad s:=\tfrac h{X_0},
\end{equation*}
which gives
\begin{equation}
\label{eq:ca-common-shock-increment}
 \tfrac{C-C_0}{h}=-\tfrac{kC_0}{X_0}\psi(h/X_0,k)+\left(\tfrac{X_0}{X}\right)^k\theta(a_1+ch).
\end{equation}
The file \path{barriers.py} constructs the expressions in~\eqref{eq:ca-common-sonic-deflation} using Taylor identities and computes the expressions in~\eqref{eq:ca-common-shock-deflation} by polynomial division in \texttt{SymPy}. It evaluates $\psi$ using~\eqref{eq:ca-common-normalized-power}.

\subsubsection{Radius comparison}
\label{subsec:ca-radius-formulas}

The radius comparison  function $R^\sharp$ is defined differently according to the $\gamma$ range. The file \path{fields.py} defines $R^\sharp$.  On $1<\gamma\leq\frac{8192}{8191}$, we define
\begin{equation*}
 T(X):=X-\log\left(\tfrac{2X+3}{3}\right),
\end{equation*}
and
\begin{equation*}
 \tfrac{R^\sharp}{\theta} :=T(X)-T(X_0)+\tfrac{23}{100}\theta(X-X_0)
 +\left(\tfrac{31}{200}\theta-\tfrac1{12}\theta\log\theta\right)(X^2-X_0^2).
\end{equation*}
On $\frac{8192}{8191}\leq\gamma<\frac{11}{10}$, we introduce the polynomial
\begin{equation*}
 R_0(X):=\tfrac X3+\tfrac{2X^2}{9}-\tfrac{8X^3}{81}+\tfrac{4X^4}{81},
\end{equation*}
and use
\begin{equation*}
\begin{split}
 R^\sharp:={}&\theta\left[R_0(X)-R_0(X_0)
 +\left(\tfrac35-\tfrac\theta2+\tfrac14(1-\tfrac{5\theta}{2})\right)\theta(X-X_0)\right]
 +\tfrac{\theta L}{200}x(1-x)\\
 &+\begin{cases}
 0,&\tfrac{8192}{8191}\leq\gamma\leq\tfrac{51}{50},\\
 \tfrac{\theta L}{1000}(1-x)^4,&\tfrac{51}{50}\leq\gamma\leq\tfrac{11}{10}.
 \end{cases}
\end{split}
\end{equation*}
Finally, for $\frac{11}{10}\leq\gamma\leq\frac53$, we set
\begin{equation*}
 R^\sharp:=\tfrac{\theta L(1-x)}{34000}\left((12270-3330x)\gamma+6223+433x\right).
\end{equation*}
 We compute the derivative in~\eqref{eq:ca-S-margin}.

For $1<\gamma\leq\frac{8192}{8191}$, we verify~\eqref{eq:ca-S-margin} in the following steps.

1. We seek to prove
\begin{equation*}
 \tfrac{\partial_XR^\sharp}{\theta}-\tfrac{D(X,C(X))}{N_x[X,C(X)]}>0,\qquad X_0\leq X<X_6.
\end{equation*}
Since $N_x[X,C(X)]<0$ by~\eqref{eq:ca-common-field-signs} and Lemma~\ref{lem:ca-common-profile}, multiplication by $-N_x>0$ gives the equivalent inequality
\begin{equation*}
 D(X,C(X))-\tfrac{\partial_XR^\sharp}{\theta}N_x[X,C(X)]>0.
\end{equation*}

2. We first establish the lower bound $C(X)\geq X_6+\frac7{10}\theta(X_6-X)$. On the sonic half, interval arithmetic verifies $\frac{c_1}{\theta L}>\frac7{10}$, $c_2>0$ and $c_3>0$. The quadratic and cubic terms of $C_6^\flat$ in~\eqref{eq:ca-common-sonic-barriers} are therefore nonnegative, so
\begin{equation*}
 C_6^\flat(x)\geq X_6+c_1x\geq X_6+\tfrac7{10}\theta(X_6-X),\qquad 0\leq x\leq\tfrac12.
\end{equation*}
On the shock half, interval arithmetic checks
\begin{equation*}
 \tfrac{C_1^\flat(x)-X_6}{\theta}-\tfrac7{10}(X_6-X)>0,\qquad \tfrac12\leq x\leq1.
\end{equation*}
This shows the desired inequality.

3. The routine \texttt{radius\_monotonicity} uses interval arithmetic to verify 
\begin{equation*}
 \partial_{(C^2)}\left(D-\tfrac{\partial_XR^\sharp}{\theta}N_x\right)
 =-1-\tfrac{\partial_XR^\sharp}{\theta}(3\theta X-3+4\theta q)>0.
\end{equation*}
More precisely, \texttt{radius\_monotonicity} evaluates the three coefficients of $\theta$, $X$, and $-\theta\log\theta\,X$:
\begin{equation*}
\begin{split}
 &\partial_{C^2}\left(D-\tfrac{\partial_XR^\sharp}{\theta}N_x\right)
 =\theta\left(\tfrac{69}{100}-\tfrac{4q}{3}-\tfrac{23\theta q}{25}\right)
 -\theta\log\theta\,X\left(\tfrac{3-\theta(3X+4q)}6\right)\\
 &\qquad+X\left[\tfrac4{2X+3}+\theta\left(\tfrac{93}{100}-\tfrac{3(2X+1)}{2X+3}-\tfrac{16q}{3(2X+3)}\right)-\theta^2\left(\tfrac{69}{100}+\tfrac{31q}{25}+\tfrac{93X}{100}\right)\right].
\end{split}
\end{equation*}
Each of the three coefficients is proved positive by interval arithmetic. Thus $D-\frac{\partial_XR^\sharp}{\theta}N_x$ is monotone increasing in $C^2$; to prove~\eqref{eq:ca-S-margin}, it is enough to check that this expression is positive on the line $C=X_6+\frac7{10}\theta(X_6-X)$.

4. \texttt{build\_line\_quotient\_kernel} uses SymPy to construct the polynomials \texttt{Pa} and \texttt{Pt}, defined in
\begin{equation*}
 \left.\left(D(X,C)-\tfrac{\partial_XR^\sharp}{\theta}N_x[X,C]\right)\right|_{C=X_6+\tfrac7{10}\theta(X_6-X)}=\tfrac{\theta(X_6-X)}{(2X+3)(1+\theta(1-\theta^2)X_6)}
 \left[\mathtt{Pt}+(-\log\theta)X\,\mathtt{Pa}\right].
\end{equation*}
Before constructing these polynomials, we express $X_6-\frac12$ as a linear combination of $\theta\log\theta$ and $\theta$; this allows us to factor out $\theta$ as in the displayed identity.
Since the prefactor is positive for $X<X_6$, we use interval arithmetic to verify
\begin{equation*}
 \mathtt{Pa}\geq0,\qquad \mathtt{Pt}+(-\log\theta)X\,\mathtt{Pa}>0.
\end{equation*}
Observe that when the box includes $\theta=0$, the upper endpoint of $-\log\theta$ is infinite, but $X \mathtt{Pa}\geq0$, so we obtain the desired lower bound.

For $T(X)-T(X_0)$ we use the quotient 
\begin{equation*}
 \tfrac{\log\left(1+\tfrac{2(X-X_0)}{2X_0+3}\right)}{\tfrac{2(X-X_0)}{2X_0+3}}.
\end{equation*}
The routine \texttt{log1p\_ratio} evaluates this quotient.

On the remaining ranges we evaluate~\eqref{eq:ca-S-margin} directly using the profile bounds.

We observe that the three endpoints are all nodes of the $\gamma$ table in \path{data/parameters-z}, so each subdivided interval uses the corresponding formula for $R^\sharp$.

\subsubsection{Stability expressions}
\label{subsec:ca-reconstruction}

We observe the identities
\begin{subequations}
\label{eq:ca-original-derivatives}
\begin{align}
 \Vbar&=\crrb(\theta X-1),& \Qbar&=\tfrac{\crrb\theta C}{\alpha},\label{eq:ca-original-values}\\
 R\partial_R\Vbar&=\crrb\tfrac{N_x}{D},& R\partial_R\Qbar&=\tfrac{\crrb\theta C N_c}{\alpha XD},\label{eq:ca-original-radial-derivatives}\\
 R\partial_R\Kbar&=-\tfrac{1+\Vbar}{\crrb+\Vbar}.\label{eq:ca-original-entropy-derivative}
\end{align}
\end{subequations}
Therefore the differentiated variables satisfy
\begin{subequations}
\label{eq:ca-ring-reconstruction}
\begin{align}
 \overline{\rZ}&=R\partial_R\Vbar-R\partial_R\Qbar+\tfrac{\Qbar}{\gamma}R\partial_R\Kbar,\label{eq:ca-reconstruction-Z}\\
 \overline{\rA}&=\tfrac{\alpha\Qbar}{\gamma}R\partial_R\Kbar,\label{eq:ca-reconstruction-A}\\
 \overline{\rW}&=R\partial_R\Vbar+R\partial_R\Qbar-\tfrac{\Qbar}{\gamma}R\partial_R\Kbar.\label{eq:ca-reconstruction-W}
\end{align}
\end{subequations}
From~\eqref{eq:ca-phase-radius}--\eqref{eq:ca-phase-F} we also observe the identity
\begin{equation}
\label{eq:ca-dotY}
 R\partial_R\Ybar=\tfrac{N_x}{\theta D}\partial_X\Ybar+\tfrac{C N_c}{XD}\partial_C\Ybar.
\end{equation}
These equations are used to rewrite any expression involving differentiated variables uniquely in terms of $(X,C)$.

The background coefficient formulas are
\begin{subequations}
\label{eq:ca-background-sign-expressions}
\begin{align}
 -\ell_{ZA}&=\tfrac{\overline{\rW}-\overline{\rZ}}2+\tfrac{4\crrb\theta C}{\gamma},\label{eq:ca-background-minus-ZA}\\
 \ell_{ZW}&=-\tfrac{\overline{\rA}}2+\tfrac{(\gamma+4)\crrb\theta C}{2\gamma}
 +\tfrac{(3\gamma-5)\crrb(\theta X-1)}4+\tfrac{\gamma-3}4\overline{\rZ},\label{eq:ca-background-ZW}\\
 -\ell_{WZ}&=-\tfrac{\overline{\rA}}2+\tfrac{(\gamma+4)\crrb\theta C}{2\gamma}
 -\tfrac{(3\gamma-5)\crrb(\theta X-1)}4+\tfrac{3-\gamma}4\overline{\rW}.\label{eq:ca-background-minus-WZ}
\end{align}
\end{subequations}
We use interval arithmetic and the enclosures~\eqref{eq:ca-z-enclosure} and~\eqref{eq:ca-profile-enclosure} to prove the signs~\eqref{eq:ca-background-signs}. For $\ell_{ZW}$, to deal with the degeneracy in the limit $\theta \downarrow 0$, the code verifies the sign of $\frac{X\ell_{ZW}}{\theta}$ (the factor $\theta$ is cancelled symbolically on the sonic half $0<x<\frac{1}{2}$; on the shock half $\frac12<x<1$, we instead bound the term $\frac{C-X_6}{\theta}$; see Section~\ref{subsec:ca-endpoint}).

When evaluating expressions involving first derivatives of the profile on the sonic half $0<x<\frac12$, we first cancel the common factor $x$ in ratios such as $\frac{N_x}{D}$ and $\frac{N_c}{D}$; for instance, the code evaluates $\frac{N_x}{D}$ as $\frac{N_x/x}{D/x}$ using~\eqref{eq:ca-common-sonic-deflation}.

Substituting $R^\sharp$ for $\log R$ into~\eqref{eq:profile-row-quantities}, we introduce
\begin{subequations}
\label{eq:ca-weighted-row-quantities}
\begin{align}
 \widetilde{\mathcal R}_A
 &:=\tfrac{21}{10}\ell_{AZ}-\kappa_\gamma\overline{\rA}+R^\sharp\mathcal J_A,
 \label{eq:ca-weighted-row-A}\\
 \widetilde{\mathcal R}_W
 &:=-\ell_{WZ}-(1-\alpha)\ell_{WA}+\kappa_\gamma\overline{\rW}+R^\sharp\mathcal J_W,
\notag\\
 \widetilde{\Psi}_Z
 &:=\ell_{ZZ}-(1-\alpha)\ell_{ZA}+\tfrac{11}{10}\ell_{ZW}
 +\kappa_\gamma\overline{\rZ}+R^\sharp\mathcal J_Z.
\notag
\end{align}
\end{subequations}
In the limit $\theta \downarrow 0$, we write $R^\sharp R\partial_R\overline{\rA}$ and $R^\sharp R\partial_R\overline{\rZ}$ as products of $\frac{R^\sharp}{\theta X}$ with $\theta X R\partial_R\overline{\rA}$ and $\theta X R\partial_R\overline{\rZ}$, respectively. For the $W$ component we use $\frac{R^\sharp}{\theta}$ and $\theta R\partial_R\overline{\rW}$.

For the $W$ comparison, we use the function $H(\theta,x)$
\begin{equation}
\label{eq:ca-common-W-height}
\begin{split}
 H(\theta,x):={}&1-\tfrac{6\theta x}{(1+\theta)(2+\theta)}
 -\tfrac{19}{100}(1-x)+\tfrac{81}{100}(1-x)^2-\tfrac{79}{100}(1-x)^3+\tfrac{67}{250}(1-x)^4\\
 &+\tfrac{53}{200}\theta x\log\left(1+\tfrac{2(1-x)}{\theta}\right).
\end{split}
\end{equation}
This function satisfies the exact bounds
\begin{equation}
\label{eq:ca-common-W-height-bound}
 H(\theta,1)=-b_W(\gamma)=\tfrac{(1-\theta)(2-\theta)}{(1+\theta)(2+\theta)},
 \qquad H(\theta,x)\leq\tfrac{549}{500}<\tfrac{11}{10}.
\end{equation}
To verify the upper bound for $H$ we denote the quartic part by $p(x)$. We have
\begin{equation*}
 \tfrac{p(0)-p(x)}{x}
 =\tfrac{49}{500}+\tfrac{(1-x)(134(1-x)^2-261(1-x)+144)}{500}\geq\tfrac{49}{500},
 \qquad 0<x\leq1,
\end{equation*}
since the quadratic has positive leading coefficient and discriminant $-9063$. We also have $1+b_W(\gamma)\geq\theta$ for $0\leq\theta\leq1$. Differentiating $y-\log(1+y^2)$ proves $\log(1+y^2)\leq y$ for $y\geq0$, and hence
\begin{equation*}
 \tfrac{53}{200}\theta\log\left(1+\tfrac{2(1-x)}{\theta}\right)-(1+b_W(\gamma))
 \leq\tfrac{53}{200}\sqrt{2\theta}-\theta
 \leq\tfrac{2809}{80000}<\tfrac{49}{500}.
\end{equation*}
The identity $H=1+p(x)+x[\frac{53}{200}\theta\log(1+2(1-x)/\theta)-(1+b_W(\gamma))]$ now gives~\eqref{eq:ca-common-W-height-bound}.

Then, to prove the profile inequalities~\eqref{eq:three-profile-inequalities}, we follow the proof of Lemma~\ref{lem:ca-common-stability} and verify the following three inequalities by interval arithmetic, writing the $W$ comparison in the coordinate $x$ using~\eqref{eq:ca-phase-radius}:
\begin{equation*}
\begin{aligned}
 -\tfrac{(\overline{\rZ}+\overline{\rW})(X+C)}{2L}\,\tfrac{\dd H}{\dd x}
 -\ell_{WW}H-\widetilde{\mathcal R}_W&\geq0,
\\
 \ell_{AA}(1-\alpha)+\widetilde{\mathcal R}_A&<0,
\\
 \widetilde{\Psi}_Z&<0.
\end{aligned}
\end{equation*}
Second derivatives of the profile enter $\mathcal J_i$ through $R\partial_R\Ybar$ in~\eqref{eq:profile-forcing-vectors}. On the sonic half $0<x<\frac{1}{2}$, fix $\gamma,z,\omega$ and let $C(x)$ denote the cubic polynomial in~\eqref{eq:ca-omega-representation}. The chain rule and~\eqref{eq:ca-common-sonic-defect} give
\begin{equation}
\label{eq:ca-sonic-derivative-correction}
 R\partial_R\Ybar
 ={}-\tfrac{N_x/x}{L\theta(D/x)}\tfrac{\dd}{\dd x}\Ybar(X_6-Lx,C(x))
 -\tfrac{x}{LX(D/x)}\tfrac{\mathfrak I[C]}{\theta x^3}\bigl(x\partial_C\Ybar\bigr).
\end{equation}
Here the left-hand side is evaluated using~\eqref{eq:ca-dotY}, and all phase quantities are evaluated at $(X,C)=(X_6-Lx,C(x))$. The code evaluates algebraically simplified forms of this identity, using the bounded expression~\eqref{eq:ca-common-sonic-deflation}.
For the second term in~\eqref{eq:ca-sonic-derivative-correction}, we use the identity
\begin{equation*}
 x\partial_C\left(\tfrac{N_x}{D}\right)
 =\tfrac{\partial_C N_x}{D/x}
 -\tfrac{(N_x/x)\partial_C D}{(D/x)^2}.
\end{equation*}
The right-hand side uses the simplified quantities $N_x/x$ and $D/x$ in~\eqref{eq:ca-common-sonic-deflation}, where $D/x$ is bounded away from zero. We treat $N_c/D$ in the same way.

\subsubsection{Interval arithmetic and subdivision}

In every computation we use Sage intervals.

For a given expression, to prove positivity on a given box, \texttt{verify\_boxes} evaluates it using interval arithmetic. If the resulting lower bound is not positive, it tries to improve the bound using derivatives. It evaluates the function at the midpoint of the box and bounds each partial derivative over the whole box, obtaining the enclosure:
\begin{equation*}
 f(B)\subset f(m)+\sum_j|\partial_jf(B)|(B_j-m_j),
\end{equation*}
where $m$ is the exact midpoint of $B$ and the derivatives are bounded on the whole box. For absolute values, we use the analogous enclosure
\begin{equation*}
 |f(B)|\subset |f(m)|+\sum_j|\partial_jf(B)|(B_j-m_j).
\end{equation*}
It then intersects the interval arithmetic and derivative bound to obtain an improved bound. If positivity remains unresolved, it subdivides the box. Derivative size suggests in which direction to split the box (though this gets periodically renormalized to avoid missing splitting along one direction).

The exact symbolic expressions required by each lemma are constructed at the beginning and then stored in named variables. We use SymPy to differentiate and simplify these expressions as described above. The routine \texttt{compile\_expression} in \path{expressions.py} then uses SymPy's \texttt{lambdify}: a printer (\texttt{ExactPrinter}) converts the symbolic expression to text, from which \texttt{lambdify} constructs a function that evaluates constants and functions using Sage interval arithmetic~\cite{SymPyLambdify}. We then evaluate it on successive interval boxes, without repeating the symbolic computation. This avoids manually transcribing the derived formulas each time.

\subsubsection{The first parameter box}
\label{subsec:ca-endpoint}
The gamma interval containing $\gamma \downarrow 1$ (equivalently, $\theta \downarrow 0$) requires special care, because of the degeneracy of some expressions in the limit. We now describe the routines used here to deal with it. This box corresponds to
\begin{equation*}
 0\leq\theta\leq\theta_1=\tfrac1{8192},\qquad
 1<\gamma\leq\tfrac{134217728}{134217727},
\end{equation*}
 The interpolated bounds~\eqref{eq:ca-common-face-interpolation} give that $X_6$ lives in the strip
\begin{equation}
\label{eq:ca-common-endpoint-faces}
 X_6=\tfrac12+\tfrac{\theta\log\theta}{6}+\theta r,
 \qquad -\tfrac7{20}\leq r\leq-\tfrac{149}{500}.
\end{equation}
The three functions evaluated in \path{expressions.py} are
\begin{equation*}
\begin{aligned}
 \phi(\theta,a)&:=\theta\log(1+a/\theta),&\phi(0,a)&:=0,
\\
 E(u)&:=\int_0^1e^{-u\sigma}\,\dd\sigma=\tfrac{1-e^{-u}}u,&E(0)&:=1,
\\
 J(v)&:=-\tfrac{v\log v}{1-v},&J(0)&:=0,\quad J(1):=1.
\end{aligned}
\end{equation*}
For $\theta,a,u\geq0$ and $0\leq v\leq1$, the function $\phi$ increases in each argument, $E$ decreases, and $J$ increases. Their interval bounds are therefore obtained by evaluating the appropriate endpoints.

For $0\leq u\leq1$, the bounds for $E$ and its derivatives follow from
\begin{equation*}
 \left|e^{-u\sigma}-\sum_{j=0}^{16}\tfrac{(-u\sigma)^j}{j!}\right|
 \leq\tfrac{(u\sigma)^{17}}{17!},\qquad 0\leq\sigma\leq1.
\end{equation*}
Multiplying by $\sigma^n$ and integrating over $[0,1]$ gives
\begin{equation}
\label{eq:ca-E-Taylor-enclosure}
 (-1)^nE^{(n)}(u)\in\sum_{j=0}^{16}\tfrac{(-u)^j}{j!(n+j+1)}
 +\left[-\tfrac{u^{17}}{17!(n+18)},\tfrac{u^{17}}{17!(n+18)}\right],
 \qquad n=0,1,2.
\end{equation}
For $u>1$, the same exponential moments are bounded by integration by parts.

1. \emph{Shock barriers.} To evaluate $(X_0/X)^k$ and $\psi(h/X_0,k)$ in~\eqref{eq:ca-common-shock-increment} as $X_0=\theta/(4-\theta^2)$ tends to zero, we use
\begin{equation}
\label{eq:ca-common-normalized-power}
 v:=\tfrac{X_0}{X},\qquad
 u:=\tfrac k\theta\phi\bigl(\theta,(4-\theta^2)h\bigr),\qquad
 \left(\tfrac{X_0}{X}\right)^k=e^{-u},\qquad \psi(h/X_0,k)=J(v)E(u).
\end{equation}
These identities hold for $\theta>0$. We bound $v=(\theta/X)/(4-\theta^2)$ as described below and evaluate $\phi$, $E$ and $J$ using the routines above. 

2. \emph{The ratio $(C-X_6)/\theta$.} To verify~\eqref{eq:ca-common-midpoint-signs} and the lower bound used for~\eqref{eq:ca-S-margin}, we evaluate $(C-X_6)/\theta$ for $C=C_1^\flat$ or $C=C_1^\sharp$, defined in~\eqref{eq:ca-common-shock-lower} and~\eqref{eq:ca-common-shock-upper}, respectively. Substituting the definitions, together with~\eqref{eq:ca-shock-point}, \eqref{eq:ca-common-normalized-power} and~\eqref{eq:ca-common-endpoint-faces}, we obtain
\begin{equation*}
\begin{split}
 \tfrac{C-X_6}{\theta}={}&\tfrac{\theta}{2(4-\theta^2)}-r+\left(\tfrac{X_0}{X}\right)^k(a_1h+ch^2)
 -\tfrac{kC_0}{\theta}E(u)\log\bigl((4-\theta^2)X\bigr)\\
 &+\left(\tfrac{kC_0}{\theta}E(u)-\tfrac16\right)\log\theta,
 \qquad c=a_2+\varepsilon,
\end{split}
\end{equation*}
with $\varepsilon=\varepsilon_\flat$ or $\varepsilon=\varepsilon_\sharp$, respectively, as defined in~\eqref{eq:ca-common-shock-errors}. We split the last term into $(kC_0/\theta-\frac16)\log\theta$ and $(kC_0/\theta)(E(u)-1)\log\theta$. For fixed $X_6$, the expression $kC_0/\theta$ converges to $X_6/3$ as $\theta \to 0$. We therefore cancel $\theta$ symbolically in $(kC_0/\theta-X_6/3)/\theta$ and use~\eqref{eq:ca-common-endpoint-faces} to write
\begin{equation*}
 \left(\tfrac{kC_0}{\theta}-\tfrac16\right)\log\theta
 =\tfrac{\tfrac{kC_0}{\theta}-\tfrac{X_6}{3}}{\theta}\,\theta\log\theta
 +\tfrac r3\theta\log\theta+\tfrac1{18}\theta\log^2\theta.
\end{equation*}
For the second term, we use
\begin{equation*}
 \tfrac{kC_0}{\theta}(E(u)-1)\log\theta
 =C_0\left(\tfrac k\theta\right)^2\tfrac{E(u)-1}{u}
 \left[\theta\log\theta\,\log\bigl((4-\theta^2)X\bigr)-\theta\log^2\theta\right].
\end{equation*}
We evaluate $C_0$ and $k/\theta=\frac{2q}{3-4\theta q}$ by interval arithmetic using~\eqref{eq:ca-shock-point} and~\eqref{eq:ca-shock-factor}. The quotient satisfies
\begin{equation*}
 -\tfrac12\leq\tfrac{E(u)-1}{u}
 =-\int_0^1(1-\sigma)e^{-u\sigma}\,\dd\sigma\leq0,\qquad u\geq0.
\end{equation*}
We bound this integral by the same Taylor argument as in~\eqref{eq:ca-E-Taylor-enclosure}, with continuous value $-\frac12$ at $u=0$.

For the expression in square brackets, consider a box $B_1\times B_2$ for $(\theta,X)$. If $\inf B_2>0$, we evaluate the expression directly by interval arithmetic (with the care for $\theta \log \theta$ explained below). If $0=\inf B_2<\sup B_2$, the logarithm has no finite lower bound on the box. We instead use $X\geq X_0$ and $\theta\log\theta\leq0$ to obtain
{\small
\begin{equation*}
 \theta\log\theta\,\max\left\{0,\log\bigl((4-(\inf B_1)^2)\sup B_2\bigr)\right\}-\theta\log^2\theta \leq\theta\log\theta\,\log\bigl((4-\theta^2)X\bigr)-\theta\log^2\theta\leq0.
\end{equation*}
}
The left side gives a lower bound: we use the minimum of $\theta\log\theta$ and the maximum of $\theta\log^2\theta$ on $B_1$. To bound these functions we will use their convexity properties. For $\theta \log \theta$, the only positive critical point is $\theta=e^{-1}$. On a closed bounded interval $B\subset[0,\infty)$, we get the enclosure
\begin{equation*}
\begin{aligned}
 \min_{\theta\in B}\theta\log\theta
 &=\begin{cases}
 -e^{-1},&e^{-1}\in B,\\
 \min\{(\inf B)\log(\inf B),(\sup B)\log(\sup B)\},&e^{-1}\notin B,
 \end{cases}
\\
 \max_{\theta\in B}\theta\log\theta
 &=\max\{(\inf B)\log(\inf B),(\sup B)\log(\sup B)\}.
\end{aligned}
\end{equation*}
For $\theta\log^2\theta$, we evaluate the endpoints and any critical points $e^{-2}$ and $1$ lying in the interval.

Sometimes the checked expression contains $X(C-X_6)/\theta$; this can be evaluated directly since the troublesome term $\log ( (4- \theta^2)X)$ becomes  $X \log ( (4- \theta^2)X) = X \log X + X \log (4-\theta^2)$. 

3. \emph{Joint parameter and spatial limits.} The ratio $\theta/X$ is bounded in the limit $\theta \downarrow 0$; but it does not converge: it tends to $4$ along the shock $X=X_0$ and to $0$ along $X=\sqrt\theta$. It is convenient to keep it as a separate coordinate, with $0\leq\theta/X\leq4$, and also use $\theta/X\leq\theta_1/X$ whenever $X>0$. These bounds on the first box are used throughout the whole computer-assisted proof on the shock half (that is, $\frac12\leq x\leq1$). 

For the $W$ inequality, the derivative of the logarithmic term in $H$ similarly depends on the ratio between $\theta$ and $1-x$. To deal with such a term, we split the domain into three regions:
\begin{equation*}
\begin{aligned}
 1-\theta\leq x\leq1:\quad &\xi:=\tfrac{1-x}{\theta}\in[0,1],\qquad
 X=\theta\left(\tfrac1{4-\theta^2}+L\xi\right),\qquad
 \tfrac\theta X=\tfrac1{\tfrac1{4-\theta^2}+L\xi},
\\
 \tfrac12\leq x\leq1-\theta:\quad &\nu:=\tfrac\theta{1-x}\in[0,1],\qquad
 X=(1-x)\left(\tfrac\nu{4-\theta^2}+L\right),\qquad
 \tfrac\theta X=\tfrac\nu{\tfrac\nu{4-\theta^2}+L},
\\
 0\leq x\leq\tfrac12:\quad &X=X_6-Lx,
\end{aligned}
\end{equation*}
 The first two variables are used in the computations involving the shock barriers and~\eqref{eq:ca-common-normalized-power}; the third is used in the computations involving the sonic barriers. 

\subsection{Details of the proofs}
\label{subsec:ca-computational-proofs}

The file \path{lemmas.py} contains the functions that run the computations needed for each Lemma (or part of it).

Every computation is carried out in the variables $\gamma$ (except on the first box, where we use the variable $\theta$), the z-interpolation parameter $\lambda$ in~\eqref{eq:ca-common-parameter-substitution}, the spatial coordinate $x$, and the parameter $\omega\in[0,1]$. This last parameter is used to substitute on each half, 
\begin{equation}
\label{eq:ca-omega-representation}
 C=C^\flat+\omega(C^\sharp-C^\flat),\qquad 0\leq\omega\leq1,
\end{equation}
using $C_6^\flat,C_6^\sharp$ on the sonic half and $C_1^\flat,C_1^\sharp$ on the shock half.

We now briefly sum up the computer-assisted checks performed for every lemma and record in the tables in Section~\ref{subsec:ca-interval-tables} lower bounds and (when available) upper bounds for all the quantities whose signs we checked in the Proofs of Lemmas~\ref{lem:ca-common-barriers}--\ref{lem:ca-common-stability}.

\begin{proof}[Computer-assisted checks for Lemma~\ref{lem:ca-common-barriers}]
1. Part~(\ref{item:ca-common-admissibility}). We substitute~\eqref{eq:ca-common-parameter-substitution} into~\eqref{eq:ca-common-admissible-z} and verify $z_g(\gamma)<z<z_M(\gamma)$ by interval arithmetic. For the lower inequality, the difference $z-z_g(\gamma)$ is normalized by dividing by $\theta$.

2. Part~(\ref{item:ca-common-qualitative}). We evaluate $C_6^\flat,C_6^\sharp$ from~\eqref{eq:ca-common-sonic-barriers} on $0\leq x\leq\frac12$ and $C_1^\flat,C_1^\sharp$ from~\eqref{eq:ca-common-shock-barriers} on $\frac12\leq x\leq1$, and prove that their interval lower bounds are positive. We also verify that $c_3/\theta$ does not vanish, with $c_3$ computed using~\eqref{eq:guderley:sonic-coefficient-recurrence} and~\eqref{eq:ca-common-sonic-Taylor}. 

3. Part~(\ref{item:ca-common-field-signs}). For $C=C_6^\flat$ and $C=C_6^\sharp$, we verify positivity of $-N_x[X,C]/x$ and $-D(X,C)/x$, using~\eqref{eq:ca-common-sonic-deflation} on $0\leq x\leq\frac12$. For $C=C_1^\flat$ and $C=C_1^\sharp$, we evaluate~\eqref{eq:ca-scaled-phase-fields} and verify $N_x[X,C]<0$ and $D(X,C)<0$ on $\frac12\leq x\leq1$. 

4. Part~(\ref{item:ca-common-sonic-signs}). To verify~\eqref{eq:ca-common-sonic-signs}, we evaluate the last two expressions in~\eqref{eq:ca-common-sonic-deflation}, with $C_6^\flat,C_6^\sharp$ defined in~\eqref{eq:ca-common-sonic-barriers}, and find positive interval lower bounds on $0\leq x\leq\frac12$.

5. Part~(\ref{item:ca-common-shock-signs}). To verify~\eqref{eq:ca-common-shock-signs}, we evaluate $\zeta/(\theta h)$ using~\eqref{eq:ca-common-shock-deflation} and~\eqref{eq:ca-common-shock-increment}, with $\varepsilon=\varepsilon_\flat$ and $\varepsilon=\varepsilon_\sharp$ from~\eqref{eq:ca-common-shock-errors}. We find a negative upper bound in the first case and a positive lower bound in the second, on $\frac12\leq x\leq1$.
\end{proof}

\begin{proof}[Computer-assisted checks for Lemma~\ref{lem:ca-common-parameter}]
The routine \texttt{A2} verifies the inequality~\eqref{eq:ca-common-midpoint-signs} at $z=z_-(\gamma)$ and $z=z_+(\gamma)$. Both inequalities are evaluated after dividing by $\theta$. 
\end{proof}

\begin{proof}[Computer-assisted checks for Lemma~\ref{lem:ca-common-radius}]
The routine \texttt{A5} uses the radius formulas in Section~\ref{subsec:ca-radius-formulas}, and we prove inequality~\eqref{eq:ca-S-margin}. We obtain the recorded bounds by dividing by $-N_x$; on the sonic half, we cancel the factor $X_6-X=Lx$ using~\eqref{eq:ca-common-sonic-deflation}. On the first box, we normalize by dividing by $\theta$. 
\end{proof}

\begin{proof}[Computer-assisted checks for Proposition~\ref{prop:ca-background-signs}]
The selection \texttt{A6} uses~\eqref{eq:ca-original-derivatives}--\eqref{eq:ca-background-sign-expressions} to write the quantities in~\eqref{eq:ca-background-signs} in terms of $x$ and $C$. It verifies the five background signs on the parameter and profile strips.
\end{proof}

\begin{proof}[Computer-assisted checks for Lemma~\ref{lem:ca-common-stability}]
1. For the $A$ inequality, \texttt{A7-A} checks positivity of the negative of~\eqref{eq:ca-A-margin}, using $\widetilde{\mathcal R}_A$ from~\eqref{eq:ca-weighted-row-A}. The shock inequality~\eqref{eq:ca-A-shock-margin} follows from~\eqref{eq:finite-pressure-RH-feedback-coefficient-A}.

2. For the $W$ inequality, \texttt{A7-W} checks the inequality~\eqref{eq:ca-common-W-residual} for the function~\eqref{eq:ca-common-W-height}.

3. For the $Z$ inequality, \texttt{A7-Z} checks positivity of $-\widetilde{\Psi}_Z$.
\end{proof}

The complete computation is run from the directory containing \path{script.py} with
\begin{verbatim}
sage -python -B script.py --lemma all --workers 4
\end{verbatim}
The options \texttt{A1-i}--\texttt{A1-v} follow the five parts of Lemma~\ref{lem:ca-common-barriers}. The other ones are \texttt{A2}, \texttt{A5}, \texttt{A6}, \texttt{A7-A}, \texttt{A7-W} and \texttt{A7-Z}. Each option, for example \texttt{--lemma A7-W}, checks only the signs of the corresponding Lemma. Selecting \texttt{all} runs each Lemma in order.

\subsection{Interval bounds}
\label{subsec:ca-interval-tables}

{\small
\begin{longtable}{@{}p{0.60\textwidth}rr@{}}
\caption{Bounds for the expressions checked in Lemma~\ref{lem:ca-common-barriers}.}\label{tab:ca-parameter-bounds}\\
\toprule Quantity & Lower bound & Upper bound \\
\midrule\endfirsthead
\toprule Quantity & Lower bound & Upper bound \\
\midrule\endhead
$z-z_g(\gamma)$ & $0.017\theta$ & $0.5\theta$ \\*
$z_M(\gamma)-z$ & $3.2\times10^{-4}$ & $0.18$ \\*
\midrule
$C_6^\flat$ & $0.41$ & $0.56$ \\*
$C_6^\sharp$ & $0.41$ & $0.57$ \\*
$C_1^\flat$ & $0.44$ & $0.63$ \\*
$C_1^\sharp$ & $0.45$ & $0.63$ \\*
$\frac{|c_3|}{\theta}$ & $5.8\times10^{-4}$ & $0.055$ \\*
\midrule
$-\frac{N_x[X,C_6^\flat]}x$ & $0.36$ & $1.1$ \\*
$-\frac{D(X,C_6^\flat)}x$ & $0.27$ & $0.51$ \\*
$-\frac{N_x[X,C_6^\sharp]}x$ & $0.36$ & $1.1$ \\*
$-\frac{D(X,C_6^\sharp)}x$ & $0.27$ & $0.51$ \\*
$-N_x[X,C_1^\flat]$ & $0.17$ & $0.76$ \\*
$-D(X,C_1^\flat)$ & $0.13$ & $0.32$ \\*
$-N_x[X,C_1^\sharp]$ & $0.17$ & $0.76$ \\*
$-D(X,C_1^\sharp)$ & $0.14$ & $0.32$ \\*
\midrule
$-\frac{\mathfrak I[C_6^\flat]}{\theta x^3}$ & $1.4\times10^{-9}$ & $0.042$ \\*
$\frac{\mathfrak I[C_6^\sharp]}{\theta x^3}$ & $6.4\times10^{-6}$ & $0.21$ \\*
$-\frac{\zeta}{\theta h},\quad \varepsilon=\varepsilon_\flat$ & $8.3\times10^{-9}$ & $0.72$ \\*
$\frac{\zeta}{\theta h},\quad \varepsilon=\varepsilon_\sharp$ & $1.8\times10^{-7}$ & $3.7$ \\*
\bottomrule
\end{longtable}}

{\small
\begin{longtable}{@{}p{0.60\textwidth}rr@{}}
\caption{Midpoint inequality~\eqref{eq:ca-common-midpoint-signs} for Lemma~\ref{lem:ca-common-parameter}.}\label{tab:ca-comparison-bounds}\\
\toprule Quantity & Lower bound & Upper bound \\
\midrule\endfirsthead
\toprule Quantity & Lower bound & Upper bound \\
\midrule\endhead
$C_1^\flat(\frac12;z_-)-C_6^\sharp(\frac12;z_-)$ & $7.9\times10^{-7}\theta$ & $1.1\times10^{-3}\theta$ \\*
$C_6^\flat(\frac12;z_+)-C_1^\sharp(\frac12;z_+)$ & $1.1\times10^{-7}\theta$ & $2.7\times10^{-3}\theta$ \\*
\bottomrule
\end{longtable}}

{\small
\begin{longtable}{@{}p{0.60\textwidth}r@{}}
\caption{Lower bounds for $\frac{\partial_XR^\sharp}{\theta}-\frac{D}{N_x}$ in~\eqref{eq:ca-S-margin}, on $X_0\leq X<X_6$.}\label{tab:ca-radius-direct-bounds}\\
\toprule Parameter range & Lower bound \\
\midrule\endfirsthead
\toprule Parameter range & Lower bound \\
\midrule\endhead
$1<\gamma\leq\frac{134217728}{134217727}$ (first parameter box) & $0.0036\theta$ \\*
$\frac{134217728}{134217727}\leq\gamma\leq\frac53$ & $4.4\times10^{-7}$ \\*
\bottomrule
\end{longtable}}

{\small
\begin{longtable}{@{}p{0.60\textwidth}rr@{}}
\caption{Bounds for the five expressions in Proposition~\ref{prop:ca-background-signs}.}\label{tab:ca-stability-bounds}\\
\toprule Quantity & Lower bound & Upper bound \\
\midrule\endfirsthead
\toprule Quantity & Lower bound & Upper bound \\
\midrule\endhead
$\overline{\rZ}$ & $0.72$ & $2.3$ \\*
$\overline{\rW}$ & $0.15$ & $5.2$ \\*
$-\ell_{ZA}$ & $0.029$ & $2.2$ \\*
$-\ell_{WZ}$ & $0.66$ & $3.2$ \\*
$\ell_{ZW}$ & $9.4\times10^{-5}\frac\theta X$ & $0.51\frac\theta X$ \\*
\bottomrule
\end{longtable}}

{\small
\begin{longtable}{@{}p{0.60\textwidth}rr@{}}
\caption{Bounds for the three stability expressions in Lemma~\ref{lem:ca-common-stability}.}\label{tab:ca-stability-residual-bounds}\\
\toprule Quantity & Lower bound & Upper bound \\
\midrule\endfirsthead
\toprule Quantity & Lower bound & Upper bound \\
\midrule\endhead
$-\ell_{AA}(1-\alpha)-\widetilde{\mathcal R}_A$ & $3.3\times10^{-4}$ & $3.3$ \\*
$\LW R\partial_RH-\ell_{WW}H-\widetilde{\mathcal R}_W$ & $9.7\times10^{-6}$ & $9.6$ \\*
$-\widetilde\Psi_Z$ & $7.9\times10^{-5}$ & $2.2$ \\*
\bottomrule
\end{longtable}}

\bibliographystyle{plain}
\bibliography{guderley_stability}

\end{document}